\documentclass[aap,citesort,MSNbibl,dvips]{arximspdf}
\usepackage{stfloats}
\usepackage{dcolumn}
\usepackage{multirow}
\usepackage{graphicx}

\doi{10.1214/10-AAP729}
\volume{21}
\issue{5}
\pubyear{2011}
\firstpage{1694}
\lastpage{1748}

\makeatletter

\def\bptnote#1{}

\newcolumntype{d}[1]{D{.}{.}{#1}}
\newcolumntype{k}[1]{D{,}{}{#1}}

\newcommand{\midd}{\|}

\newtheorem{theorem}{Theorem}[section]
\newtheorem{lemma}[theorem]{Lemma}
\newproclaim{definition}[theorem]{Definition}
\newproclaim{remark}[theorem]{Remark}

\newcommand{\D}{\mathsf{D}}
\newcommand{\const}{D}
\newcommand{\ER}{\mathsf{R}}
\newcommand{\Cov}{\operatorname{Cov}}
\newcommand{\Fs}{\mathsf{F}}
\newcommand{\de}{{d}}
\newcommand{\cX}{{\mathcal X}}
\newcommand{\cF}{{\mathcal F}}
\newcommand{\J}{{\mathcal J}}
\newcommand{\K}{\mathsf{K}}
\newcommand{\hK}{\widehat{\mathsf K}}
\newcommand{\tK}{\widetilde{\mathsf K}}

\newcommand{\Ball}{\mathsf{B}}
\newcommand{\Tree}{\mathsf{T}}
\newcommand{\rroot}{{\o}}
\newcommand{\droot}{{\partial\o}}

\newcommand{\hEv}{\widehat{\mathsf E}}
\newcommand{\hEvA}{\widehat{\mathsf A}}
\newcommand{\hEvC}{\widehat{\mathsf C}}
\newcommand{\Ev}{\mathsf{E}}
\newcommand{\EvA}{\mathsf{A}}
\newcommand{\EvC}{\mathsf{C}}
\newcommand{\sign}{\operatorname{sign}}
\newcommand{\G}{\mathcal{G}}
\newcommand{\HH}{\mathcal{H}}
\newcommand{\V}{\mathcal{V}}
\newcommand{\Ed}{\mathcal{E}}
\newcommand{\I}{\mathcal{I}}
\newcommand{\SSS}{\mathcal{S}}
\newcommand{\Sh}{\widehat{\mathcal{S}}}
\newcommand{\W}{\mathcal{W}}
\newcommand{\us}{\underline{\sigma}\hspace*{0.2pt}}
\newcommand{\A}{\mathcal{A}}
\newcommand{\up}{\underline{+1}}
\newcommand{\prob}{{\mathbb P}}
\newcommand{\qprob}{{\mathbb Q}}
\newcommand{\Ex}{{\mathbb E}}
\newcommand{\E}{{\mathbb E}}
\newcommand{\reals}{{\mathbb R}}
\newcommand{\naturals}{{\mathbb N}}
\newcommand{\R}{\mathcal{R}}
\newcommand{\Z}{\mathbb{Z}}
\newcommand{\ind}{{\mathbb I}}
\newcommand{\Psih}{\widehat{\Psi}}
\newcommand{\di}{\partial i}
\newcommand{\ve}{\varepsilon}
\newcommand{\C}{\mathcal{C}}
\newcommand{\T}{\mathcal{T}}
\newcommand{\traj}{\lambda}
\newcommand{\bias}{\beta}
\newcommand{\coinset}{\bolds\omega}

\makeatother

\begin{document}
\begin{frontmatter}

\title{Majority dynamics on trees and the dynamic cavity
method\thanksref{T1}}
\runtitle{Majority dynamics on trees}

\thankstext{T1}{Supported in part by a Terman fellowship, the NSF
CAREER Award CCF-0743978
and the NSF Grant DMS-08-06211.}

\begin{aug}
\author[A]{\fnms{Yashodhan} \snm{Kanoria}\thanksref{t2}\ead[label=e1]{ykanoria@stanford.edu}} and
\author[B]{\fnms{Andrea} \snm{Montanari}\corref{}\ead[label=e2]{montanari@stanford.edu}}
\runauthor{Y. Kanoria and A. Montanari}
\affiliation{Stanford University}
\address[A]{Department of Electrical Engineering\\
Stanford University\\
350 Serra Mall\\
Stanford, California 94305-9505\\
USA\\
\printead{e1}} %adresu isvedimo komanda gale!
\address[B]{Department of Electrical Engineering\\
\quad and Department of Statistics\\
Stanford University\\
350 Serra Mall\\
Stanford, California 94305-9505\\
USA\\
\printead{e2}}
\end{aug}

\thankstext{t2}{Supported by a 3Com Corporation Stanford Graduate
Fellowship.}

% HISTORY:
\received{\smonth{5} \syear{2010}}

% ABSTRACT
%
\begin{abstract}
A voter sits on each vertex of an infinite tree of degree
$k$, and has to decide between two alternative opinions. At each time step,
each voter switches to the opinion of the majority of her neighbors.
We analyze this majority process when opinions are initialized to
independent and identically distributed random variables.

In particular, we bound the threshold value of the initial
bias such that the process converges to consensus.
In order to prove an upper bound,
we characterize the process of a
single node in the large $k$-limit. This approach is inspired by
the theory of mean field spin-glass and can potentially
be generalized to a wider class of models. We also derive a lower
bound that is nontrivial for small, odd values of $k$.
\end{abstract}

% KEYWORDS
%

\end{frontmatter}

%s1 ###
\section{Introduction}\label{intro}

%s1.1 ###
\subsection{The majority process}

Consider a graph $\G$ with vertex set $\V$,
and edge set~$\Ed$.
In the following, we shall denote by $\di$ the set of neighbors of
$i\in\V$,
and assume $|\di|<\infty$ (i.e., $\G$ is locally finite).
To each vertex $i \in\V$, we assign
an initial spin $\sigma_i(0) \in\{-1,+1\}$.
The vector of all initial spins is denoted by $\us(0)$.
Configuration $\us(t) = \{\sigma_i(t)\dvtx i\in\V\}$ at subsequent times
$t= 1, 2,\ldots$ are determined according to the following
majority update rule.
If $\partial i$ is the set of
neighbors of node $i \in\mathcal{V}$, we let
%
%e1 ###
%
\begin{equation}\label{eq:MajorityRule}
\sigma_i(t+1) = \sign\biggl( \sum_{j \in\partial i} \sigma_j(t)\biggr),
\end{equation}
when $ \sum_{j \in\partial i} \sigma_j(t)\neq0$. If
$\sum_{j \in\partial i} \sigma_j(t) = 0$, then we let
%
%e2 ###
%
\begin{equation}\label{eq:RandomUpdate}
\sigma_i(t+1) = \cases{
\sigma_i(t), &\quad with probability $1/2$,\cr
-\sigma_i(t), &\quad with probability $1/2$.}
\end{equation}
In order to construct this process, we associate to each vertex $i\in
\V$,
a sequence of i.i.d. Bernoulli$(1/2)$ random variables
$\omega_i = \{\omega_{i,0},\omega_{i,1},\omega_{i,2},\ldots\}$,
whereby $\omega_{i,t}$ is used
to break the (eventual) tie at time $t$.
A realization of the process is then determined by the triple $(\G
,\coinset,\us(0))$,
with $\coinset=\{\omega_i\}$.

In this work, we will study the asymptotic dynamic of this process
when $\G$ is an infinite regular tree of degree $k\ge2$.
Let $\prob_{\theta}$ be the law of the majority process
where, in the initial configuration, the spins $\sigma_i(0)$
are i.i.d. with $\prob_{\theta}\{\sigma_i(0)=+1\}=(1+\theta)/2$.
We define the \textit{consensus threshold} as the smallest bias in
the initial condition such that the dynamics converges
to the all $+1$ configuration
%
%e3 ###
%
\begin{equation}\label{eq:ThresholdDef}
\theta_{*}(k) = \inf
\Bigl\{\theta\dvtx\prob_{\theta}\Bigl( \lim_{t \rightarrow\infty}
\us(t) = \underline{+1} \Bigr) = 1 \Bigr\}.
\end{equation}
Here convergence to the all-$(+1)$ configuration is understood
to be point-wise. We shall call $\theta_*(k)$ the \textit{consensus threshold}
of the $k$-regular tree.

Two simple observations will be useful in stating our results.

\textit{Monotonicity}.
Denote by $\succeq$ the natural partial ordering between configurations
(i.e., $\us\succeq\us'$ if and only if $\sigma_i\ge\sigma'_i$
for all $i\in\V$).
Then the majority dynamics preserves this
partial ordering. More precisely, given two copies of the
process with initial conditions $\us(0)\succeq\us'(0)$, there exists a
coupling between the two processes
such that $\us(t)\succeq\us'(t)$ for all $t\ge0$.

\textit{Symmetry}. Let $-\us$ denote the configuration obtained by inverting
all the spin values in $\us$. Then two copies of the process with initial
conditions $\us'(0) = -\us(0)$ can be coupled in such a way
that $\us'(t) = -\us(t)$ for all $t\ge0$.

It immediately follows from these properties that
\[
0 \le\theta_{*} (k)\le1 .
\]

In this work, we prove upper and lower bounds on $\theta_*$. The upper
bound follows from an
analysis of the majority process using a new technique that we call the
\textit{dynamic cavity method}.
This technique provides a precise characterization of the spin trajectory,
that is, of the process $\{\sigma_{i}(t)\}_{t\ge0}$ for a given
vertex $i$.
In particular, in the limit of large degree $k$,
this becomes a function of a well-defined Gaussian process.
Among other things, this
characterization will be used to prove that
\[
\theta_*(k)= O(1/k^M) \qquad \mbox{for any }M>0 .
\]
Thus, $\theta_*(k)$
rapidly approaches $0$ with increasing degree $k$. This result is
stated below as Theorem \ref{thm:UpperBound}.

We also prove lower bounds on $\theta_*(k)$
based on the formation of stable structures of
$-1$ spins at time $T$. Such structures, once formed, persist for all
future times, and hence prevent convergence to $\us(1)$.
These lower bounds $\theta_{\mathrm{lb}}(k,T)$ are nontrivial, that is,
strictly positive for small odd values of $k$
(cf. Table \ref{table:EmpiricalThresh_LB} in Section \ref
{sec:numerical_results}). This result is stated below as Theorem \ref
{thm:LowerBound}.

A significant part of this paper is devoted
to the rigorous development of the dynamic cavity method.
We consider this a key contribution
of this work.
The cavity method has been successful in analysis of probabilistic
models having locally
tree structured graphs \cite{SpinGlass,Talabook,DemboBrazil}.
The basic idea of this method
is to remove a node from the graph thus forming a ``cavity.''
One then assumes that the behavior of the other nodes
(surrounding the cavity) is known.
The removed node is then put back in to derive a dynamic-programming type
recursion.

Here, we show how to
extend this method to the study of a stochastic process on a tree-like graph,
specifically
the majority process.
In this setting, the cavity recursion can be
interpreted as an inductive procedure with respect to time $t$.
We ``fix''
the behavior of a selected vertex $i$ up to time $t$,
obtain a consistent characterization of its ``environment''
up to the same time $t$. From this, we can compute
the probability distribution of the trajectory
$\{\sigma_{i}(t')\}_{0\le t'\le t+1}$ up
to time $t+1$. The cavity recursion determines completely the
distribution of the
the spin trajectory at an arbitrary node, although in implicit form.

In order to analyze the cavity recursion, we consider the
large $k$ regime. However, since we want to study the
decay of $\theta_*(k)$ with $k$, we cannot rely on generic tools
and need to carry out an accurate probability calculation.
In order to achieve this goal, we establish a convenient form
of the local central limit theorem for binary random vectors.
We use this central limit theorem to ``solve'' the cavity recursion for
large $k$. The solution is given by
a ``cavity process'' that can be defined explicitly in terms of an appropriate
Gaussian process.
%
%************************************************
%
%s1.2 ###
\subsection{Preliminary remarks}

It is not too difficult to show that $\theta_{*}(k)<1$ for all~$k$.
The majority process is related to the simpler process of bootstrap
percolation \mbox{\cite{PeresCore,FontesTree}}. The next lemma formalizes
this connection, yielding a nontrivial upper bound on $\theta_{*}(k)$.
\begin{lemma}\label{lemma:LessThanOne}
For all $k\ge3$, denote by $\rho_{\mathrm{c}}(k)$ the threshold density
for the appearance of an infinite cluster of occupied vertices in
bootstrap percolation with threshold $\lfloor(k+1)/2\rfloor$.
Then
%
%e4 ###
%
\begin{equation}
\theta_{*}(k) \le\theta_{\mathrm{u}}(k) \equiv1-2\rho_{\mathrm
{c}}(k) <1 .
\end{equation}
\end{lemma}

This result follows from the fact that if the initial $-1$'s cannot
form an infinite structure
under bootstrap percolation, then they eventually all disappear under
the majority dynamics.
We defer a full proof of this lemma to the Appendix \ref{app:prelim}.

A numerical evaluation of this upper bound \cite{FontesTree} yields
$\theta_{\mathrm{u}}(5) \approx0.670$,
$\theta_{\mathrm{u}}(6) \approx0.774$,
$\theta_{\mathrm{u}}(7) \approx0.600$.
It is possible to show that
$\theta_{\mathrm{u}}(k)=O(\sqrt{(\log k)/k})$. It turns out that
this is far from being the correct $k$-dependence.
We will prove a much tighter bound in Theorem \ref{thm:UpperBound}.

The next lemma simplifies the task of proving upper bounds on
$\theta_*(k)$ for large~$k$, by showing that it is sufficient to prove
$\E_{\theta}\{\sigma_i(t)\}> 1-\delta_* /k$ for some constant
$\delta_*$ to conclude $\us(t)\to\up$.
\begin{lemma}\label{lemma:Local}
Assume $\G$ to be the regular tree of degree $k$.
There exists $k_*,\delta_*>0$ such that for $k\ge k_*$,
if $\E_{\theta}\{\sigma_i(t)\}> 1-(\delta_*/k)$, then
$\theta_*(k)\le\theta$.
\end{lemma}

We use a standard expansion argument to show that such convergence
occurs for typical random graphs
in the configuration
model, and then extend the result to the infinite tree. Again %Next, we
%show that the spin distribution in the configuration model concentrates
%The proofs of the Lemmas \ref{lemma:LessThanOne} and \ref{lemma:Local}
the proof can be found in Appendix \ref{app:prelim}.

%s1.3 ###
\subsection{Organization of the paper}
We state our main results in Section \ref{sec:Results}. Section~\ref
{sec:RelWork} surveys related work. We develop the
cavity method and the resulting upper bound in Section \ref
{sec:CavityProcess}. Our lower bound is proved in
Section~\ref{sec:LowerBound}.

%
%***********************************************************
%
%s2 ###
\section{Results}\label{sec:Results}

We can now state our main results. They consist of the following:
\begin{longlist}
\item The exact cavity recursion (Lemma \ref
{lemma:ExactCavity} in Section \ref{sec:ExactCavity}).
\item
Convergence to the cavity process (Theorems \ref{thm:unbiased_convergence}
and \ref{thm:biased_convergence} in
Section~\ref{sec:UB_cavity_inResults}).
\item Upper bound on $\theta_*$ (Theorem
\ref{thm:UpperBound} in Section \ref{sec:UB_cavity_inResults}) as
a consequence of convergence to the cavity process.
\item Lower bound on $\theta_*$ (Theorem \ref
{thm:LowerBound} in Section \ref{sec:LowerBound_inResults}) due
to formation of blocking structures of $-1$'s.
\end{longlist}
Section \ref{sec:numerical_results} contains numerical illustration of
some of our results.

%s2.1 ###
\subsection{The exact cavity recursion}
\label{sec:ExactCavity}

First, we state an exact recursive characterization of spin
trajectories at nodes.
This is the key tool we use in our development of the cavity process.
Moreover, our
lower bound is based on a very similar recursive analysis. %, besides
%requiring notation from here.
%{\bf[Y: We have this subsection because we need the notation for the
%lower bound.
%Not sure how to convey this.]}

Let $\G_{\rroot}=(\V_{\rroot},\Ed_{\rroot})$ be the
tree rooted at vertex $\rroot$ with degree $k-1$ at the root and
$k$ at all the other vertices, and let $u= \{u(0),u(1),u(2),\ldots\}$
be an arbitrary sequence of real numbers.
We define a modified Markov chain
over spins $\{\sigma_i\}_{i\in\V_{\rroot}}$ as follows. For $i\neq
\rroot$,
$\sigma_i(t)$ is updated according to the rules (\ref
{eq:MajorityRule}) and
(\ref{eq:RandomUpdate}). For the root spin we have instead
%
%e5 ###
%
\begin{equation}\label{eq:ModifiedUpdateRoot}
\sigma_{\rroot}(t+1) = \sign\Biggl(\sum_{i=1}^{k-1}\sigma
_i(t)+u(t)\Biggr) ,
\end{equation}
where $1,\ldots,k-1$ denote the neighbors of the root. In the
case $\sum_{i=1}^{k-1}\sigma_i(t)+u(t)=0$, $\sigma_\rroot(t+1)$ is
drawn as in (\ref{eq:RandomUpdate}), that is, uniformly at random.
We will call this
the ``dynamics under external field.''

We will call the sequence $u= \{u(0),u(1),u(2),\ldots\}$
``external field applied at the root.'' We define the notation
$u^T \equiv(u(0), u(1), \ldots, u(T))$ and similarly
$\sigma_i^T \equiv(\sigma_i(0),\sigma_i(1),\ldots,\sigma_i(T))$.
We denote by
$\prob( \sigma_\rroot^T \midd u^T)$ the probability of
observing a trajectory
$\sigma_\rroot^T$
for the root spin under the above dynamics.
Let us stress two elementary facts:
(i) $\prob( \sigma_\rroot^T \midd u^T)$ is not a
conditional probability;
(ii) as implied by the notation, the distribution of
$\sigma_\rroot^T$ does not depend on $u(t)$, $t>T$
[and indeed does not depend on $u(T)$ either, but we include it for
notational convenience].

As before, we assume that in the initial configuration, the spins
are i.i.d. Bernoulli random variables, and denote by $\prob_0(\sigma_i(0))$
their common distribution.
\begin{lemma}
\label{lemma:ExactCavity}
The following recursion holds
%
%e7 ###
%e6 ###
%
\begin{eqnarray}
\label{eq:Recursion}
\prob(\sigma_\rroot^{T+1}\midd u^{T+1}) &=& \prob_0(\sigma
_\rroot(0))
\sum_{\sigma_1^{T}\cdots\sigma_{k-1}^{T}}
\prod_{t=0}^T
\K_{u(t)}\bigl(\sigma_{\rroot}(t+1)|\sigma_{\partial\rroot
}(t)\bigr)\nonumber\\[-8pt]\\[-8pt]
&&\hphantom{\prob_0\bigl(\sigma
_\rroot(0)\bigr)
\sum_{\sigma_1^{T}\cdots\sigma_{k-1}^{T}}}
{}\times
\prod_{i=1}^{k-1}\prob(\sigma_i^{T}\midd \sigma_\rroot^{T}
) ,
\nonumber\\
\label{eq:Kdef}
\K_{u(t)}(\cdots)
&\equiv&\cases{
\displaystyle \ind\Biggl\{\sigma_\rroot(t+1) = \sign\Biggl(\sum_{i=1}^{k-1}\sigma_i(t)+
u(t)\Biggr)\Biggr\}, \vspace*{2pt}\cr
\qquad\hspace*{12.7pt}\mbox{if $\displaystyle \sum_{i=1}^{k-1}\sigma_i(t)+
u(t)\neq0 $},\cr
\dfrac{1}{2}, \qquad \mbox{otherwise}.}
\end{eqnarray}
\end{lemma}

%The correctness of this recursion for the majority dynamics is fairly
%self evident.
This recursion is based on the following intuition: for ease of
explanation, we will assume $k$ is odd. The situation of the ``child''
nodes $1, 2,
\ldots, k-1$ (and their respective subtrees) with respect to the root
trajectory $\sigma_\rroot^T$ is essentially the same as the situation
of $\rroot$ with respect to the fixed trajectory $u^T$. If the root
trajectory had been ``fixed'' to $\sigma_\rroot^T$, the child
trajectories would have been
i.i.d. according to $\prob(\cdot\midd\sigma_\rroot^T)$.
However, the root trajectory $\sigma_\rroot^T$ is itself a function
of the child trajectories $\sigma_1^{T}, \ldots, \sigma_{k-1}^{T}$
and $u^T$, instead of being a fixed trajectory.
%(though the child nodes have no way of telling that this is the case).
Thus we sum over the product of terms $\prob(\sigma_i^{T}\midd
\sigma_\rroot^{T})$ multiplied by a ``consistency'' indicator
$\prod_{t=0}^T
\K_{u(t)}(\sigma_{\rroot}(t+1)|\sigma_{\partial\rroot
}(t))$. The term $\prob_0(\sigma_\rroot(0))$ appears for
obvious reasons.

We provide a rigorous proof of Lemma \ref{lemma:ExactCavity} in
Appendix \ref{app:ExactCavityRec}.
%{\bf[Y: Change definition of $K$ to make it single argument no
%subscript? Define
%$S(\partial\rroot)(t) \equiv\sum_{i=1}^{k-1} \sigma_i(t)$? Current
%seems convenient
%here but very painful in the proof of $\qprob- \prob$ lemmas.]}

The same proof applies, in fact, to quite a general class
of processes on the regular rooted tree $\G_{\rroot}$.
More precisely, consider a model with spins taking value in a finite domain
$\sigma_i(t)\in\cX$, and are updated in parallel according to
the rule (for $i\neq\rroot$)
%
%e8 ###
%
\begin{equation}\label{eq:GeneralUpdate}
\sigma_i(t+1) = f\bigl(\sigma_i(t),
\us_{\di\setminus\pi(i)}(t),\sigma_{\pi(i)}(t),\omega_{i,t}\bigr),
\end{equation}
where $\pi(i)$ is the parent of node $i$ (i.e., the only
neighbor of $i$ that is closer to the root) and $\{\omega_{i,t}\}$
are a collection
of i.i.d. random variables. For the root $\rroot$, the above rule is modified
by replacing $\sigma_{\pi(i)}(t)$ by the arbitrary quantity $u(t)$.

The next remark follows from a verbatim repetition of our proof of
Lemma \ref{lemma:ExactCavity} (in Appendix \ref{app:ExactCavityRec}).
\begin{remark}
\label{rem:dynamic_cavity_general_update}
For a model with general update rule (\ref{eq:GeneralUpdate}),
the distribution of the root trajectory satisfies (\ref{eq:Recursion})
with the kernel
%
%e9 ###
%
\begin{eqnarray}
&&\K\bigl(\sigma_{\rroot}(t+1)|\sigma_\rroot(t),\sigma_{\partial
\rroot}(t)\bigr)\nonumber\\[-8pt]\\[-8pt]
&&\qquad\equiv
\E_{\omega_{\rroot,t}}
\bigl\{
\ind\bigl(\sigma_{\rroot}(t+1) =
f(\sigma_\rroot(t),
\us_{\partial\rroot}(t),u(t),\omega_{\rroot,t}) \bigr)
\bigr\} .\nonumber
\end{eqnarray}
\end{remark}

%s2.2 ###
\subsection{Upper bounds and the dynamic cavity method}
\label{sec:UB_cavity_inResults}

While for small odd $k$ the consensus threshold is strictly positive,
our next result shows that it approaches $0$ very rapidly as $k\to
\infty$.
\begin{theorem}\label{thm:UpperBound}
The consensus threshold on $k$ regular trees converges to
$0$ as $k\to\infty$ faster than any polynomial. In other words,
for any finite $B>0$ and any $K>0$,
%
%e10 ###
%
\begin{equation}
\theta_*(k)\le K k^{-B}
\end{equation}
for $k \geq k_*(B,K)$.
\end{theorem}

%***********************************
%
Fix a vertex $i\in\V$, and consider the process $\{\sigma_i(t)\}
_{t\ge0}$.
The proof of Theorem~\ref{thm:UpperBound} is obtained by developing a
pretty complete characterization of this process in the large $k$ limit.
We first consider the unbiased case (i.e., $\theta=0$) and prove
the convergence of this process $\{\sigma_i(t)\}_{t\ge0}$
to a well-defined limit as $k\to\infty$. We will call this limit the
\textit{cavity process}, for the case of unbiased initialization
(i.e., for
$\theta=0$). We formally define the cavity process below and then state
our result on convergence
to the cavity process.
% relies on the use of an appropriate central limit theorem
%on the exact recursion (\ref{app:ExactCavityRec}).
%
\begin{definition}[(Effective process)]
Let $C = \{C(t,s)\}_{t,s \in\Z_+}$ be a positive definite symmetric matrix,
and $R = \{R(t,s)\}_{ t,s \in\Z_+, t>s}$,
$h = \{h(t)\}_{ t \in\Z_+}$ two arbitrary sets of real numbers.

A sample path of the \textit{effective process} with parameters
$C, R, h$ is generated as follows:
let $\tau(0)$ be a Bernoulli$(1/2)$ random variable and
$\{\eta(t)\}_{t\in\Z_+}$ be jointly Gaussian zero mean
random variables with covariance $C$, independent from $\tau(0)$.
For any $t\ge0$ we let
%
%e11 ###
%
\begin{equation}\label{eq:asymp_dyn_def}
\tau(t+1)=\sign\Biggl(\eta(t) + \sum
_{s=0}^{t-1}{R(t,s)\tau(s)} + h(t) \Biggr) .
\end{equation}
\end{definition}

Notice that the distribution of the effective process depends on the three
parameters $C,R,h$. We will denote expectation with respect
to its distribution as $\E_{C,R,h}$. The functions $C( \cdot,
\cdot)$
and $R( \cdot, \cdot)$ will be referred to as \textit
{correlation} and
\textit{response} functions. By convention, we let $R(t,s)=0$ if $t\le s$.
Finally, $h$ is a perturbation
parameter needed to state our definition of the cavity process in terms
of the effective process.
\begin{definition}[(Consistent parameters $C$, $R$)]\label{def:consistency}
Let $C, R$ be parameters in the definition of the effective process. We
say that $C$, $R$ are consistent if they satisfy
%
%e13 ###
%e12 ###
%
\begin{eqnarray}
\label{eq:cavity_def1}
C(t,s) &= &\E_{C,R,0}[\tau(t)\tau(s)]\qquad
\forall t,s \ge0 ,
\\
\label{eq:cavity_def2}
R(t,s) &= &\frac{\partial
}{\partial h(s)}\E_{C,R,h}[\tau(t)] \bigg|_{h=0}\qquad
\forall0 \le s < t .
\end{eqnarray}
\end{definition}

It is natural to ask whether consistent choices of $C$ and $R$ exist,
and in that case, whether they are unique or not.
This question is addressed in Lemma \ref{lemma:Existence} below,
which proves that \textit{there exist unique
consistent $R$ and $C$}, that is, there is a unique solution
of (\ref{eq:cavity_def1}) and (\ref{eq:cavity_def2}).
In fact, these values are determined recursively.
One starts $C(0,0)=1$ [and indeed $C(t,t)=1$ for all $t$].
This leads to uniquely determined values for $C(1,0)$
and $R(1,0)$, which then determines unique values for $C(2,s)$, $R(2,s)$
and so on.
\begin{definition}[(Cavity process)]
Let $C$, $R$ be the unique consistent parameters (cf. Definition \ref
{def:consistency}) as per Lemma
\ref{lemma:Existence}. The \textit{cavity process} $\{\tau(t)\}
_{t\in
\Z_+}$
defined as the effective process with parameters $C$, $R$ and with $h=0$.
\end{definition}

In the following, we will denote by $\prob_{\mathrm{cav}}$
the law of the cavity process.
Our next theorem establishes convergence of the majority process with
unbiased initialization to the cavity process.
\begin{theorem}
\label{thm:unbiased_convergence}
Consider the majority process on a regular tree of degree $k$
with uniform initialization $\theta=0$.
Then for any $i \in\mathcal{V}$ and any time $T\ge0$, we have
\[
( \sigma_i(0), \sigma_i(1), \ldots, \sigma_i(T) )
\stackrel{d}{\rightarrow} ( \tau(0), \tau(1), \ldots, \tau
(T) ) ,
\]
where $\{\tau(0)\}_{t\ge0}$ is distributed according to the cavity
process and convergence is understood to be in distribution
as $k\to\infty$.
\end{theorem}

Theorem \ref{thm:unbiased_convergence} is proved in Section \ref
{sec:unbiased_convergence}.

Let us describe the intuitive picture which forms the basis of the last theorem.
The trajectory at target node $i$ follows the majority rule in
(\ref{eq:MajorityRule}). The study
of this rule is complicated by the fact that the spins of the
neighboring nodes $\partial i$
at time $t>0$ are not independent of each other. The past trajectory
$\sigma_i^{t-1}$ of target node $i$ affects the
spins of nodes in $\partial i$ at time $t$. The exact recursion
equation (\ref{eq:Recursion}) allows
an analytical treatment despite this dependence. We use a local central
limit theorem (Theorem \ref{thm:sum_approx}
in Section \ref{sec:sum_approx}, proved in Appendix \ref{app:CLT})
on the exact recursion equation (\ref{eq:Recursion}), to show convergence
to the cavity process
inductively in $T$. The response term $\sum_{s=0}^{t-1}{R(t,s)\sigma
(s)}$ captures the effect
of the spin trajectory up to time $t-1$ at the target node, on its
environment at time $t$. The key part of the proof is in Lemma \ref
{lemma:RCconvergence}.

We finally turn to the case of biased initialization $\E_{\theta}\{
\sigma_i(0)
\}=\theta$.
\begin{theorem}
\label{thm:biased_convergence}
For $T_*$ a nonnegative integer and $\bias_0\ge0$,
consider the majority process on a regular tree of degree $k$ with
i.i.d. initialization with bias $\theta= \bias_0/k^{(T_*+1)/2}$.
Then for any $i\in V$ and $T\le T_*$, we have
\[
( \sigma_i(0), \sigma_i(1), \ldots, \sigma_i(T) )
\stackrel{d}{\rightarrow} ( \tau(0), \tau(1), \ldots, \tau
(T) ) ,
\]
where $\{\tau(0)\}_{t\ge0}$ is distributed according to the cavity
process and convergence is understood to be in distribution
as $k\to\infty$.

Further, if $\bias_0>0$, then for any $i\in V$ and $T\ge T_*+2$, we have
%
%e14 ###
%
\begin{eqnarray}\label{eq:biased_convergence}
&&( \sigma_i(0), \sigma_i(1), \ldots, \sigma_i(T)
)\nonumber\\[-8pt]\\[-8pt]
&&\qquad
\stackrel{d}{\rightarrow} \bigl( \tau(0), \tau(1), \ldots, \tau
(T_*), \sigma(T_*+1), +1, +1, \ldots, +1 \bigr) ,\nonumber
\end{eqnarray}
where the random variable
$\sigma(T_*+1)$ dominates stochastically
$\tau(T_*+1)$, and $\prob\{\sigma(T_*+1)>\tau(T_*+1)\}$
is strictly positive.

Finally, there exist $\const= \const(\bias_0, T_*)$, with
$\const(\bias_0,T_*)>0$ for $\bias_0>0$ such that, for any $T\ge T_*+2$,
%
%e15 ###
%
\begin{equation}\label{eq:BoundMagn}
\E_{\theta}\{\sigma_i(T)\} \ge1-e^{-\const(\bias_0,T_*)k}.
\end{equation}
\end{theorem}

Theorem \ref{thm:biased_convergence} is proved in Section \ref
{sec:biased_convergence}.
Theorem \ref{thm:UpperBound} is an immediate corollary of this general result.
\begin{pf*}{Proof of Theorem \ref{thm:UpperBound}}
Choose $T_*=\lceil2B \rceil$ and $\bias_0=K$ in Theorem \ref
{thm:biased_convergence}
and use (\ref{eq:BoundMagn}) to check the assumptions of
Lemma \ref{lemma:Local}, whereby for $t\ge T_*+2$, it is sufficient to
take $k\ge k_*$ such that $\delta_*/k\ge e^{-\const(\bias_0,T_*) k}$.
\end{pf*}

Clearly, Theorem \ref{thm:unbiased_convergence} is a special case of
Theorem \ref{thm:biased_convergence}
(just take $T_*$ large enough and $\bias_0=0$).
However our proof proceeds by first analyzing the unbiased
case $\theta=0$, and then turning to the biased one $\theta>0$.
The latter is treated by establishing a delicate relationship
between processes with biased and unbiased initializations,
derived in Lemmas \ref{lemma:qminusp}
and \ref{lemma:qminusptstar}. In the unbiased case, one has
$\E\{\sigma_i(t)\} = 0$ by symmetry at all times.
In the biased case, we will prove a quantitative estimate of how the
fraction of $+1$ spins evolves with time. Let $\theta_t = \E[\sigma
_i(t)]$ for an arbitrary node
$i$. If $\theta_t = O(k^{-1})$ we obtain $\theta_{t+1} = \sqrt{k}
\theta_t R(t+1, t) (1 + o(1))$.
In words, as long as the fraction of $+1$ spins is small enough,
it gets multiplied at each step by a factor of order $\sqrt{k}$.
By iterating this procedure with $\theta_0=\bias_0/k^{(T_*+1)/2}$,
we get $\theta_{T_*} = \Theta(k^{-1/2})$
and $\theta_{T_*+1}= \Theta(1)$.
At the next iteration, the fraction of $+1$ spins
approaches $1$ and the bias saturates to $\theta_{T_*+2}=1-e^{-\Theta(k)}$.

Theorem \ref{thm:biased_convergence} also implies that, for large
degree trees,
consensus to majority takes place very abruptly.
Indeed the bias toward $+1$ passes from $k^{-1/2}$
(at $t= T_*$) to $1-e^{-\Theta(k)}$ (at $t= T_*+2$) in $2$ iterations.
Numerical illustrations of this phenomenon are provided in
Section \ref{sec:numerical_results},
specifically Figures \ref{fig:tstar1} and~\ref{fig:tstar3}.%\looseness=1
%
%****************************************************
%
%s2.3 ###
\subsection{Lower bounds}
\label{sec:LowerBound_inResults}

We state
a sequence of recursively computable lower bounds. These lower bounds
are based on the formation of ``stable'' structures of $-1$'s. Once
formed, these stable structures persist for all future times, hence
preventing the system from reach the $\underline{+1}$ consensus. A key
issue we overcome is that such stable structures do not exist at time 0
(w.p. 1) for any $k$ and $\theta>0$. The lower bound $\theta_{\mathrm{lb}}
(k,T)$ stated below is based on the formation w.p. 1 of stable
structures of $-1$'s at time $T$, as a result of majority dynamics up to
that time.
\begin{theorem}\label{thm:LowerBound}
Consider any $T \geq0$. For all $\sigma^T$, $u^T\in\{-1,+1\}^{T+1}$ define
%
%e16 ###
%
\begin{eqnarray}\label{eq:initial_psi}
\Psi_{{\mathrm{odd}},T}^0(\sigma_\rroot^T\midd u^T) &=& \prob
(\sigma_\rroot^T\midd u^T) , \nonumber\\[-8pt]\\[-8pt]
\Psi_{{\mathrm{even}},T}^0(\sigma_\rroot^T\midd u^T) &=& \prob
(\sigma_\rroot^T\midd u^T) \ind\bigl(\sigma_\rroot(T)=-1\bigr)
.\nonumber
\end{eqnarray}
Define $\Psi_{{\mathrm{odd}},T}^{d+1}(\sigma_\rroot^T\midd u^T
), \Psi_{{\mathrm{even}},T}^{d+1}(\sigma_\rroot^T\midd u^T)$
for $d\ge0$ recursively as per
%
%e19 ###
%e18 ###
%e17 ###
%
\begin{eqnarray}
\label{eq:Alt_fp1}
&&\Psi_{{\mathrm{odd}},T}^{d+1}(\sigma_\rroot^T\midd u^T)\nonumber\\[-1pt]
&&\qquad=
\prob_0(\sigma_\rroot(0))
\sum_{r=\lceil({k+1})/{2} \rceil-1}^{k-1}\pmatrix
{k-1\cr r} \nonumber\\[-1pt]
&&\qquad\quad{}\times\sum_{\sigma_1^{T}\cdots\sigma_{k-1}^{T}}
\prod_{t=0}^{T-1}
\K_{u(t)}\bigl(\sigma_{\rroot}(t+1)|\sigma_{\partial\rroot
}(t)\bigr) \\[-1pt]
&&\qquad\quad\hphantom{{}\times\sum_{\sigma_1^{T}\cdots\sigma_{k-1}^{T}}}
{}\times \prod_{i=1}^{r}\Psi_{{\mathrm{even}},T}^d(\sigma
_i^{T}\midd \sigma_\rroot^{T})\nonumber\\[-1pt]
&&\qquad\quad\hphantom{{}\times\sum_{\sigma_1^{T}\cdots\sigma_{k-1}^{T}}}
{}\times
\prod_{i=r+1}^{k-1}\{\prob(\sigma_i^{T}\midd \sigma_\rroot
^{T})- \Psi_{{\mathrm{even}},T}^d(\sigma_i^{T}\midd \sigma
_\rroot^{T})\} ,\nonumber
\\[-1pt]
\label{eq:Alt_fp2}
&&\Psi_{{\mathrm{even}},T}^{d+1}(\sigma_\rroot^T\midd u^T) \nonumber\\[-1pt]
&&\qquad=
\ind\bigl(\sigma_\rroot(T)=-1\bigr) \prob_0(\sigma_\rroot(0))
\sum_{r=\lceil({k+1})/{2} \rceil-1}^{k-1}\pmatrix
{k-1\cr r} \nonumber\\[-1pt]
&&\qquad\quad{}\times\sum_{\sigma_1^{T}\cdots\sigma_{k-1}^{T}}
\prod_{t=0}^{T-1}
\K_{u(t)}\bigl(\sigma_{\rroot}(t+1)|\sigma_{\partial\rroot
}(t)\bigr) \\[-1pt]
&&\qquad\quad\hphantom{{}\times\sum_{\sigma_1^{T}\cdots\sigma_{k-1}^{T}}}
{}\times \prod_{i=1}^{r}\Psi_{{\mathrm{odd}},T}^d(\sigma
_i^{T}\midd \sigma_\rroot^{T})\nonumber\\[-1pt]
&&\qquad\quad\hphantom{{}\times\sum_{\sigma_1^{T}\cdots\sigma_{k-1}^{T}}}
{}\times
\prod_{i=r+1}^{k-1}\{\prob(\sigma_i^{T}\midd \sigma_\rroot
^{T})- \Psi_{{\mathrm{odd}},T}^d(\sigma_i^{T}\midd \sigma
_\rroot^{T})\} ,\nonumber
\end{eqnarray}
\begin{eqnarray}
\K_{u(t)}(\cdots)
&\equiv&\cases{
\displaystyle \ind\Biggl\{\sigma_\rroot(t+1) = \sign\Biggl(\sum_{i=1}^{k-1}\sigma_i(t)+
u(t)\Biggr)\Biggr\}, \vspace*{2pt}\cr
\qquad\hspace*{12.96pt} \mbox{if $\displaystyle \sum_{i=1}^{k-1}\sigma_i(t)+
u(t)\neq0 $},\cr
\dfrac{1}{2}, \qquad \mbox{otherwise}.}
\end{eqnarray}
Let $\Psi_{{\mathrm{odd}},T}(\sigma^T \midd u^T ) = \lim
_{d\rightarrow
\infty} \Psi^d_{{\mathrm{odd}},T}(\sigma^T \midd u^T )$.
This limit exists.

Define $\theta_{\mathrm{lb}}(k,T)\equiv\sup\{\theta\in[0,1]\dvtx
{\Psi}_{{\mathrm{odd}},T}(\sigma^T \midd u^T )> 0$ for all
$\sigma^T ,u^T \}$. Then, for every $k, T$
%
%e20 ###
%
\begin{equation}
\theta_*(k)\ge\theta_{\mathrm{lb}}(k,T) .
\end{equation}
\end{theorem}

It is obvious that evaluating the lower bound
$\theta_{\mathrm{lb}}(k,T)$ analytically is quite challenging.
An exception is provided by the case $k=3$, where it is not
too hard to show that $\theta_{\mathrm{lb}}(k=3,T=1)>0$.

We will instead evaluate the lower bounds
$\theta_{\mathrm{lb}}(k,T)$ numerically.
The above recursion allows us to do it through a number of operations
(sums and multiplications) of order $2^{k(T+1)}T(T+k)$.
As explained in Section \ref{sec:LowerBound},
the recursion can be considerably simplified by exploiting the symmetries
of the
problem, while remaining exponential in $k$ and $T$.
Evaluating the lower bound for $k=3$, $5$, $7$
and $T=3$, we get $\theta_*(3)> 0.573$,
$\theta_*(5)> 0.052$ and $\theta_*(7)> 0.0080$.
This shows convincingly that $\theta_*(k)>0$ for $k\le7, k$ odd.
%
%****************************************************
%
%s2.4 ###
\subsection{Numerical illustration}\label{sec:numerical_results}

The objective of this section is to provide illustrations
of our results, and help to develop some intuition on the majority process.

It is obviously difficult to simulate
the majority dynamics on infinite trees. On the other
hand, the state of any node $i$ after $t$ iterations only depends
on the state of its neighbors in the graph up to distance $t$.
It is natural to consider sequences of finite graphs having an
increasing number of vertices $n$, that converge locally
to trees (in the sense of \cite{AldousSteele}). Random regular graphs drawn
from the configuration model \cite{BollobasConf}
are a natural choice.
A sequence of random $k$ regular graphs does indeed converge to
the regular tree of degree $k$ almost surely (e.g., see \cite{DemboBrazil}).

Moreover, as demonstrated in Lemmas \ref{lemma:GraphTree} and
\ref{lemma:Concentration}, the fraction
of nodes that are $+1$ at time $t$ in the configuration model converges to
the probability in the infinite tree of an arbitrary node being $+1$.

It is worth emphasizing that we are using random regular graphs
as a tool for computing the evolution of the fraction
of $(+1)$'s on the infinite tree. This approach is
supported by Lemmas \ref{lemma:GraphTree} and
\ref{lemma:Concentration}. On the other hand,
we will not attack the problem of defining a consensus
threshold for finite graphs. This indeed requires some care as we briefly
explain for clarity.

The consensus threshold $\theta_*$ is well defined for a
general infinite graph $\G$. If $\G$ is finite,
then trivially $\theta_*(\G)=1$: indeed for any
$\theta<1$ there is a positive probability that $\us(0)$ is the
all $-1$ configurations. However, given a sequence of graphs
with an increasing number of vertices $n$, one can define a
threshold function $\theta_{*,n}(\gamma)$ such that $\us(t)\to
\underline{+1}$
with probability $\gamma$ for $\theta=\theta_{*,n}(\gamma)$.
%$\theta_{*,n}(\gamma)= \textup{inf} \{\theta: \prob(\us(t)\to
It is an open question to determine which graph sequences
exhibit a sharp threshold [in the sense that $\theta_{*,n}(\gamma)$
has a limit independent of $\gamma\in(0,1)$ as $n\to\infty$].
It is a natural conjecture that such a sharp threshold does indeed
exist for sequences of random regular graphs.

We carried out numerical simulations with random regular graphs
of degree $k$.\setcounter{footnote}{2}\footnote{We used graphs of size up to $n=5\cdot10^4$,
generated according
to a modified \textit{configuration model}
\cite{BollobasConf}
(with eventual self-edges and
double edges rewired randomly).
The initial bias was implemented by drawing a uniformly random
configuration with $n(1+\theta)/2$ spins $\sigma_i=+1$.}
In this case, there appears empirically
to be a sharp threshold bias that converges, as $n\to\infty$ to a limit
$\theta_{*,\mathrm{rgraph}}(k)$.
Above this threshold, the dynamics converge with high probability
to all $+1$.
Below this threshold, the dynamics converge instead to either a stationary
point or to a length-two cycle
\cite{Olivos}.
Threshold biases found for small values of $k$ are shown in Table~\ref{table:EmpiricalThresh_LB}.

%t1 ###
%
\begin{table}
\tablewidth=275pt
\caption{Empirical thresholds $\theta_{*,\mathrm{rgraph}}(k)$ and
computed lower bounds on $\theta_*(k)$} \label{table:EmpiricalThresh_LB}
\begin{tabular*}{\tablewidth}{@{\extracolsep{\fill}}lcd{1.3}@{}}
\hline
$\bolds{k}$ & $\bolds{\theta_{*,\mathrm{rgraph}}(k)}$ & \multicolumn{1}{c@{}}{\textbf{Lower bd on $\bolds{\theta
_*(k)}$ from Theorem \ref{thm:LowerBound}}}\\
\hline
$3$ & $0.58\pm0.01$ & 0.574\\
$4$ & $0.000\pm0.001$ & 0 \\
$5$ & $0.054\pm0.001$ & 0.052\\
$6$ & $0.000\pm0.001$ & 0\\
$7$ & $0.010\pm0.001$ & 0.008\\
\hline
\end{tabular*}
\end{table}

The empirical threshold for the configuration model approaches $0$
rapidly with increasing $k$,
for $k$ odd, and appears to be identically $0$ for all even $k$.
The origin of the odd--even difference lies in the fact that, for $k$
odd, the majority dynamics is deterministic. For $k$ even,
the possibility of ties leads to random choices [cf.
(\ref{eq:RandomUpdate})] thus reducing the chance of blocking structures.
Getting a rigorous understanding of this phenomenon
is an open problem.

For comparison, we have shown above the best lower bound value we could
compute based on Theorem \ref{thm:LowerBound} (combined with the
trivial lower bound of $0$). The lower bounds we have obtained for the
tree process are very close to the empirical thresholds $\theta
_{*,\mathrm{rgraph}}(k)$.
% when $\theta_{*,\mathrm{rgraph}}(k)>0$ suggesting that we have
%identified
%the correct consensus blocking mechanism in computing the lower
%bounds. %, each for $T=3$, i.e., based
%on formation of a persistent structure of $-1$ spins by time $3$.
A full table of computed lower bound values is available in Table \ref
{table:lower_bounds}.

%
%f1 ###
%
\begin{figure}

\includegraphics{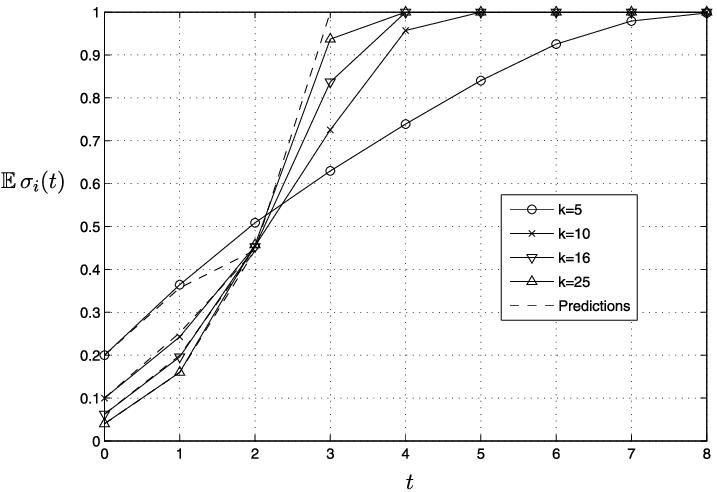}

\caption{Change of bias $\E\sigma_i(t)$ over time $t$, with
with initial bias $\E\sigma_i(0) \equiv\theta= 0.5/k$ (i.e.,
in our notation $T_*=1$, $\bias_0=0.5$).
The ``prediction'' is based on $\bias_1, \ldots, \bias_{T_*}$ computed
according
to (\protect\ref{eq:omega_recursion}) and $\bias_{T_*+1}$ computed
according to the
modified cavity process [see Lemma \protect\ref{lemma:rooted_biased_convergence} and
(\protect\ref{eq:modified_cavity_recursion})].}\label{fig:tstar1}
\end{figure}

%
%f2 ###
%
\begin{figure}

\includegraphics{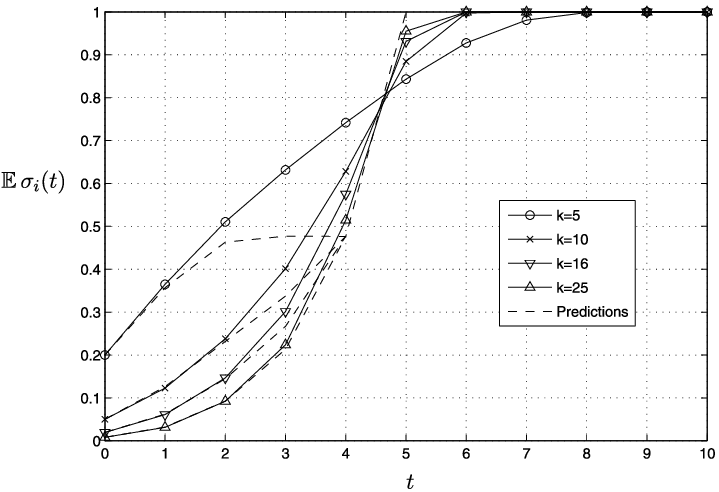}

\caption{Change of bias $\E\sigma_i(t)$ over time $t$, with
initial bias $\theta=\E\sigma_i(0) = 2.5/k^{2}$
(i.e., $T_*=3$, $\bias_0=2.5$).}
\label{fig:tstar3}
\end{figure}

Figures \ref{fig:tstar1} and \ref{fig:tstar3} compare our predictions
for the evolution of $\theta_t$ with the average observed values for
finite values of $k$.
%We used large graphs (n=50000) drawn from the configuration
%model in our simulations.
Theorem \ref{thm:biased_convergence} predicts
%Informally, if $\theta\approx\bias_0/k^{(T_*+1)/2}$, then
almost complete consensus is reached sharply at iteration $T_*+2$.
%This phenomenon is illustrated through numerical simulations
%in Figures \ref{fig:tstar1} and \ref{fig:tstar3}.
We see that the prediction
provided by our method is quite accurate already for $k\gtrsim15$.
In particular, consensus develops fairly rapidly between iteration $T_*$
and $T_*+2$.
%
%*****************************************************************
%
%s3 ###
\section{Related work}
\label{sec:RelWork}

The majority process is a simple example of a stochastic dynamics evolving
according to local rules on a graph. In the last few years,
considerable effort has been devoted to the study
of high-dimensional probability distributions with
an underlying sparse graph structure (e.g., see \cite{IPC}). Such distributions
are referred to as Markov random fields,
graphical models, spin models or constraint satisfaction problems,
depending on the context. Common algorithmic and analytic tools
were developed to address a number of questions ranging from
statistical physics to computer science. Among such tools, we
recall local weak convergence \cite{BenjSchr,AldousSteele} and
correlation decay \cite{Dobrushin,Weitz,Georgii}.
% variational
%approximations, and the cavity method \cite{MezParZec_science,OurPNAS}.

The objective of the present paper
is to initiate a similar development in the context of
stochastic dynamical processes
that ``factor'' according to a sparse graph structure.
Rather than addressing a generic setting, we focus instead
on a challenging concrete question, and try to develop tools that are
amenable to generalization.

The majority process can be regarded as a example of
interacting particle system \cite{Liggett} or as a cellular automaton,
two topics with a long record of important results.
In particular, it bears some resemblance with the voter model.
The latter is however considerably simpler because of the underlying
martingale structure. Further, the voter model
does not exhibit any sharp threshold for $\theta_*(k) <1$.

More closely related to the model studied in this paper is the
zero-temperature Glauber dynamics for the Ising model, which obeys
the same update rule as in \mbox{(\ref{eq:MajorityRule}),
(\ref{eq:RandomUpdate})}. Let us stress that Glauber dynamics is defined
to be asynchronous: each spin is updated at the arrival times of an
independent Poisson clock of rate $1$.
Fontes, Schonmann and Sidoravicius \cite{Fontes} studied this dynamics
on $d$-dimensional grids, proving that the consensus threshold
is $\theta_*<1$ for all $d\ge2$. Howard \cite{Howard} studied
zero-temperature
Glauber dynamics on $3$-regular trees and found that infinite ``spin
chains'' of
both signs are formed almost surely
at positive times if we start with an unbiased
initialization. In our notation, this implies $\theta_* > 0$ for this model.
Positive-temperature Glauber dynamics
on trees was the object of several recent papers \cite{BergerEtAl,Martinelli}.
While no ``complete consensus'' can take place for positive temperature,
at small enough temperature this model exhibits coarsening,
namely the growth of a positively
(or negatively) biased domain.
In particular, Caputo and Martinelli \cite{CaputoMartinelli} proved
that the corresponding threshold $\theta_{*,\mathrm{coars}}(k)\to0$
as $k\to\infty$.

As mentioned, an important difference with respect to these
studies lies in the fact that we focus on synchronous dynamics.
Indeed our methods are
somewhat simpler to apply to the synchronous case.
Nevertheless, we think that
they can be generalized to the asynchronous setting as well.
In particular, we expect that a limit theorem analogous
to Theorem \ref{thm:unbiased_convergence} (with a proper definition
of the cavity process) can be proved
for Glauber dynamics as well.
As for the lower bounds on~$\theta_*(k)$, we imagine that arguments similar
to the ones leading to Theorem \ref{thm:LowerBound} can be developed
also for the asynchronous case. For instance, the result by
Howard \cite{Howard} on $k=3$ referred indeed to asynchronous dynamics.

More important is the difference between
trees and grids. The methods developed in this paper are well suited
for analyzing stochastic processes on locally tree-like graphs, while
a good part of the literature on Glauber dynamics focused on $d$-dimensional
grids.

Variations of the majority dynamics on locally tree-like graphs have
been studied recently within the statistical mechanics literature
\cite{Hatchett04,Neri2009}. Both of these papers
analyze the synchronous dynamics. In particular, the latter paper uses a
nonrigorous version of the cavity method.

The main technical ideas developed in this paper are quite
far from the ones within interacting particle systems.
More precisely, we develop a dynamical analogue of
the so-called ``cavity method'' that has been successful in the
analysis of probabilistic models on sparse random graphs. The basic idea
in that context is to exploit the locally tree-like structure
of such graphs to derive an approximate dynamic-programming type
recursion. This idea was further developed mathematically in the
local weak convergence framework of Aldous and Steele
\cite{AldousSteele}. Adapting this framework
to the study of a stochastic process is far from straightforward.
First of all, one has to determine what quantity to write
the recursion for. It turns out that an exact recursion can be proved
for the probability distribution of the trajectory of the root
spin in a modified majority process (see Section \ref{sec:ExactCavity}
for a precise definition). The next difficulty consists of extracting
useful information from this recursion which is rather implicit and intricate.
We demonstrate that this can be done for large $k$ using an appropriate
local central limit theorem proved in Appendix \ref{app:CLT}.
This allows to prove convergence to the cavity process;
see Theorems \ref{thm:unbiased_convergence} and
\ref{thm:biased_convergence}. There has been previous work in this spirit
(for other models)
that uses a normal approximation or a large degree limit (e.g., see
\cite{Talabook,Chatterjee}).

The use of a dynamic cavity method for analyzing stochastic
dynamics was pioneered in the statistical physics literature
on mean field spin glasses.
In that context, one is typically interested in the asymptotic
behavior of Langevin dynamics for large system sizes.
The energy function is taken to be a spin-glass Hamiltonian,
and the cavity method can be used to explore this asymptotics.
A lucid (albeit nonrigorous) discussion can be found in
\cite{SpinGlass}, Chapter VI.
This approach allows one to derive limit deterministic equations
for the covariance and the ``response function'' of the process under
study. The study of such equations led to a deeper
understanding of fascinating phenomena such as ``aging'' in spin glasses
\cite{BouchaudEtAl}.
For some models, the limit equations were proved rigorously
after a tour de force in stochastic processes theory \cite{DemboEtAl}.
Theorem \ref{thm:unbiased_convergence} presents remarkable
structural similarities with these results. It suggests that this
type of approach might be useful in analyzing a large array of
stochastic dynamics on graphs.

Over the last couple of years, the cavity method has also
been successfully applied in nonrigorous studies
of quantum spin models on trees \cite{Quantum1,Quantum2}, a topic of interest
in condensed matter physics.
While this paper does treat quantum spin models,
there are strong mathematical similarities
between the dynamic cavity method adopted here
and the cavity analysis of \cite{Quantum1,Quantum2}.
It would be interesting to adapt the rigorous methods developed here
to the analysis of quantum models.

The majority dynamics can be viewed as a simple model of iterative
``social learning'' (see, e.g., \cite{BalaGoyal98,DeGroot74}). In this
context, the initial spin $\sigma_i(0)$ at node $i$ can be interpreted
as a noisy version of some underlying ``state of the world'' that agents
are attempting to learn from each other. Tools developed in this work,
such as the dynamic cavity method, should be applicable to other models
of social learning (cf. Remark \ref{rem:dynamic_cavity_general_update}).

The majority process and similar
models have been studied in the economic theory literature
\cite{Morris,Kleinberg}
within the general theme of ``learning in games.'' In this context,
each node corresponds to a strategic agent and each of the two
states to a different strategy. The dynamics studied
in this paper is just a best-response dynamics, whereby each agent plays
a symmetric coordination game with each of its neighbors. It
would be interesting to apply the present methodology to more
general game-theoretic models.

%
%*****************************************************************
%

%
%********************************************************************
%
%s4 ###
\section{\texorpdfstring{The dynamic cavity method and proof of Theorem
\protect\ref{thm:biased_convergence}}{The dynamic cavity method and proof of
Theorem 2.8}}
\label{sec:ProofUB}

%In this Section we prove the upper bound
%in Theorem \ref{thm:UpperBound}, as well as the
%convergence to the cavity process in Theorem
%Indeed, we will prove the following stronger result that implies both
%these theorems.
%%
%

The proof of Theorem \ref{thm:biased_convergence} is organized as follows.
Section \ref{sec:notations} introduces some notation.
We start by proving some basic properties of the cavity process
in Section \ref{sec:CavityProcess}.
%We then prove an exact (albeit
%quite complicated) recursive characterization of the process
%$\{\sigma_i(t)\}_{t\ge0}$ for $i\in\V$ in Section
We state a local central limit theorem for lattice random variables in
Section \ref{sec:sum_approx}.
A proof of Theorem \ref{thm:unbiased_convergence} follows in Section
\ref{sec:unbiased_convergence}.
%It is convenient to prove the unbiased case separately, since it is
%technically simpler.
Finally, in Section \ref{sec:biased_convergence}, we derive a delicate
relationship
between the biased and unbiased processes and prove Theorem \ref
{thm:biased_convergence}.

%s4.1 ###
\subsection{Notation}
\label{sec:notations}
We use the following notation throughout this section.
For a sequence $a(0), a(1), a(2),\ldots,$ and given $t\ge s$, we let
$a_s^t\equiv(a(s),a(s+1)$, $\ldots,a(t))$. Further, given the correlation
and response functions $C$ and $R$, and an integer $T\ge0$,
we define the $(T+1)\times(T+1)$ matrices $C_T = \{C(t,s)\}_{t,s\le T}$
and $R_T = \{R(t,s)\}_{s<t\le T}$.

Given $m\in\reals^d$ and $\Sigma\in\reals^{d\times d}$, we
let $\phi_{m,\Sigma}(x)$ be the
density at $x \in\reals^d$ of a Gaussian random variable with mean
$\mu$ and covariance $\Sigma$.
Finally, if $\A\in\reals^d$ is a rectangle,
$\A= [a_1,b_1]\times[a_2,b_2]\times\cdots\times[a_d,b_d]$
(with $a_i\le b_i$), we let
%
%e21 ###
%
\begin{equation}
\Phi_{m,\Sigma}(\A) \equiv\int_{\A} \phi_{m,\Sigma}(x) \prod
_{i\in[d]\dvtx b_i>a_i}
\de x_i .
\end{equation}
Notice that those coordinates such that $a_i=b_i$ are not integrated over.
For a partition $\{1,\ldots,d\}=\I_0\cup\I_+\cup\I_-$
and a vector $a\in\Z^d$, define
%
%e23 ###
%e22 ###
%
\begin{eqnarray}\quad
\label{eq:DefA2}
A(a,\I) &\equiv& \{z \in\Z^d \dvtx z_i=a_i \ \forall i \in
\I_0,
z_i\ge a_i \ \forall i \in\I_+, z_i\le a_i \ \forall i \in
\I_-\} ,\\
\label{eq:DefA1}
\A_{\infty}(\I) &\equiv& \{z \in\reals^d \dvtx z_i=0 \ \forall i
\in\I_0,
z_i\ge0 \ \forall i \in\I_+, z_i\le0 \ \forall i \in\I
_-\}
.
\end{eqnarray}
%
%************************************************************
%
%s4.2 ###
\subsection{The cavity process}
\label{sec:CavityProcess}

We start by checking that consistent $R$, $C$ are uniquely defined,
thus justifying the definition of the cavity process.
\begin{lemma}\label{lemma:Existence}
There exist unique consistent $C$, $R$ (cf. Definition \ref{def:consistency}).
\end{lemma}
\begin{pf}
For $T \geq0 $, let $C_T, R_T$ denote the restriction of $C, R$ to
index values of at most $T$. Define
\[
\Ev(T) \equiv \mbox{There exists unique
$C_T, R_T$, such that (\ref{eq:cavity_def1}) is satisfied for all
$s,t\le T$}
\]
and
\[
\mbox{(\ref{eq:cavity_def2}) is satisfied for all $s<t\le T$}.
\]
We want to show that $\Ev(T)$ holds for all $T \geq0$. We proceed by
induction.
%The statement is equivalent
%to the following: For every $T\ge0$, there exists unique
%$C_T, R_T$, such that (\ref{eq:cavity_def1}) is satisfied for all
%$s,t\le T$
%and (\ref{eq:cavity_def2}) is satisfied for all $s<t\le T$. We
%will abbreviate this
%by saying simply that (\ref{eq:cavity_def1}) and (
%are satisfied up to $T$.
%We will prove this statement by induction over $T$.
%More precisely, we consider the following statements
%%
%%
%satisfying (\ref{eq:cavity_def1}), (\ref{eq:cavity_def2}) up to
%$T$
%exists, with $C_T$ positive semidefinite',}\\
%%
%satisfying
%(\ref{eq:cavity_def1}), (\ref{eq:cavity_def2}) up to $T$.'}
%%
%%
%We further let $\Ev(T) \equiv\Ev_{\rm E}(T)\wedge\Ev_{U}(T)$.
%We will prove by induction that $\Ev(T)$ holds for every $T$.
Clearly, $\Ev(0)$ holds with $C(0,0)=1$.

Suppose $\Ev(T)$ holds. Denote by $C_T$
and $R_T$ the corresponding consistent covariance and
response function, that exist and are unique by hypothesis.
We construct consistent $C_{T+1}$ and $R_{T+1}$ by suitably
extending $C_T$ and $R_T$.
Let $(\tau(0), \tau(1), \ldots,
\tau(T+1))$ be a sample path of the uniquely defined
effective process with parameters $C_T$, $R_T$ and $h$ as per
(\ref{eq:asymp_dyn_def}).
Define $C(s,T+1)=C(T+1,s)$ for all $s\le T+1$
by (\ref{eq:cavity_def2}) with $t=T+1$.
Define $R(T+1,s)$ for all $s\le T$ by (\ref{eq:cavity_def1})
with $t=T+1$. The resulting $C_{T+1}, R_{T+1}$ are clearly consistent
up to $T+1$.
Notice that $C_{T+1}$ is positive semidefinite by construction.
We now need to argue that there is no other consistent $C_{T+1}, R_{T+1}$.
But this is clearly true since the restriction up to time $T$ \textit
{must} match
$C_T, R_T$ for consistency up to $T$, and the extension to $T+1$ defined
above is the only way
to satisfy (\ref{eq:cavity_def1}) and (\ref{eq:cavity_def2})
with $t=T+1$.
Thus, $\Ev(T+1)$ holds.
Induction completes the proof.
\end{pf}
\begin{lemma}\label{lemma:cavity_non_degenerate}
Let\vspace*{1pt} $\{C(t,s)\}_{t,s\ge0}$ be the correlation function of the cavity
process. For any $T\ge0$ the matrix $C_T$ is strictly positive definite,
and $\prob(\tau^{T}=\traj^{T})>0$ for each $\traj^{T}
\in\{\pm1\}^{T+1}$.
\end{lemma}
\begin{pf}
As a preliminary remark notice that, by
Lemma \ref{lemma:Existence}, $R(t,s)$ is well defined for all $s<t$. Moreover,
it is easy to see that it is always finite.

We prove the lemma by induction. Clearly, $C_0$ is positive definite and
$\prob(\tau(0)=\pm1) = \frac{1}{2} > 0$. Suppose, $C_T$ is positive
definite.
Now, from the definition of the cavity process, we have
%
%e24 ###
%
\begin{equation}\label{eq:traj_prob}
\prob(\tau^{T+1}=\traj^{T+1})= \tfrac{1}{2} \Phi_{\mu(\traj
^T),C_T}(\A_\infty(\I_C(\traj))),
\end{equation}
where $\I_C(\traj^T)$ is the partition of $\{1,2,\ldots,T\}$
defined as follows:
%
%e26 ###
%e25 ###
%
\begin{eqnarray}
\I_C(\traj)&\equiv& (\varnothing,\I_{C,+},\I_{C,-}),\qquad
\I_{C,+}= \{i\dvtx\traj(i+1)=+1\},\nonumber\\[-8pt]\\[-8pt]
\I_{C,-}&=& \{i\dvtx\traj (i+1)=-1\},\nonumber\\
\label{eq:ic_mu_defn}
\mu(\traj^T)&\equiv&(\mu_0(\traj^T), \ldots, \mu_t(\traj^T))
\qquad\mbox{with }
\mu_r(\traj^T)\equiv\sum_{s=0}^{r-1}R(r,s)\traj(s) .
\end{eqnarray}
Since $C_T$ is positive definite, we have $\Phi_{\mu(\traj
^T),C_T}(x) > 0
\ \forall x \in\reals^{T+1}$, whence
$\prob(\tau^{T+1}=\traj^{T+1})>0$ for all
$\traj^{T+1} \in\{-1, +1\}^{T+2}$.
Notice that $C_{T+1}$ is positive semidefinite
by the definition of cavity process.
If $C_{T+1}$ is not strictly positive definite, there must
be a linear combination of $(\tau(0),\ldots,\tau(T+1))$ that
is equal to $0$ with probability $1$.
Since the distribution of $\tau^{T+1}$ gives positive weight to
each possible configuration, there must exist a nontrivial linear function
in $\reals^{T+2}$ that vanishes on very point of $\{\pm1\}^{T+2}$,
which is impossible. This proves that $C_{T+1}$ is strictly positive definite.
\end{pf}

The above proof provides, in fact, a procedure to
determine $C(t,s)$ and $R(t,s)$ by recursion over $t$.
However, while the recursion for $C$ consists just of a multi-dimensional
integration over the Gaussian variables $\{\eta(t)\}$, the recursion for
$R$ [cf. (\ref{eq:cavity_def2})] is a priori more complicated
since it involves differentiation with respect to~$h$. The next lemma
provides more explicit expressions.
\begin{lemma}\label{lemma:RCexp}
The correlation and response functions $C$ and $R$ of the cavity
process are determined by the following recursion:
%
%e28 ###
%e27 ###
%
\begin{eqnarray}\qquad
\label{eq:CorrRecursion}
C(t+1,s) & = & \frac{1}{2} \sum_{\traj^{t+1} \in\{\pm1\}^{(t+2)}}
\traj(t+1) \traj(s)
\Phi_{\mu(\traj^t),C_t}(A_\infty(\I_C(\traj)))\nonumber\\[-8pt]\\[-8pt]
&&\eqntext{\forall0 \le s \le t,}\\
\label{eq:RespRecursion}
R(t+1,s)
&=& \frac{1}{2} \sum_{\traj^{t+1} \in\{\pm1\}^{(t+2)}}
\traj(t+1) \traj(s+1)
\Phi_{\mu(\traj^t),C_t}(A_\infty(\I_R(\traj,s)))
\nonumber\\[-8pt]\\[-8pt]
&&\eqntext{\forall0 \le s \le t,}
\end{eqnarray}
with boundary condition $R(t,s)= 0$ for $t\le s$, $C(t,t) = 1$ and
$C(s,t) = C(t,s)$. Here, $\I_C$ and
$\I_R$ are partitions of $\T=\{0,1,\ldots,t\}$
of the form $\I\equiv(\I_0,\I_+, \I_-)$, with $\I_C$ and $\mu$
defined as per
(\ref{eq:ic_mu_defn}) with $T=t$ and $\I_R$ is defined by
%
%e29 ###
%
\begin{eqnarray}
\I_R(\traj,s)&\equiv&(\{s\},\I_{R,+},\I_{R,-}),\qquad
\I_{R,+}= \{i\dvtx\traj(i+1)=+1\}\setminus\{s\},\nonumber\\[-8pt]\\[-8pt]
\I_{R,-}&=& \{
i\dvtx\traj(i+1)=-1\}\setminus\{s\} .\nonumber
\end{eqnarray}
\end{lemma}

We provide a proof of this lemma in Appendix \ref{app:RCexp}.

Equation (\ref{eq:RespRecursion}) yields in particular
%
%e30 ###
%
\begin{equation}\label{eq:R_t_tplus1}
R(t+1,t) = \sum_{\traj^{t}\in\{\pm1\}^{t+1}} \Phi_{\mu(\traj
^t),C_t}(A_\infty(\I_R(\traj,t))) .
\end{equation}
Note that $R(t+1,t)>0 \ \forall t \ge0$, since it is a sum of
positive terms.
These facts will be used later in Section
\ref{sec:biased_convergence}.\vadjust{\goodbreak}

%The values of $R$ and $C$ evaluated for small values of $s,t$ are as follows
%(see Tables \ref{table2} and \ref{table3}).

Tables \ref{table2} and \ref{table3} contain computed
values of $C$ and $R$, respectively, for small
values of $s$, $t$.

%$C(t,s)$ (Row for each $t$, column for each $s$.)
%1&1&1\\
%2&2&2\\
%1&1&1\\
%2&2&2\\

%t2 ###
%
\begin{table}
\tablewidth=235pt
\caption{Computed $C(t,s)$ values}\label{table2}
\begin{tabular*}{\tablewidth}{@{\extracolsep{\fill}}l
@{\hspace*{-5pt}}c@{\hspace*{-5pt}}d{1.4} d{1.4} c c@{}}
\hline
& &\multicolumn{4}{c@{}}{\hspace*{-6pt}$\bolds{s}$}\\
[-4pt]
& &\multicolumn{4}{c@{}}{\hspace*{-6pt}\hrulefill}\\
$\bolds{t}$ & & \multicolumn{1}{c}{\hspace*{-6pt}\textbf{0}} & \multicolumn{1}{c}{\textbf{1}} & \multicolumn{1}{c}{\textbf{2}} & \multicolumn{1}{c@{}}{\textbf{3}}\\
\hline
$ 0 $
& \multirow{4}{2pt}{\rule{0.5pt}{41pt}}
& 1\\
1 & & 0 & 1 & & \\
2 & & 0.5751 & 0 & 1 \\
3 & & 0 & 0.7600 & 0 & 1\\
\hline
\end{tabular*}
\end{table}

%t3
\begin{table}[b]
\tablewidth=235pt
\caption{Computed $R(t,s)$ values}\label{table3}
\begin{tabular*}{\tablewidth}{@{\extracolsep{\fill}}l
@{\hspace*{-5pt}}c@{\hspace*{-5pt}}d{1.4} d{1.4} d{1.4} c@{}}
\hline
& &\multicolumn{4}{c@{}}{\hspace*{-6pt}$\bolds{s}$}\\
[-4pt]
& &\multicolumn{4}{c@{}}{\hspace*{-6pt}\hrulefill}\\
$\bolds{t}$ & & \multicolumn{1}{c}{\hspace*{-6pt}\textbf{0}} & \multicolumn{1}{c}{\textbf{1}} & \multicolumn{1}{c}{\textbf{2}} & \multicolumn{1}{c@{}}{\textbf{3}}\\
\hline
1
& \multirow{4}{2pt}{\rule{0.5pt}{41pt}}
& 0.7979\\
2 & & 0 & 0.5804 \\
3 & & 0.4164 & 0 & 0.4607 \\
4 & & 0 & 0.2920 & 0 & 0.3950\\
\hline
\end{tabular*}
\end{table}

Note how $C(t,s)=0$ when $t$ and $s$ have different parity, and
$R(t,s)=0$ when $t$ and $s$ have the same parity. This is a simple
consequence of the fact that the dynamics is ``bipartite.''
This also allows us to reduce the dimensionality of integrals in
(\ref{eq:CorrRecursion}) and (\ref{eq:RespRecursion}),
making numerical computations easier.

%
%*******************************************************************
%

%
%*******************************************************************
%
%s4.3 ###
\subsection{A central limit theorem}
\label{sec:sum_approx}

We will use repeatedly the following local
central limit theorem for lattice random variables.
\begin{theorem}\label{thm:sum_approx}
For any $B>0$ and $d \in\mathbb{N}$, there exists a finite constant
$\const=\const(B,d)$ such that the
following is true.
Let $X_1,X_2, \ldots, X_N$, be i.i.d. random vectors with $X_1 \in\{
+1,-1\}^d$
and
\[
\|\Ex X_1\| \le\frac{B}{\sqrt{N}} ,\qquad
\min_{s\in\{+1,-1\}^d}\prob(X_1 = s) \ge\frac{1}{B} .
\]

Let $p_N$ be the distribution of $S_N=\sum_{i=1}^N X_i$.
For a partition $\{1,\ldots,d\}=\I_0\cup\I_+\cup\I_-$
and a vector $a\in\Z^d$, with $\|a\|_{\infty}\le B \log N$, define
$A(a,\I)$, $\A_{\infty}(\I)$ as in (\ref{eq:DefA1}) and
(\ref{eq:DefA2}).

Assume the coordinates $a_i$ to have the same parity as $N$.
We then have
%
%e31 ###
%
\begin{eqnarray}\label{eq:sum_approx}\quad
\sum_{y \in A(a,\I) } p_N(y) &=&
\frac{2^{|\I_0|}}{N^{|\I_0|/2}} \Phi_{\sqrt{N}\Ex X_1,
\Cov(X_1)}(\A_{\infty}(\I)) \bigl(1+\mathsf{Err}(a,\I
,N)\bigr) ,
\\
|\mathsf{Err}(a,\I,N)|&\le&\const(B,d) N^{-1/(2|\I_0|+2)}.\nonumber
\end{eqnarray}
\end{theorem}

A simple proof of this result can be obtained using
the Bernoulli decomposition method of \cite{McDonald1,McDonald2}
and is reported in Appendix \ref{app:CLT}. Indeed Appendix \ref{app:CLT}
proves a slightly stronger result.
%
%*******************************************************************
%
%s4.4 ###
\subsection{\texorpdfstring{Unbiased initialization: Proof of Theorem
\protect\ref{thm:unbiased_convergence}}{Unbiased initialization: Proof of Theorem
2.7}}
\label{sec:unbiased_convergence}

Throughout this section, we consider the case of unbiased
initialization, that is, $\theta=0$.

Before passing to the details of the actual proof,
we attempt to provide some intuition.

%s4.4.1 ###
\subsubsection{\texorpdfstring{Theorem
\protect\ref{thm:unbiased_convergence}: Basic intuition}{Theorem
2.7: Basic intuition}}

The central idea consists in studying the dynamics
at the root of the rooted tree $\G_{\rroot}=(\V_{\rroot},\Ed
_{\rroot})$
with updates modified according to (\ref{eq:ModifiedUpdateRoot}).
The dynamics\vspace*{2pt} at the root is indeed completely characterized by the
recursion (\ref{eq:Recursion}). Let $y^T= \sum_{i=1}^{k-1} \sigma_i^T$,
and write $\prob( y^T \midd \sigma_\rroot^{T})$ for its
distribution under
the product measure $\prod_{i=1}^{k-1}\prob(\sigma_i^{T}\midd
\sigma_\rroot^{T})$. Since $\sigma_\rroot^T$ only depends on its
neighbors through their sum $y^T$, all that matters is in fact the
distribution $\prob( y^T \midd \sigma_\rroot^{T})$. A further simplification
arises in the large $k$ limit because we can apply the central
limit theorem to show that $y^T$ converges to a Gaussian random variable.

Two complications however arise:
(i) The mean and variance of this Gaussian depend in an a priori
arbitrary way on $\sigma_\rroot^{T}$ itself; (ii) In order to
track this dependence, it is necessary to establish a central limit theorem
for $y^T$.

In order to illustrate these points, it is useful to
follow the first few steps of the dynamics.
First take $T=0$. We know that $\prob(\sigma_i(0) \midd \sigma
_\rroot
(0)) = \prob_0(\sigma_i(0))=1/2$. Thus, using (\ref{eq:Recursion})
for $T=0$, we get
\[
\prob(\sigma_\rroot^{1}\midd u^{1} )= \frac{1}{2^k}
\sum_{\sigma_1(0)\cdots\sigma_{k-1}(0)}
\K_{u(0)}(\sigma_{\rroot}(1)|\sigma_{\partial\rroot}(0))
.
\]
This expression can be estimated by approximating
$\prob(y(0) \midd u(0))$ with the Gaussian distribution
$\mathcal{N}(0, k-1)$. In particular, using the expression
(\ref{eq:Kdef}) for $\K_{u(0)}(\sigma_{\rroot}(1)|\sigma
_{\partial\rroot}(0))$
we get, for $(k-1)$ even and $u(0)\in\{+1,-1\}$
\begin{eqnarray*}
\E(\sigma_\rroot(1)\midd u^{1}) &=& \E
\sign\Biggl(\sum_{i=1}^{k-1}\sigma_i(0)+u(0)\Biggr)\\
&=& \E
\sign\bigl(y(0)+u(0)\bigr)
= u(0) \prob\bigl(y(0)=0\bigr)\\
&\approx& u(0) \prob\bigl(
\sqrt{k-1} Z\in[-1, 1]\bigr)\approx\sqrt{\frac{2}{\pi k}}
u(0) ,
\end{eqnarray*}
where $Z$ denotes a unit normal random variable.
Using the fact that $\sigma_\rroot(0)$ is independent
of $\sigma_\rroot(1)$ by the bipartite nature of the dynamics,
we obtain the estimate
\[
\prob(\sigma_\rroot^{1}\midd u^{1})\approx\frac{1}{4}
\biggl(1+ \frac{R(1,0)u(0)}{\sqrt{k}} \sigma_\rroot(1)\biggr) ,
\]
where $R(1,0)= \sqrt{2/\pi}$
as per (\ref{eq:R_t_tplus1}). %Notice that this coincides with the
%prediction of Theorem \ref{thm:unbiased_convergence}.
It follows that $\E_{\prob(\cdot\midd u^1)}[\sigma_\rroot(1)]
\approx
R(1,0)\times\break u(0)k^{-1/2}$. Also,
$\E_{\prob( \cdot\midd u^1)}[\sigma_\rroot(1)\sigma_\rroot(0)]=
C(1,0)=0$ and
$\E_{\prob(\cdot\midd u^1)}[\sigma_\rroot(1)\sigma_\rroot(1)]=
C(1,1)=1$.
It follows that $\prob(y^1 \midd u^1)$ has a Gaussian approximation
$ \mathcal{N}( \sqrt{k} \mu(\sigma_\rroot^1)$, $k C_1)$,
where
$\mu(\sigma_\rroot)^1 = ( {0}\enskip{R(1,0) \sigma_\rroot(0)} )$.
Note how the mean and standard deviation of $y^1$ are each of the
same order $\Theta(\sqrt{k})$.

When passing to $T=1$ in (\ref{eq:Recursion}), we can make this
normal approximation for the environment $y^1=
\sum_{i=1}^{k-1} \sigma_i^1$, up to time $1$.
We hence obtain a stochastic description of the root spin process up to
time $2$. Essentially the same argument is extended to any time $T$
by induction, as is explained in detail in the following.
%
%*****************************************************
%
%s4.4.2 ###
\subsubsection{\texorpdfstring{Theorem
\protect\ref{thm:unbiased_convergence}: The actual proof}{Theorem
2.7: The actual proof}}

The next lemma rigorizes the above intuition and extends it to all
times $T$ by induction.
\begin{lemma}\label{lemma:RCconvergence}
Let $T\ge0$ and $u^T$ with $u(t)\in\{+1,-1\}$ be given. Assume
$\sigma_\rroot^T$ to be distributed according to
$\prob( \cdot\midd u^T)$. Then, as $k\to\infty$, we have
%
%e32 ###
%
\begin{eqnarray}\label{eq:CovarianceMean}
|\E\{\sigma_{\rroot}(t)\sigma_{\rroot}(s)\} -C(t,s)|
&=& o(1) ,\nonumber\\[-8pt]\\[-8pt]
\Biggl|
\E\sigma_{\rroot}(t)-\frac{1}{\sqrt{k}}\sum_{s=0}^{t-1}R(t,s)
u(s)\Biggr| &=&
o(k^{-1/2})\nonumber
\end{eqnarray}
for all $t,s \leq T$.
Further, for any $u^T$, $\sigma_\rroot^T \stackrel{d}{\rightarrow}
\tau^T$, with $\tau^T$
distributed according to the cavity process.
\end{lemma}
\begin{pf}
The proof is by induction on the number of steps $T$. Obviously, the thesis
holds for $T=0$.

Assume\vspace*{1pt} that it holds up to time $T$. Consider the exact
recursion equation (\ref{eq:Recursion})
and fix a sequence $\sigma_\rroot(0),\ldots,\sigma_\rroot(T+1)$.
Under the measure $\prod_{i=1}^{k-1}\prob(\sigma_i^{T}\midd
\sigma_\rroot^{T})$, the vectors $\sigma_1^T,\ldots,\sigma
_{k-1}^T$ are independent and identically distributed.
Further, by the induction hypothesis
\[
\E\sigma_1(t) = \frac{1}{\sqrt{k}}\sum_{s=0}^{t-1}R(t,s)\sigma
_{\rroot}(s)+
o(k^{-1/2}) ,\qquad
\E\{\sigma_1(t)\sigma_1(s) \} = C(t,s)+o(1) .
\]
By the central limit theorem
$\{\frac{1}{\sqrt{k}}\sum_{i=1}^k \sigma_i(t)\}_{0\le t\le T}$
converge in distribution to
%
%e33 ###
%
\begin{equation}
\Biggl\{\eta(t) + \sum_{s=0}^{t-1}R(t,s) \sigma_\rroot(s)\Biggr\}
_{0\le t\le T} ,
\end{equation}
where $\{\eta(t)\}_{0\le t\le T}$ is a centered Gaussian vector with
covariance $\E\{\eta(t)\eta(s)\} = C(t,s)$. Since the product of indicator
functions in (\ref{eq:Recursion}) is a bounded function
of the vector $\{\frac{1}{\sqrt{k}}\sum_{i=1}^k \sigma_i(t)\}_{0\le
t\le T}$,
and the normal distribution is everywhere continuous, we have
%
%e34 ###
%
\begin{eqnarray}\label{eq:UnbiasedConv}\quad
&&\lim_{k\to\infty}
\prob(\sigma_\rroot^{T+1}\midd u^{T+1})\nonumber\\[-8pt]\\[-8pt]
&&\qquad= \prob_0
(\sigma_\rroot(0))
\E_{\eta}\Biggl[
\prod_{t=0}^T \ind\Biggl\{\sigma_\rroot(t+1) = \sign\Biggl(
\eta(t)+\sum_{s=0}^{t-1}R(t,s)\sigma_\rroot(s)\Biggr)\Biggr\}
\Biggr] ,\nonumber
\end{eqnarray}
that is, $\sigma_\rroot^{T+1}$ converges in distribution to
the first $T+1$ steps of the cavity process.
This implies the first equation in (\ref{eq:CovarianceMean}).
It is therefore sufficient to prove the second equation
in (\ref{eq:CovarianceMean}), for $t=T+1$.

To get the estimate of the mean, we use again (\ref{eq:Recursion}),
and consider the distribution $\prob(\sigma_\rroot^{T+1}\midd
0^{T+1})$
whereby the root perturbation is set to $0$.
This satisfies the recursion equation (\ref{eq:Recursion}), with $u(t)
= 0$:
%
%e35 ###
%
\begin{eqnarray}\label{eq:Recursion0}
&&\prob(\sigma_\rroot^{T+1}\midd 0^{T+1}) \nonumber\\[-8pt]\\[-8pt]
&&\qquad= \prob_0(\sigma
_\rroot(0))
\sum_{\sigma_1^{T}\cdots\sigma_{k-1}^{T}}
\prod_{t=0}^T
\K_{0}\bigl(\sigma_{\rroot}(t+1)|\sigma_{\partial\rroot}(t)\bigr)
\prod_{i=1}^{k-1}\prob(\sigma_i^{T}\midd \sigma_\rroot^{T}
) .\nonumber
\end{eqnarray}
Since $|u(t)|\le1$,
$K_{u(t)}(\cdots) = K_0(\cdots)$ for all values of $t$, except those
in which $\sum_{i=1}^{k-1}\sigma_i(t)\in\{+1,0,-1\}$.
Let $\I_0 =\{ t \dvtx|{\sum}_{i=1}^{k-1}\sigma_i(t)|\le1\}$.
Further, irrespective of $u(t)$,
$\K_{u(t)}(\sigma_{\rroot}(t+1)|\sigma_{\partial\rroot
}(t))$
is nonvanishing only if
$\sigma_\rroot(t+1) \sum_{i=1}^{k-1}\sigma_i(t) \ge-1$.
By taking the difference of (\ref{eq:Recursion}) and (\ref
{eq:Recursion0}),
we get
%
%e36 ###
%
\begin{eqnarray}\label{eq:BigTerm}
&&
\prob(\sigma_\rroot^{T+1}\midd u^{T+1})-\prob(\sigma
_\rroot^{T+1}\midd 0^{T+1}) \nonumber\\
&&\qquad= \prob_0(\sigma_{\rroot}(0))
\sum_{\sigma_1^{T}\cdots\sigma_{k-1}^{T}}
\prod_{i=1}^{k-1}\prob(\sigma_i^{T}\midd \sigma_\rroot^{T}
)\prod_{t=0}^T
\ind\Biggl\{\sigma_\rroot(t+1) \sum_{i=1}^{k-1}\sigma_i(t) \ge-1
\Biggr\}
\nonumber\\[-8pt]\\[-8pt]
&&\qquad\quad\hphantom{\prob_0(\sigma_{\rroot}(0))
\sum_{\sigma_1^{T}\cdots\sigma_{k-1}^{T}}}
{}\times
\biggl(\prod_{t \in\I_0} \K_{u(t)}\bigl(\sigma_{\rroot
}(t+1)|\sigma_{\partial\rroot}(t)\bigr) \nonumber\\
&&\qquad\quad\hspace*{98.2pt}{} - \prod_{t \in\I_0} \K
_{0}\bigl(\sigma_{\rroot}(t+1)|\sigma_{\partial\rroot}(t)\bigr)
\biggr) .\nonumber
\end{eqnarray}

Let $y^T= \sum_{i=1}^{k-1} \sigma_i^T$,
and write $\prob( y^T \midd \sigma_\rroot^{T})$ for its
distribution under
the product measure $\prod_{i=1}^{k-1}\prob(\sigma_i^{T}\midd
\sigma_\rroot^{T})$. Further, let
\begin{eqnarray*}
\I_+&\equiv&\{t\dvtx t<T, t \notin\I_0, \sigma_\rroot(t+1)=+1 \} ,\\
\I_-&=&\{t\dvtx t<T, t \notin\I_0, \sigma_\rroot(t+1)=-1 \} .
\end{eqnarray*}
Then the above expression takes the form
\begin{eqnarray*}
&&\prob(\sigma_\rroot^{T+1}\midd u^{T+1})-\prob(\sigma
_\rroot^{T+1}\midd 0^{T+1})\\
&&\qquad=\prob_0(\sigma_\rroot(0))
\sum_{y^{T}}
\prob( y^T \midd (\sigma_\rroot)^{T}) \prod_{t\in\I_+}
\ind\{ y(t) > 1 \} \\
&&\qquad\quad\hphantom{\prob_0(\sigma_\rroot(0))
\sum_{y^{T}}}
{}\times\prod_{t\in\I_-}
\ind\{ y(t) < -1 \} f_{\I_0}(\{y(t)\}_{t\in\I_0}) ,
\end{eqnarray*}
where we defined $f(\{y(t)\}_{t\in\I_0})$ to be the
term in parentheses in (\ref{eq:BigTerm}).

Now, we can apply Theorem \ref{thm:sum_approx} for every possible
$\I_0$, by letting $X_i=\sigma_i^{T}$, so that $d=T+1$, and
$N=k-1$. Note that our induction hypothesis equation (\ref{eq:CovarianceMean})
on the mean implies that
%
%e37 ###
%
\begin{equation}\label{eq:sigma_i_mean}
\Biggl|
\E\sigma_i(t)-\frac{1}{\sqrt{k}}\sum_{s=0}^{t-1}R(t,s) \sigma
_\rroot(s)\Biggr| =
o(k^{-1/2})
\end{equation}
for all $t\le T$.
In particular $\|\E X_1\| \le B/\sqrt{k}$ as needed.
Further, by Lemma \ref{lemma:cavity_non_degenerate}, our induction
hypothesis equation (\ref{eq:CovarianceMean}), and the convergence
result (\ref{eq:UnbiasedConv}), we have
$\min_s\prob\{X_1=s\}\ge1/B$ for all $k$ large enough.

Now\vspace*{2pt} $f_{\I_0}(\{y(t)\}_{t\in\I_0})=0$ for $\I_0=\varnothing$.
From Theorem \ref{thm:sum_approx}, the contribution for any $\I_0 \ne
\varnothing$ is
$\Theta( k^{-|\I_0|/2})$. It follows\vspace*{1pt} that the dominating terms
correspond to $\I_0=\{t_0\}$.
If we let $\mu'(\sigma_\rroot)=\sqrt{k-1}\Ex[(\sigma_1)^T]$,
$V(\sigma_\rroot)=\Cov((\sigma_1)^T)$, then
%
%e38 ###
%
\begin{eqnarray}\label{eq:AlmostFinalLemma}
&&\prob(\sigma_\rroot^{T+1}\midd u^{T+1})-\prob(\sigma
_\rroot^{T+1}\midd 0^{T+1})\nonumber\\
&&\qquad =
\frac{2\prob_0(\sigma_\rroot(0))}{\sqrt{k-1}}
\sum_{t_0=0}^T
\Phi_{\mu'(\sigma_{\rroot}), V(\sigma_{\rroot})} (\A_{\infty
}(\I)) \nonumber\\[-8pt]\\[-8pt]
&&\qquad\quad\hphantom{\frac{2\prob_0(\sigma_\rroot(0))}{\sqrt{k-1}}
\sum_{t_0=0}^T }
{}\times\sum_{|y(t_0)|\le1}
\bigl\{ \K_{u(t_0)}\bigl(\sigma_{\rroot}(t_0+1)|y(t_0)\bigr)\nonumber\\
&&\qquad\quad\hspace*{120pt}{} -
\K_{0}\bigl(\sigma_{\rroot}(t_0+1)|y(t_0)\bigr) \bigr\}
\bigl(1+o(1)\bigr) ,\nonumber
\end{eqnarray}
where, with an abuse of notation,
we wrote $\K_{\cdot}(\sigma_{\rroot}(t_0+1)|y(t_0))$
for $\K_{\cdot}(\sigma_{\rroot}(t_0+1)|\sigma_{\partial
\rroot}(t_0))$
when $\sum_{i=1}^{k-1}\sigma_i(t_0) = y(t_0)$.
Further, the rectangle $\A_{\infty}(\I)$
is defined as in Theorem \ref{thm:sum_approx}.

If $k$ is odd, then the only term in the above sum is $y(t_0) =0$. An
simple explicit calculation shows that
\[
\K_{u(t_0)}\bigl(\sigma_{\rroot}(t_0+1)|y(t_0)=0\bigr) -
\K_{0}\bigl(\sigma_{\rroot}(t_0+1)|y(t_0)=0\bigr) = \tfrac
{1}{2}u(t_0)\sigma_{\rroot}(t_0+1) .
\]
If $k$ is even, two terms contribute to the sum:
$y(t_0) =+1$ and $y(t_0)=-1$, with
\begin{eqnarray*}
&&\K_{u(t_0)}\bigl(\sigma_{\rroot}(t_0+1)|y(t_0)=+1\bigr) -
\K_{0}\bigl(\sigma_{\rroot}(t_0+1)|y(t_0)=+1\bigr)\\
&&\qquad= -\tfrac
{1}{2}\sigma_{\rroot}(t_0+1)
\ind\bigl(u(t_0)=-1\bigr) ,
\end{eqnarray*}
\begin{eqnarray*}
&&\K_{u(t_0)}\bigl(\sigma_{\rroot}(t_0+1)|y(t_0)=-1\bigr) -
\K_{0}\bigl(\sigma_{\rroot}(t_0+1)|y(t_0)=-1\bigr)\\
&&\qquad= -\tfrac
{1}{2}\sigma_{\rroot}(t_0+1)
\ind\bigl(u(t_0)=+1\bigr) .
\end{eqnarray*}
Also, by (\ref{eq:sigma_i_mean}) we have
$\lim_{k\rightarrow\infty} \mu'(\sigma_\rroot) = \mu(\sigma
_\rroot)$
with $\mu( \cdot)$ defined as in (\ref{eq:ic_mu_defn}).
The induction hypothesis equation (\ref{eq:CovarianceMean}) on the covariance
of $\sigma_{\rroot}$ further implies
$\lim_{k\rightarrow\infty} V(\sigma_\rroot) = C_T$.
By the continuity of Gaussian distribution, we get
\[
\lim_{k\to\infty}\Phi_{\mu'(\sigma_{\rroot}), V(\sigma_{\rroot
})} (\A_{\infty}(\I))
= \Phi_{\mu(\sigma_{\rroot}), C_T} (\A_{\infty}(\I)) .
\]
Applying these remarks to (\ref{eq:AlmostFinalLemma}),
and using the fact that $\prob_0(\sigma_\rroot(0))=1/2$, we finally get
%
%e39 ###
%
\begin{eqnarray}\label{eq:prob_diff_simplified}
&&\prob(\sigma_\rroot^{T+1}\midd u^{T+1})-\prob(\sigma
_\rroot^{T+1}\midd 0^{T+1})\nonumber\\[-8pt]\\[-8pt]
&&\qquad = \frac{1}{2\sqrt{k}}
\sum_{t_0=0}^T
\Phi_{ \mu(\sigma_\rroot), C_T}(\A_{\infty}(\I))
u(t_0)\sigma_\rroot(t_0+1)
\bigl(1+o(1)\bigr) .\nonumber
\end{eqnarray}

By symmetry, we have $\Ex_{\prob(\cdot\midd 0^{T+1})}[\sigma
_\rroot
(T+1)] = 0$.
By summing over $\sigma_\rroot^T$ equation (\ref{eq:prob_diff_simplified}),
we get
\begin{eqnarray*}
\hspace*{-4pt}&&\Ex_{\prob(\cdot\midd u^{T+1})}[\sigma_\rroot(T+1)]\\
\hspace*{-4pt}&&\qquad= \frac{1}{\sqrt{k}}
\sum_{t_0=0}^T u(t_0)
\biggl\{
\frac{1}{2} \sum_{\sigma_\rroot^{T+1}} \sigma_\rroot(T+1) \sigma
_\rroot(t_0+1)
\Phi_{\mu(\sigma_\rroot), C_T}(\A_{\infty}(\I))
\biggr\}
\bigl(1+o(1)\bigr) .
\end{eqnarray*}
It is easy to verify that the expression in parentheses
matches the one for $R(T+1,t_0)$ from Lemma \ref{lemma:RCexp}.
Therefore we proved
\[
\Ex_{\prob(\cdot\midd u^{T+1})}[\sigma_{\rroot}(T+1)]=
\frac{1}{\sqrt{k}} \sum_{s \in\I} u(s)R(T+1,s) + o\bigl(1/\sqrt{k}\bigr) ,
\]
which finishes the proof of the induction step.
\end{pf}

In the next section, we will use this estimate to prove
Theorem \ref{thm:biased_convergence}, which in particular implies
Theorem \ref{thm:unbiased_convergence}. Let us notice however that
Theorem \ref{thm:unbiased_convergence} admits a direct proof as a consequence
of the last lemma.
\begin{pf*}{Proof of Theorem \ref{thm:unbiased_convergence}}
As part of Lemma \ref{lemma:RCconvergence}, we have proved that
$\sigma_\rroot^T \stackrel{\mathrm{d}}{\rightarrow}
\tau^T$ for each ``fixed'' trajectory $u^T$;
see (\ref{eq:UnbiasedConv}).
In particular, this holds
for the extreme trajectories
$u_-^T=(-1, -1, \ldots, -1)$ and for $u_+^T=(+1, +1, \ldots, +1)$. By
monotonicity,
the true trajectory of a spin $\sigma_i$ in the regular
tree $\G$ lies between the
trajectories $\sigma_{\rroot,-}^T$ and $\sigma_{\rroot,+}^T$
distributed according to $\prob( \cdot\midd u_+^T )$
and $\prob( \cdot\midd u_-^T )$.
Since both $\sigma_{\rroot,-}^T$ and
$\sigma_{\rroot,+}^T$ converge in distribution to the cavity process
$\tau^T$, the original trajectory $\sigma_i^T$
converges to the cavity process as well.
\end{pf*}

%
%*****************************************************************
%
%s4.5 ###
\subsection{\texorpdfstring{Biased initialization:
Proof of Theorem \protect\ref{thm:biased_convergence}}{Biased initialization:
Proof of Theorem 2.8}}
\label{sec:biased_convergence}

In this subsection, we prove Theorem \ref{thm:biased_convergence}.
The proof is based on Lemmas \ref{lemma:qminusp} and
\ref{lemma:qminusptstar} that capture the asymptotic behavior
of the recursion (\ref{eq:Recursion}) as $k\to\infty$ in two different
regimes.

Throughout this subsection, we adopt a special notation to
simplify calculations. We reserve $\prob(\sigma_\rroot^T \midd
u^T)$
for the family of measures indexed by $u^T$ and introduced
in Section \ref{sec:ExactCavity}, in the case $\prob_0(\sigma
_{\rroot}(0)
=\pm1) = 1/2$.
We use instead $\qprob(\sigma_\rroot^T \midd u^T)$ when the
initialization is
\[
\qprob_0\bigl(\sigma_{\rroot}(0) = \pm1\bigr) = \frac{1}{2}\pm\frac{\bias
_0}{k^{(T_*+1)/2}} ,
\]
that is, when in the initial configuration, the spins of $\G_{\rroot
}$ are
i.i.d. Bernoulli with expectation $\E\sigma_i(0) = 2\bias_0k^{-(T_*+1)/2}$.

Before providing the formal argument, we will describe the
basic intuition.

%s4.5.1 ###
\subsubsection{\texorpdfstring{Theorem
\protect\ref{thm:biased_convergence}: Basic intuition}{Theorem
2.8: Basic intuition}}

The proof relies on a delicate comparison between
the unbiased case (analyzed in the previous section) and
the biased case treated here. An obvious but important fact
is that, if the initialization is unbiased, then the distribution
of $\sigma_i^T$ is exactly symmetric under inversion of all the spins.

How does the evolution with initial bias $\theta= 2 \bias_0 / k^{(T_*+1)/2}$,
$\bias_0 >0 $, differ from the unbiased initialization trajectory?
Consider the coupling between these two processes constructed as follows.
Initialize the two processes by drawing the initial spins at each vertex
according to the optimal coupling between a Bernoulli$(1/2)$ and
a Bernoulli$(\theta)$ random variable. At each subsequent step, the
new spin values are either chosen deterministically, or according
to a fair coin toss (in the case of a tie). In the former case, the
coupling is obvious. In the latter---that is, if a tie occurs in both processes
at the same vertex---the coupling is constructed by using the same coin.
Notice that this coupling is monotone: for $\theta>0$
at each time the biased process dominates the unbiased one.

Denote by $\sigma_i(t)$ the spin at node $i$ in the unbiased dynamics and
by $\sigma'_i(t)$ the spin at node $i$ in the biased dynamics.
Since the coupling is monotone, the two processes
can disagree at node $i$ and time $t$ only if $\sigma_i'(t)=+1$
and $\sigma_i(t)=-1$.

For concreteness, let us consider $T_*=2$, that is, $\theta= 2 \bias_0
/ k^{3/2}$.
At time $t=0$, of the $k$ neighbors of an arbitrary node $i$, on average
$\bias_0 / k^{1/2}$ of them will be different in the two processes.
Further, each neighbor will disagree or not be independent of the others.
Since $\bias_0 / k^{1/2}$ is much smaller than $1$,
most often no neighbors will differ, occasionally $1$ neighbor will
differ and very rarely more than $1$ neighbor will differ. For simplicity,
assume $k$ is odd.
The main event leading to $\sigma_i(t) \neq\sigma'_i(t)$ at $t=1$
will be the following: the two process disagree on one neighbor of $i$
(call this event $\Ev_1$),
and the spins that do not disagree across the two processes
add up to $0$ (call this $\Ev_2$).
Since $\E[\sigma_i(1)]=0$ exactly,
we have
\[
\E[\sigma'_i(1)] = 2 \prob\{\sigma'_i(1)\neq\sigma_i(1)\}\approx
2 \prob(\Ev_1) \prob(\Ev_2) .
\]
The probability of $\Ev_1$ is estimated by the expected number of
disagreements $\prob(\Ev_1)\approx\bias_0 / k^{1/2}$.
The probability of a near tie on the other spins is
instead estimated through a Gaussian approximation
as at the beginning of Section \ref{sec:unbiased_convergence},
$\prob(\Ev_1)\approx\prob(\sqrt{k} Z\in[-1,1]) \approx\sqrt
{2/(\pi k)}$.
One gets therefore
\[
\E[\sigma'_i(1)] \approx2 R(1,0) \frac{\bias_0}{k}\equiv
\frac{2\bias_1}{k} ,
\]
where $\bias_1$ is defined as in the statement of Theorem
\ref{thm:biased_convergence}.

The next steps follow along the same lines.
Consider the neighbors of $i$ at time $t=1$. The two processes
have a probability close to $\bias_1/k$ of disagreeing.
Assuming that disagreements are again roughly independent of each other,
we expect $O (\log k)$ disagreements at most.
This leads, by an argument similar to the above, to
$\E[\sigma'_i(2)] \approx2 R(2,1) \bias_1/\sqrt{k}
\equiv2\bias_2/\sqrt{k}$.

Finally, consider the neighbors of $i$ at time $t=2$.
We expect the two processes to disagree---on average---on $\sqrt
{k}\bias_2$ neighbors.
The two processes still agree on most of the neighbors but the sum of these
spins is also---by central limit theorem---of order
$\sqrt{k}$. This leads to $\E[\sigma'_i(3)] = \Theta(1)$.
Continuing for one more step, we get $\E[\sigma'_i(4)] \approx1$.

The next section will make this calculation rigorous, the main challenge
being of course a precise control of dependencies.
The simple coupling argument above is insufficient to achieve this
goal.
We will instead use once more the exact cavity recursion
(\ref{eq:Recursion}), together with appropriate analytical arguments.
%
%**************************************************
%
%s4.5.2 ###
\subsubsection{\texorpdfstring{Theorem
\protect\ref{thm:biased_convergence}: The actual proof}{Theorem
2.8: The actual proof}}

The above intuition is rigorized by Lemmas \ref{lemma:qminusp} and
\ref{lemma:qminusptstar}. We use the modified dynamics introduced
in Section \ref{sec:ExactCavity}, so that
the exact cavity recursion in Lemma \ref{lemma:ExactCavity} can be used.
\begin{lemma}\label{lemma:qminusp}
For $\sigma^T\in\{\pm1\}^{T+1}$, let $\I_+ = \{t\dvtx\sigma
(t+1)=+1\}$,
$\I_- = \{t\dvtx\sigma(t+1)=-1\}$ and $\I_0 = \{T\}$. Define
%
%e40 ###
%
\begin{equation}\label{eq:ITDef}
I_T(\sigma^{T}) = \Phi_{\mu(\sigma),C_T}(\A_{\infty}(\I
)) ,
\end{equation}
where
$\mu(\sigma)=(\mu_0(\sigma), \ldots, \mu_T(\sigma))$ with
$\mu_r(\sigma)=\sum_{s=0}^{r-1}R(r,s)\sigma(s)$.
Set by definition $I_{-1} = 1$.
Finally, for $ 0 \le T < T_*-1$, define $\bias_{T+1}$ recursively by
%
%e41 ###
%
\begin{equation}\label{eq:omega_recursion}
\bias_{T+1} = R(T+1,T) \bias_{T} .
\end{equation}
Then, for $ 0 \le T < T_*-1$ and for all
$\sigma_\rroot^{T+1}, u^{T+1} \in\{\pm1\}^{T+2}$, we have
%
%e42 ###
%
\begin{eqnarray}\label{eq:qminusp}
&&\qprob(\sigma_\rroot^{T+1}\midd u^{T+1}) -
\prob(\sigma_\rroot^{T+1}\midd u^{T+1})\nonumber\\[-8pt]\\[-8pt]
&&\qquad=\frac{\bias
_T}{k^{(T_*-T)/2}} \sigma_\rroot(T+1) I_T(\sigma_\rroot
^{T}) \bigl(1+o(1)\bigr) .\nonumber
\end{eqnarray}
Further, for all $u^{T+1} \in\{\pm1\}^{T+2}$, we have
%
%e43 ###
%
\begin{eqnarray}\label{eq:sumqminusp}
&&\sum_{\sigma_\rroot^{T+1}}\sigma_\rroot(T+1)\{
\qprob(\sigma_\rroot^{T+1}\midd u^{T+1}) -
\prob(\sigma_\rroot^{T+1}\midd u^{T+1})\}\nonumber\\[-8pt]\\[-8pt]
&&\qquad
=\frac{2\bias_{T+1}}{k^{(T_*-T)/2}}\bigl(1+o(1)\bigr) .\nonumber
\end{eqnarray}
\end{lemma}
\begin{pf}
The proof is by induction over $T$, for $0\le T<T_*-1$, whereby
in the base case ($T+1=0$), (\ref{eq:qminusp}) corresponds to
%
%e44 ###
%
\begin{equation}
\qprob_0(\sigma_\rroot(0)) -
\prob_0(\sigma_\rroot(0))=\frac{\bias_0}{k^{(T_*+1)/2}} \sigma
_{\rroot}(0)
\bigl(1+o(1)\bigr)
\end{equation}
and holds by definition.
Making use of (\ref{eq:Recursion}) for both $\prob$ and $\qprob
$, we get
%
%e45 ###
%
\begin{eqnarray}\label{eq:qminusp2}
&&\qprob(\sigma_\rroot^{T+1}\midd u^{T+1}) -\prob(\sigma
_\rroot^{T+1}\midd u^{T+1})
\nonumber\\[-2pt]
&&\qquad= \qprob_0(\sigma_\rroot(0))
\sum_{\sigma_1^{T}\cdots\sigma_{k-1}^{T}}
\prod_{t=0}^T \K_{u(t)}\bigl(\sigma_{\rroot}(t+1)|\sigma_{\droot
}(t)\bigr) \prod_{i=1}^{k-1}\qprob(\sigma_i^{T}\midd \sigma
_\rroot^{T})\nonumber\\[-2pt]
&&\qquad\quad{}-\prob_0(\sigma_\rroot(0))
\sum_{\sigma_1^{T}\cdots\sigma_{k-1}^{T}}
\prod_{t=0}^T \K_{u(t)}\bigl(\sigma_{\rroot}(t+1)|\sigma_{\droot
}(t)\bigr)
\prod_{i=1}^{k-1}\prob(\sigma_i^{T}\midd \sigma_\rroot^{T}
)\nonumber\\[-10pt]\\[-10pt]
&&\qquad= \frac{1}{2}
\sum_{\sigma_1^{T}\cdots\sigma_{k-1}^{T}}
\prod_{t=0}^T \K_{u(t)}\bigl(\sigma_{\rroot}(t+1)|\sigma_{\droot
}(t)\bigr)
\nonumber\\[-2pt]
&&\qquad\quad\hphantom{\frac{1}{2}
\sum_{\sigma_1^{T}\cdots\sigma_{k-1}^{T}}}
{}\times
\Biggl\{\prod_{i=1}^{k-1}\qprob(\sigma_i^{T}\midd \sigma_\rroot
^{T})- \prod_{i=1}^{k-1}\prob(\sigma_i^{T}\midd \sigma
_\rroot^{T})\Biggr\}\nonumber\\[-2pt]
&&\qquad\quad{}+O\bigl(k^{-(T_*+1)/2}\bigr) .\nonumber
\end{eqnarray}
%
%=& \frac{1}{2}\sum_{\sigma_1^{T}\dots\sigma_{k-1}^{T}}
%) \cdot
%&\cdot\{\sum_{r=1}^{k-1} \pmatrix{k-1}{r} \prod_{i=1}^r (
%(\sigma_i^{T}\midd\sigma_\rroot^{T})\}+
%O(k^{-(T_*+1)/2})\nonumber\\
%=& \frac{1}{2}\sum_{r=1}^{k-1}\D(r,k) + O (k^{-(T_*+1)/2})\nonumber
%,
%%
%
%where we grouped terms according to their power in $\qprob-\prob$.
Now we use
\begin{eqnarray*}
&&\prod_{i=1}^{k-1}\qprob(\sigma_i^{T}\midd \sigma_\rroot^{T})-
\prod_{i=1}^{k-1}\prob(\sigma_i^{T}\midd \sigma_\rroot^{T})
\\[-2pt]
&&\qquad=\sum_{r=1}^{k-1} \pmatrix{k-1\cr r} \prod_{i=1}^r \{\qprob
(\sigma_i^{T}\midd\sigma_\rroot^T)- \prob(\sigma
_i^{T}\midd\sigma_\rroot^T)\} \\[-2pt]
&&\qquad\quad\hphantom{\sum_{r=1}^{k-1}}
\times{}\prod_{i=r+1}^{k-1}\prob
(\sigma_i^{T}\midd \sigma_\rroot^{T})
\end{eqnarray*}
in (\ref{eq:qminusp2}) to obtain
%
%e46 ###
%
\begin{equation}
\qprob(\sigma_\rroot^{T+1}\midd u^{T+1}) -\prob(\sigma
_\rroot^{T+1}\midd u^{T+1})
= \frac{1}{2}\sum_{r=1}^{k-1}\D(r,k) + O \bigl(k^{-(T_*+1)/2}\bigr) ,
\end{equation}
where
%
%e48 ###
%e47 ###
%
\begin{eqnarray}
\D(r,k) &\equiv&\pmatrix{k-1\cr r} \sum_{\sigma_1^{T}\cdots\sigma
_{k-1}^{T}}
\prod_{t=0}^T \K_{u(t)}\bigl(\sigma_{\rroot}(t+1)|\sigma_{\droot
}(t)\bigr) \cdot
\\[-2pt]
\label{eq:Drk_defined}
&&\hphantom{\pmatrix{k-1\cr r} \sum_{\sigma_1^{T}\cdots\sigma
_{k-1}^{T}}}
{}\times\prod_{i=1}^r \{\qprob(\sigma_i^{T}\midd\sigma
_\rroot^T)- \prob(\sigma_i^{T}\midd\sigma_\rroot^T
)\} \nonumber\\[-10pt]\\[-10pt]
&&\hphantom{\pmatrix{k-1\cr r} \sum_{\sigma_1^{T}\cdots\sigma
_{k-1}^{T}}}
{}\times\prod_{i=r+1}^{k-1}\prob(\sigma_i^{T}\midd \sigma
_\rroot^{T}) .\nonumber
\end{eqnarray}

We claim that only the term $r=1$ is relevant for large $k$:
%
%e49 ###
%
\begin{equation}\label{eq:ClaimLemma}
\sum_{r=2}^{k-1}|\D(r,k)| = o\bigl(k^{-(T_*-T)/2}\bigr) .
\end{equation}
Before proving this claim, let us show that it implies the thesis.
Set $r_0=1$ (we introduce this notation because the calculation below holds
for larger values of $r_0$ and this fact will be exploited in the
next lemma).

The $r=1$ term can be rewritten as
\begin{eqnarray*}
\D(1,k)&=&(k-1) \sum_{\{\sigma_i^{T}\}}
\prod_{t=0}^T
\K_{u(t)}\bigl(\sigma_{\rroot}(t+1)|\sigma_{\droot}(t+1)\bigr)\\
&&\hphantom{(k-1) \sum_{\{\sigma_i^{T}\}}}
{}\times
\{\qprob((\sigma_1)^{T}\midd\sigma_\rroot^T)- \prob
((\sigma_1)^{T}\midd\sigma_\rroot^T)\}\\
&&\hphantom{(k-1) \sum_{\{\sigma_i^{T}\}}}
{}\times\prod
_{i=2}^{k-1}\prob(\sigma_i^{T}\midd
\sigma_\rroot^{T}) .
\end{eqnarray*}
For $t \in\{0,1,\ldots,T\}$, let
%
%e50 ###
%
\begin{equation}\label{eq:Sdef}
\SSS_t\equiv
\{(\sigma_{2})^{T}\cdots(\sigma_{k-1})^{T}\dvtx|\sigma
_{2}(t)+\cdots+\sigma_{k-1}(t)+u(t)|\le r_0\} .
\end{equation}
If $(\sigma_{2})^{T}\cdots(\sigma_{k-1})^{T}$ is
not in $\bigcup_{t=0}^T \SSS_t$, then the sum
over $(\sigma_{1})^{T}$ can be evaluated immediately
[as $\K_{u(t)}(\cdots)$ is independent of $(\sigma_1)^T$]
and is equal to $0$ due to the normalization of
$\qprob( \cdot\midd\sigma_\rroot^T)$ and
$\prob( \cdot\midd\sigma_\rroot^T)$.
We can restrict the innermost sum to
$(\sigma_{2})^{T}\cdots(\sigma_{k-1})^{T}$
in $\bigcup_{t=0}^T \SSS_t$, that is,
$|{\sum}_{i=2}^{k-1}\sigma_i(t)+ u(t)|\le r_0$ for some
$t \in\{0,\ldots, T\}$. Let $\I_0\subseteq\{0,\ldots, T\}$ be the
set of times
such that this happens.

The expectation over
$(\sigma_2)^{T},\ldots,(\sigma_{k-1})^T$ can be
estimated applying Theorem~\ref{thm:sum_approx},
with $N=k-2$, and using Lemmas \ref{lemma:cavity_non_degenerate}
and \ref{lemma:RCconvergence} to check that the hypotheses
\ref{thm:sum_approx} hold for all $k$ large enough. Using\vspace*{1pt} the
induction hypothesis
$|\qprob((\sigma_1)^{T}\midd\sigma_\rroot^T)-
\prob((\sigma_1)^{T}\midd\sigma_\rroot^T)|=O(k^{-(T_*-T+1)/2})$,
this implies that the contribution of terms with $|\I_0|\ge2$
is upper bounded as $kO(k^{-(T_*-T+1)/2})^2=o(k^{-(T_*-T)/2})$
(for $T\le T_*-1$).
Therefore, we make a negligible error if we restrict ourselves
to the case $|\I_0|=1$.

If we let $\Sh_{t_0}\equiv
\SSS_{t_o}\cap\{\bigcap_{t\neq t_0}\overline{\SSS}_t\} $, we then have
%
%e51 ###
%
\begin{eqnarray}\label{eq:Dr1_simple}
\D(1,k) &\equiv&
(k-1) \sum_{t_0=0}^T\sum_{(\sigma_1)^{T}} \{\qprob
((\sigma_1)^{T}\midd\sigma_\rroot^T)-
\prob((\sigma_1)^{T}\midd\sigma_\rroot^T)\}
\nonumber\\
&&\hphantom{(k-1) \sum_{t_0=0}^T\sum_{(\sigma_1)^{T}}}
{}\times\sum_{((\sigma_{2})^{T}\cdots(\sigma_{k-1})^{T})\in\Sh
_{t_0}} \prod_{i=2}^{k-1}
\prob(\sigma_i^{T}\midd \sigma_\rroot^{T}) \nonumber\\[-8pt]\\[-8pt]
&&\hphantom{{}\times(k-1) \sum_{t_0=0}^T\sum_{(\sigma_1)^{T}}\sum_{((\sigma_{2})^{T}\cdots(\sigma_{k-1})^{T})\in\Sh
_{t_0}}}
{}\times\prod_{t=0}^T
\K_{u(t)}\bigl(\sigma_{\rroot}(t+1)|\sigma_{\droot}(t)\bigr)\nonumber\\
&&{} +
o\bigl(k^{-(T_*-T)/2}\bigr) .\nonumber
\end{eqnarray}

Consider the main term
%
%e52 ###
%
\begin{eqnarray}\label{eq:Jdef}
&&J'_{t_0}(\sigma_\rroot^{T},(\sigma_{1})^{T})\nonumber\\[-8pt]\\[-8pt]
&&\qquad\equiv\sum_{((\sigma_{2})^{T}\cdots(\sigma_{k-1})^{T})\in\Sh
_{t_0}} \prod_{i=2}^{k-1}
\prob(\sigma_i^{T}\midd \sigma_\rroot^{T}) \prod_{t=0}^T
\K_{u(t)}\bigl(\sigma_{\rroot}(t+1)|\sigma_{\droot}(t)\bigr).\nonumber
\end{eqnarray}
The arguments of this function will often be dropped in what follows,
and we
will simply write $J'_{t_0}$. For $t\neq t_0$, the kernel
$\K_{u(t)}(\sigma_{\rroot}(t+1)|\sigma_{\droot}(t))$ can
be replaced by an
indicator function, and the constraint
$((\sigma_{2})^{T}\cdots(\sigma_{k-1})^{T})\in\Sh_{t_0}$ can be
removed. For $t=t_0$ we write
\[
\K_{u(t_0)}\bigl(\sigma_{\rroot}(t_0+1)|\sigma_{\droot}(t_0)\bigr) =
\hK_{\Omega(t_0)}'
\Biggl\{\sigma_{\rroot}(t_0+1)\Biggl(u(t_0)+\sum_{i=2}^{k-1}\sigma
_i(t_0)\Biggr)\Biggr\},
\]
where
\[
\hK_a'(x)=\cases{
1, &\quad if $-a<x\le r_0$,\cr
1/2, &\quad if $x=-a$,\cr
0, &\quad otherwise,}
\]
and $\Omega(t) = \sigma_{\rroot}(t+1)\sigma_1(t)$,
$|\Omega(t)|\le r_0$. We thus have
\begin{eqnarray*}
J'_{t_0}
&=&\sum_{(\sigma_{2})^{T}\cdots(\sigma_{k-1})^{T}} \prod_{i=2}^{k-1}
\prob(\sigma_i^{T}\midd \sigma_\rroot^{T})
\hK'_{\Omega(t_0) }\Biggl\{\sigma_{\rroot}(t_0+1)\Biggl(u(t_0)+\sum
_{i=2}^{k-1}\sigma_i(t_0)\Biggr)
\Biggr\} \\
&&\hphantom{\sum_{(\sigma_{2})^{T}\cdots(\sigma_{k-1})^{T}}}
{}\times\prod_{t=0}^T\ind\Biggl\{\sigma_{\rroot}(t_0+1)
\Biggl(u(t)+\sum_{i=2}^{k-1}\sigma_i(t_0)\Biggr)>r_0 \Biggr\} .
\end{eqnarray*}
Notice that the only dependence on $(\sigma_1)^T$ is through $\Omega(t_0)$.
Therefore, we can replace
$\hK_{\Omega(t_0) }'\{ \cdot\}$ by $\hK_{\Omega(t_0) }\{ \cdot
\}=
\hK'_{\Omega(t_0) }\{ \cdot\}-\hK_{0}'\{ \cdot\}$
because the difference, once integrated over $(\sigma_1)^T$
as in (\ref{eq:Dr1_simple}), vanishes by the
normalization of $\qprob( \cdot\midd\sigma_\rroot^T)$ and
$\prob( \cdot\midd\sigma_\rroot^T)$. We thus need to evaluate
\begin{eqnarray*}
J_{t_0}
&=&\sum_{(\sigma_{2})^{T}\cdots(\sigma_{k-1})^{T}} \prod_{i=2}^{k-1}
\prob(\sigma_i^{T}\midd \sigma_\rroot^{T})
\hK_{\Omega(t_0) }\Biggl\{\sigma_{\rroot}(t_0+1)\Biggl(u(t_0)+\sum
_{i=2}^{k-1}\sigma_i(t_0)\Biggr)
\Biggr\} \\
&&\hphantom{\sum_{(\sigma_{2})^{T}\cdots(\sigma_{k-1})^{T}}}
{}\times\prod_{t=0}^T\ind\Biggl\{\sigma_{\rroot}(t+1)\Biggl(
u(t)+\sum_{i=2}^{k-1}\sigma_i(t)\Biggr)> r_0 \Biggr\} ,
\end{eqnarray*}
where, for $a>0, a \in\mathbb{Z}$
\begin{eqnarray*}
\hK_a(x)&=&\cases{
1, &\quad if $-a<x<0$,\cr
1/2, &\quad if $x=-a$ or $x=0$,\cr
0, &\quad otherwise,}
\\
\hK_{-a}(x)&=&\cases{
-1, &\quad if $0<x<-a$,\cr
-1/2, &\quad if $x=0$ or $x=-a$,\cr
0, &\quad otherwise.}
\end{eqnarray*}
Notice that $\sum_{x \in\mathbb{Z}} \hK_a(x) = a \ \forall a \ge-r_0$.

We apply Theorem \ref{thm:sum_approx} for any value of
$s(t_0)\equiv\sum_{i=2}^{k-1}\sigma_i(t_0)$ such that
$\hK_{\Omega(t_0) }\{ \cdot\}$ is nonvanishing, and then sum over
these values.
Notice that\break $|{\sum}_{i=2}^{k-1}\sigma_i(t_0)|\le r_0+1$
and therefore the central limit theorem (Theorem \ref{thm:sum_approx}) applies.
The leading order terms are all independent of $s(t_0)$.
The $O(1/k^{1/4})$ error term in (\ref{eq:sum_approx}) is
multiplied by a
factor $r_0$ and remains therefore negligible.
We get
%
%e54 ###
%e53 ###
%
\begin{eqnarray}
J_{t_0} & = &\frac{1}{\sqrt{k}}
\sigma_{\rroot}(t_0+1)\sigma_1(t_0)
\Phi_{\mu(\sigma_{\rroot}),C_T}(\A_{\infty}(\I))
\bigl(1+o(1)\bigr)\\
\label{eq:Jestimate}
&\equiv& \frac{1}{\sqrt{k}}
\sigma_{\rroot}(t_0+1)\sigma_1(t_0)
J_{t_0}^* \bigl(1+o(1)\bigr) ,
\end{eqnarray}
where\vspace*{1pt}
$\mu(\sigma)=(\mu_0(\sigma), \ldots, \mu_T(\sigma))$ with
$\mu_r(\sigma)=\sum_{s=0}^{r-1}R(r,s)\sigma(s)$,
and $\I_+ = \{t \dvtx\sigma_{\rroot}(t)=+1\}\setminus\{t_0\}$,
$\I_- = \{t \dvtx\sigma_{\rroot}(t)=-1\}\setminus\{t_0\}$ and
$\I_0 = \{t_0\}$. Notice that, in particular $J^*_{t_0=T}=
I_T(\sigma_\rroot^T)$.

If we use this estimate in (\ref{eq:Dr1_simple}), we get
\begin{eqnarray*}
\D(1,k)&=& (k-1)\sum_{t_0=0}^T\sigma_1(t_0)\{\qprob
(\sigma_i^{T}\midd(\sigma_\rroot)_0^T)- \prob(\sigma
_i^{T}\midd\sigma_\rroot^T)\}\\
&&\hphantom{(k-1)\sum_{t_0=0}^T}
{}\times\frac{J^*_{t_0}}{\sqrt{k}} \sigma_{\rroot}(t_0+1)
\bigl(1+ o(1)\bigr) + o\bigl(k^{-(T_*-T)/2}\bigr) \\
&=&k\sum_{t_0=0}^T \frac{2\bias_{t_0}}{k^{(T_*-t_0+1)/2}}\frac
{J^*_{t_0}}{\sqrt{k}} \sigma_{\rroot}(t_0+1)
\bigl(1+ o(1)\bigr) + o\bigl(k^{-(T_*-T)/2}\bigr) \\
&=& I_T(\sigma_\rroot^T) \frac{2\bias_{T}}{k^{T_*-T}}
\sigma_{\rroot}(t_0+1) \bigl(1+ o(1)\bigr) ,
\end{eqnarray*}
which, along with (\ref{eq:R_t_tplus1}) implies the thesis
equation (\ref{eq:qminusp}).

Let us now prove the claim (\ref{eq:ClaimLemma}).
Recall that induction hypothesis we have
$\qprob(\sigma_i^{T}\midd\sigma_\rroot^T)-
\prob(\sigma_i^{T}\midd\sigma_\rroot^T) = O(k^{-(T_*-T+1)/2})$.
Since $|\K_{u(t)}(\sigma_{\rroot}(t+1)|\sigma_{\droot}(t)
)|\le1$,
this implies
\begin{eqnarray*}
|\D(r,k)|&\le& k^r \sum_{(\sigma_{1})^{T}\cdots(\sigma_{r})^{T}}
\prod_{i=1}^{r}\bigl|\qprob(\sigma_i^{T}\midd\sigma_\rroot^T
)- \prob(\sigma_i^{T}\midd\sigma_\rroot^T)\bigr|\\
&=&
O\bigl(k^{-r(T_*-T-1)/2}\bigr) .
\end{eqnarray*}
Since $T_*-T-1\ge1$, we have
\[
\sum_{r=3}^{k-1}|\D(r,k)| = O\bigl(k^{-3(T_*-T-1)/2}\bigr)= o\bigl(k^{-(T_*-T)/2}\bigr) .
\]
Further, $|\D(2,k)| = O(k^{-(T_*-T-1)})=o(k^{-(T_*-T)/2})$ unless
$T=T_*-2$.

In order to argue in the $r=2$, $T=T_*-2$ case, we will proceed analogously
to $r=1$. Consider the definition of $\D(2,k)$
in (\ref{eq:Drk_defined}). If $(\sigma_{3})^{T},\ldots,(\sigma_{k-1})^{T}$
are such that
$|{\sum}_{i=3}^{k-1}\sigma_i(t)+u(t)|> 2$ for all $t\in\{0,\ldots,T\}$
then the factors $\K_{u(t)}(\sigma_{\rroot}(t+1)|\sigma
_{\droot}(t))$
become independent of $(\sigma_{1})^T$,
$(\sigma_{2})^T$. We can therefore carry out the sum over these variables
obtaining
\begin{eqnarray*}
&&\sum_{(\sigma_{1})^{T},(\sigma_{2})^{T}}\prod_{i=1}^2 \{
\qprob(\sigma_i^{T}\midd\sigma_\rroot^T)- \prob(\sigma
_i^{T}\midd\sigma_\rroot^T)\}\\
&&\qquad=\prod_{i=1}^r \sum_{(\sigma_{i})^{T}} \{
\qprob(\sigma_i^{T}\midd\sigma_\rroot^T)- \prob(\sigma
_i^{T}\midd\sigma_\rroot^T)\} = 0 ,
\end{eqnarray*}
because both $\qprob( \cdot\midd\sigma_\rroot^T)$ and
$\qprob( \cdot\midd\sigma_\rroot^T)$ are normalized.
Therefore, we can restrict the sum to those
$(\sigma_{3})^{T},\ldots, (\sigma_{k-1})^{T}$ such that
$|{\sum}_{i=3}^{k-1}\sigma_i(t_0)+u(t_0)|\le2$ for at least
one $t_0\in\{0,\ldots,T\}$.
However, analogously to the case $r=1$, the probability
that this happens for the i.i.d. nondegenerate random vectors
$(\sigma_{3})^{T}\cdots(\sigma_{k-1})^{T}$ is at most $O(k^{-1/2})$,
using Theorem \ref{thm:sum_approx}.
Together with the induction hypothesis,
this yields $|\D(2,k)| = O(k^{-1/2}\cdot k^{-(T_*-T-1)})=o(k^{-(T_*-T)/2})$,
which proves the claim.

Finally, (\ref{eq:sumqminusp}) follows from (\ref{eq:qminusp})
using the definitions (\ref{eq:ITDef}), (\ref{eq:omega_recursion})
and the identity (\ref{eq:RespRecursion}).
\end{pf}

The next lemma says that Lemma \ref{lemma:qminusp} extends to $T=T_*-1$.
Since this case requires a different (more careful) calculation,
we state it separately, although the conclusion is the same as
for $T<T_*-1$. The proof is in Appendix~\ref{app:qminusptstar}.
\begin{lemma}\label{lemma:qminusptstar}
Let $I_T(\sigma^{T})$ be defined as in Lemma \ref{lemma:qminusp}, and
define $\bias_{T_*}$ by
%
%e55 ###
%
\begin{equation}
\bias_{T_*} = R(T_*,T_*-1) \bias_{T_*-1} .
\end{equation}
Then, for all
$\sigma_\rroot^{T_*}, u^{T_*} \in\{\pm1\}^{T_*+1}$, we have
%
%e56 ###
%
\begin{equation}\label{eq:qminusptstar}
\qprob(\sigma_\rroot^{T_*}\midd u^{T_*}) - \prob(\sigma
_\rroot^{T_*}\midd u^{T_*})=\frac{\bias_{T_*-1}}{k^{1/2}} \sigma
_{\rroot}(T_*)
I_{T_*-1}(\sigma_\rroot^{T_*-1}
) \bigl(1+o(1)\bigr) .
\end{equation}
Further, for all $u^{T_*} \in\{\pm1\}^{T_*+1}$, we have
%
%e57 ###
%
\begin{equation}\label{eq:sumqminusptstar}
\sum_{\sigma_\rroot^{T_*}}\sigma_{\rroot}(T_*)\{
\qprob(\sigma_\rroot^{T_*}\midd u^{T_*}) -
\prob(\sigma_\rroot^{T_*}\midd u^{T_*})\}
=\frac{2\bias_{T_*}}{k^{1/2}}\bigl(1+o(1)\bigr) .
\end{equation}
\end{lemma}

We now show that, for the dynamics under external field,
the process of the root spin $\{\sigma_{\rroot}(t)\}_{t\ge0}$
converges as in Theorem \ref{thm:biased_convergence}.
\begin{lemma}\label{lemma:rooted_biased_convergence}
For $T_*$ a nonnegative integer, $\bias_0>0$,
and $\{u(t)\}_{t\ge0}\in\{\pm1\}^{\naturals}$,
consider the majority process under external field $u$, on the
rooted tree
$\G_{\rroot}=(\V_{\rroot},\Ed_{\rroot})$, with
i.i.d. initialization with bias $\theta= \bias_0/k^{(T_*+1)/2}$.
Then for any $T\ge T_*+2$, we have
\[
( \sigma_\rroot(0), \sigma_\rroot(1), \ldots, \sigma_\rroot
(T) )
\stackrel{d}{\rightarrow} \bigl( \tau(0), \tau(1), \ldots, \tau
(T_*), \sigma(T_*+1), +1, +1, \ldots, +1 \bigr) ,
\]
where the random variable
$\sigma(T_*+1)$ dominates stochastically
$\tau(T_*+1)$, and $\prob\{\sigma(T_*+1)>\tau(T_*+1)\}$
is strictly positive.
Finally, there exists
$\const(\bias_0, T_*)>0$ such that, for any $T\ge T_*+2$,
\[
\E_{\theta}\{\sigma_{\rroot}(T)\} \ge1-e^{-\const(\bias_0,
T_*)k} .
\]
\end{lemma}
\begin{pf}
An immediate consequence of (\ref{eq:sumqminusptstar}) and
(\ref{eq:sumqminusp}) is that,
for all $T$, $0\le T\le T_*$
%
%e58 ###
%
\begin{equation}\label{eq:xqminusxp}
\Ex_{\qprob(\cdot\midd u^{T})} [\sigma_\rroot(T)] - \Ex
_{\prob(\cdot\midd u^{T})} [\sigma_\rroot(T)]=\frac{2\bias
_{T}}{k^{(T_*-T+1)/2}}\bigl(1+o(1)\bigr) .
\end{equation}
Further Lemmas \ref{lemma:RCconvergence} and \ref{lemma:qminusp}
imply that
%
%e59 ###
%
\begin{eqnarray}
|\E_{\qprob}\{\sigma_{\rroot}(t)\sigma_{\rroot}(s)\}
-C(t,s)| &=& o(1) ,\nonumber\\[-8pt]\\[-8pt]
\Biggl|\E_{\qprob}
\sigma_{\rroot}(t)-\frac{1}{\sqrt{k}}\sum_{s=0}^{t-1}R(t,s)
u(s)\Biggr| &=&
o(k^{-1/2})\nonumber
\end{eqnarray}
for $t,s\le T_*-1$.
At $T_*$, using Lemma \ref{lemma:qminusptstar} and
(\ref{eq:xqminusxp}) with $T=T_*$ we obtain
%
%e60 ###
%
\begin{eqnarray}\label{eq:CovarianceMeanqtstar}
|\E_{\qprob}\{\sigma_\rroot(T_*)\sigma_\rroot(s)\}
-C(t,s)| &=& o(1),\nonumber\\[-8pt]\\[-8pt]
\Biggl|\E_\qprob\Biggl[\sigma_\rroot(T_*+1)-\frac{1}{\sqrt{k}}\Biggl\{
\sum_{s=0}^{T_*-1}R(t,s) u(s) + 2\bias_{T_*} \Biggr\}\Biggr]\Biggr| &=&
o(k^{-1/2}) ,\nonumber
\end{eqnarray}
which holds for all $s\le T_*$.

Now, repeating the CLT-based argument as in the proof of Lemma \ref
{lemma:RCconvergence}, we can show that with a biased initialization,
$(\sigma_\rroot(0), \sigma_\rroot(1), \ldots, \sigma_\rroot
(T_*+1))$ converges to a modified cavity process, where the governing
equation at $T_*$ is
%
%e61 ###
%
\begin{equation}\label{eq:modified_cavity_recursion}
\sigma(T_*+1)=\sign\Biggl(\eta(T_*) +
\sum_{s=0}^{T_*-1}{R(t,s)\tau(s)} + 2\bias_{T_*} \Biggr) .
\end{equation}
Convergence to this process occurs for all $u^{T_*+1}$. Clearly,
since $\bias_{T_*}>0$, this process dominates the unmodified cavity
process. Further, we have $B(\bias_0)=\E[\sigma'(T_*+1)] > 0$. We
know $\lim_{k \rightarrow\infty}
\E[\sigma_\rroot(T_*+1)]=\E[\sigma(T_*+1)]$, and therefore there
exists $k_0$, such that for all
$k>k_0$, $\E[\sigma_\rroot(T_*+1)]>B(\bias_0, T_*)/2$. Plugging
this back into the recursion
equation (\ref{eq:Recursion}) applied to $\qprob$, and using Azuma's
inequality, we see that at $T=T_*+2$
\[
\E_{\theta}\{\sigma_\rroot(T)\} \ge1-e^{-B^2k/8}\qquad \forall
k>k_0 .
\]
Clearly, the same continues to hold for $T>T_*+2$, for sufficiently
large $k$.
\end{pf}

Finally, we can prove Theorem \ref{thm:biased_convergence}.
\begin{pf*}{Proof of Theorem \ref{thm:biased_convergence}}
As in the proof of Theorem \ref{thm:unbiased_convergence}, we
consider the dynamics on the rooted tree $\G_{\rroot}$
under external fields $u_-=(-1,-1, \ldots)$ and $u_+=(+1,+1, \ldots)$,
and we denote by $\sigma_{\rroot,-}^T$, $\sigma_{\rroot,+}^T$
be the corresponding trajectories.
By monotonicity of the dynamics, the process $\sigma_i^T$ at any vertex
of the regular tree $\G$ is dominated by $\sigma_{\rroot,+}^T$
and dominates $\sigma_{\rroot,-}^T$. Since by Lemma
\ref{lemma:rooted_biased_convergence} both $\sigma_{\rroot,+}^T$
and $\sigma_{\rroot,-}^T$ converge to the same limit, the same holds for
$\sigma_i^T$ as well.
\end{pf*}
%
%***********************************************************
%
%s5 ###
\section{\texorpdfstring{Lower bound: Proof of Theorem
\protect\ref{thm:LowerBound}}{Lower bound: Proof of Theorem 2.9}}
\label{sec:LowerBound}

In this section, we prove Theorem~\ref{thm:LowerBound}, that provides
a sequence of lower bounds on the consensus threshold
$\theta_*(k)$.

Our lower bounds are based on the formation of ``stable'' structures of
$-1$ spins, that is, once
such a structure is formed, it continues to exist at all future times,
hence preventing
consensus from being reached.

Consider $k=3$. Clearly, if there is an infinite path of $-1$ spins,
spins along the path
remain unchanged for all future times. In fact, it is sufficient to
have an infinite path
having alternate vertices with $-1$ spins, due to the ``bipartite''
nature of the dynamics.
To see this, label an arbitrary node on the path $0$. Choose
an arbitrary direction on the path and hence label nodes $ \ldots,
-2, -1, 0 , 1, 2, \ldots.$ Suppose
that at time $t_0$, nodes with even labels $ \ldots, -2, 0, 2,
\ldots$ all have spin $-1$.
At time $t_0+1$, all nodes with odd labels will have spin $-1$. At time
$t_0+2$, all nodes with
even labels will again have spin $-1$ and so on. % After one step, the
%complement of the initial set
Note that any value for $t_0$ suffices. $t_0=0$ corresponds to an
alternating core existing initially, but
it is sufficient for such structure to be formed, say, at $t_0 =3$.

This idea can be generalized to any $k$. A $\lceil\frac{k+1}{2}
\rceil$-core of $-1$ spins
is clearly stable. In fact, an alternating $\lceil\frac{k+1}{2}
\rceil$-core of $-1$ spins (having
alternate ``levels'' of $-1$ spins) is stable. We formally define such
structures below. The key point to note
is that though such structures exist in abundance at $T=0$ with small
positive $\theta$ for $k=3$, this is not the case
for larger values of $k$. We do not obtain a positive lower bound for
$k>3$ based on analysis of the
initial configuration only. Thus, we need a means to show that such
structures form in abundance at $t_0(k) > 0$
for positive bias. We develop a set of iterative equations (see Theorem
\ref{thm:LowerBound})
whose fixed point corresponds (roughly) to the probability of
formation of an alternating core at time $t_0$ at an arbitrary node. A
nontrivial fixed point implies that alternating cores
are formed in abundance. Such iterative equations are in the spirit of
the exact cavity recursion (Lemma \ref{lemma:ExactCavity})
though the analysis here is more intricate.

%s5.1 ###
\subsection{Notation and preliminaries}
Let $\HH= (\V_{\HH}, \Ed_{\HH})$ be an induced subgraph of $\G$
with vertex set $\V_{\HH}$ and edge set $\Ed_{\HH}$. We denote by
$\partial_{\HH} i$ the set of neighbors in $\HH$ of a node $i \in
\HH$.
Since $\HH$ is an induced subgraph of $\G$, we have $\V_\HH
\subseteq\V$
and, for all $i\in\V_{\HH}$, $\partial_{\HH}i = \{j\dvtx j
\in\partial i,j \in\V_\HH\}$. Given the graph $\G$,
$\V_\HH$ uniquely determines the induced subgraph $\HH$.
\begin{definition}
The subgraph $\HH$ is an $r$-\textit{core} of $\G$
with respect to spins $\sigma\dvtx\V\rightarrow\{-1,+1\}$ if $\HH$
is an induced subgraph of $\G$ such that
$|\partial_{\HH} i| \geq r$ and $\sigma_i=-1$ for all $i \in\V_\HH$.
\end{definition}

Clearly, this definition is useful only for $r \leq k$. Now, it is
easy to see that if $\HH$ is an $\lceil\frac{k+1}{2}
\rceil$-core with respect to $\us(T)$, then it is also an $
\lceil\frac{k+1}{2} \rceil$-core with respect to $\us(T')$
for all $T' > T$, by definition of majority dynamics. In fact, a
less stringent requirement suffices for persistence of negative
spins.
\begin{definition}
\label{def:alt_r_core}
$\HH$ is an \textit{alternating $r$-core}
of a graph $\G$ with respect to spins $\sigma\dvtx\V\rightarrow\{
-1,+1\}$, if $\HH$ is an induced subgraph of $\G$ such that:
\begin{enumerate}
\item
$|\partial_{\HH} i| \geq r \ \forall i \in\V_H$,
\item
there is a partition $(\V_{-,\HH},\V_{*,\HH})$ of $\V_\HH$ such
that:
\begin{enumerate}[(a)]
\item[(a)] $\sigma_i = -1$ for all $i \in\V_{-,\HH}$,
\item[(b)] $\partial_\HH i \subseteq\V_{-,\HH}$ for all $i \in\V
_{*,\HH}$ and
$\partial_\HH i \subseteq\V_{*,\HH}$ for all $i \in\V_{-,\HH}$,
that is, $\HH$ is bipartite with respect to the vertex partition
$(\V_{-,\HH},\V_{*,\HH})$. We call $\V_{-,\HH}$ the \textit
{even} vertices
and $\V_{*,\HH}$ the \textit{odd} vertices.
\end{enumerate}
\end{enumerate}
\end{definition}
\begin{lemma}
\label{lemma:alt_core_persistence} If $\HH$ is an alternating $
\lceil\frac{k+1}{2} \rceil$-core with respect to $\us(T)$,
then it is also an alternating $\lceil\frac{k+1}{2}
\rceil$-core with respect to $\us(T')$ for all $T' > T$.
\end{lemma}
\begin{pf}
We prove the lemma by induction over $T'$. For this proof only, let
\[
\Ev_{T'} \equiv\mbox{``$\HH$ is an alternating }\biggl\lceil
\frac{k+1}{2}\biggr\rceil\mbox{-core with respect to $\us(T')$.''}
\]
Clearly, $\Ev_T$ holds. Suppose $\Ev_{T'}$ holds. Let
$(\V_{-,\HH},\V_{*,H})= (\V_1,\V_2)$ be a partition of $\HH$ as in
the Definition \ref{def:alt_r_core}.
In particular $\sigma_i(T')=-1 $ for all $i \in V_1$. By the
definition of majority dynamics, we know that
$\sigma_i(T'+1)=-1$ for all $i \in\V_2$.
As a consequence $\HH$ is an alternating $\lfloor(k+1)/2\rfloor$-core
with respect to $\sigma(T'+1)$ with partition
$(\V_{-,\HH},\V_{*,H})=(\V_2,\V_1)$, and therefore $\Ev_{T'+1}$ holds.
\end{pf}

We now proceed in a manner similar to Section \ref{sec:ExactCavity}.
We consider the rooted tree $\G_{\rroot}=(\V_{\rroot},\Ed_{\rroot
})$, with
a root vertex $\rroot$ having $k-1$ ``children.'' The root spin
$\sigma_{\rroot}$ evolves under as external field $\{u(t)\}_{t\ge0}$
as in (\ref{eq:ModifiedUpdateRoot})
and we denote by $\prob(\sigma_\rroot^T\midd u^T)$
its distribution.
We use $\tilde{\partial} i$ to denote the ``children'' of node
$i \in\G_\rroot$. In this section we will assume $u^T \in\{-1,+1\}^{T+1}$.
\begin{definition}
\label{def:rooted_alt_r_core} $\HH$ is a \textit{rooted alternating
$r$-core} of $\G_\rroot$ with respect to spins $\sigma\dvtx\V
_\rroot
\rightarrow\{-1,+1\}$, if $\HH$ is a connected induced subgraph of
$\G_\rroot$ such that:
\begin{enumerate}
\item$\rroot\in\V_\HH$.
\item
$|\tilde{\partial}_{\HH} i| \geq r-1$ for all $i \in\V_\HH$,
\item
there is a partition $(\V_{-,\HH},\V_{*,\HH})$ of $\V_\HH$ such
that:
\begin{enumerate}[(a)]
\item[(a)] $\sigma_i = -1$ for all $i \in\V_{-,\HH}$,
\item[(b)] $\partial_\HH i \subseteq\V_{-,\HH}$ for all
$i \in\V_{*,\HH}$ and $\partial_\HH i \subseteq\V_{*,\HH}$ for
all $i
\in\V_{-,\HH}$, that is, $\HH$ is bipartite with respect to the vertex
partition $(\V_{-,\HH},\V_{*,\HH})$. We call $\V_{-,\HH}$ the
\textit{even} vertices and $\V_{*,\HH}$ the \textit{odd} vertices.
\end{enumerate}
\end{enumerate}
\end{definition}

Let $\G_{\rroot}^d=(\V_{\rroot}^d,\Ed_{\rroot}^d)$, be the induced
subgraph of $\G_\rroot$ containing all vertices that are at a depth
less than or equal to $d$ from $\rroot$, the depth of $\rroot$ itself
being 0. For example, $\G_\rroot^0$ contains $\rroot$ alone. Denote by
$\tilde{\partial} \G_{\rroot}^d$, the set of leaves of
$\G_{\rroot}^d$. For example, $\tilde{\partial}
\G_{\rroot}^0=\{\rroot\}$.
\begin{definition}
\label{def:rooted_partial_alt_r_core}
$\HH$ is a \textit{depth-$d$ rooted alternating $r$-core}
of $\G_\rroot$ with respect to spins $\sigma\dvtx\V_\rroot^d
\rightarrow\{-1,+1\}$, if $\HH$ is an connected induced subgraph of
$\G_\rroot^d$ such that:
\begin{enumerate}
\item$\rroot\in\V_\HH$,
\item
$|\tilde{\partial}_{\HH} i| \geq r-1$ for all $i \in\V_\HH
\setminus\tilde{\partial} \G_{\rroot}^d$,
\item
there is a partition $(\V_{-,\HH},\V_{*,\HH})$ of $\V_\HH$ such
that:
\begin{enumerate}[(a)]
\item[(a)] $\sigma_i = -1$ for all $i \in\V_{-,\HH}$,
\item[(b)] $\partial_\HH i \subseteq\V_{-,\HH}$ for all
$i \in\V_{*,\HH}$ and $\partial_\HH i \subseteq\V_{*,\HH}$ for
all $i
\in\V_{-,\HH}$, that is, $\HH$ is bipartite with respect to the vertex
partition $(\V_{-,\HH},\V_{*,\HH})$. We call $\V_{-,\HH}$ the
\textit{even} vertices and $\V_{*,\HH}$ the \textit{odd} vertices.
\end{enumerate}
\end{enumerate}
\end{definition}

We define $\HH_{\rroot,{\mathrm{even}}}(T)$ to be the maximal rooted
alternating $\lceil\frac{k+1}{2}\rceil$-core
of $\G_\rroot$ with respect to $\us(T)$, such that $\rroot$ is an
even vertex. For all $d\ge0$,
we define $\HH_{\rroot,{\mathrm{even}}}^d(T)$ to be the maximal
depth-$d$ rooted alternating
$\lceil\frac{k+1}{2}\rceil$-core of $\G_\rroot$ with
respect to $\us^d(T)$, such that $\rroot$ is even.
Here $\us^d(T)$ is the restriction of $\us(T)$ to $\V_\rroot^d$.
We similarly define $\HH_{\rroot,{\mathrm{odd}}}(T)$ and $\HH
_{\rroot
,{\mathrm{odd}}}^d(T)$.

We define $\EvC_{\mathrm{even}}(T)= \{\rroot\in
\V_{\HH_{\rroot,{\mathrm{even}}}(T)} \}$, that is, $\EvC_
{\mathrm{even}}(T)$ is
the event of $\HH_{\rroot,{\mathrm{even}}}(T)$ being nonempty. Define
$\EvC_{\mathrm{even}}^d(T)= \{\rroot\in
\V_{\HH_{\rroot,{\mathrm{even}}}(T)}^d \}$. We similarly define
$\EvC_{\mathrm{odd}}(T)$ and $\EvC_{\mathrm{odd}}^d(T)$. It is easy
to see
that $\EvC_{\mathrm{even}}^d(T) \subseteq\EvC_{\mathrm{even}}^{d'}(T),
\ \forall d' < d$. Also, $\EvC_{\mathrm{even}}(T) = \bigcap_{d\ge0}
\EvC_{\mathrm{even}}^d(T)$. Similarly, $\EvC_{\mathrm{odd}}^d(T)
\subseteq
\EvC_{\mathrm{odd}}^{d'}(T), \ \forall d' < d$ and $\EvC_
{\mathrm{odd}}(T) =
\bigcap_{d\ge0} \EvC_{\mathrm{odd}}^d(T)$. We thus have the following
remark.
\begin{lemma}\label{lemma:alt_children_seq}
$\EvC_{{\mathrm{even}}}^d(T), d\ge0$, form a monotonic nonincreasing
sequence of events in $d$ with limit $\bigcap_{d\ge0}
\EvC_{{\mathrm{even}}}^d(T) = \EvC_{\mathrm{even}}(T)$, for all
$T\ge0$,
similarly for the \textup{``}odd\textup{''} quantities.
\end{lemma}

Let $\Ev(\traj^T) \equiv\{\sigma_\rroot(t)=\traj(t), 0 \le t\le T
\}$
and define the events
\begin{eqnarray*}
\EvA_{\mathrm{even}}(T, \traj^T ) &=&\EvC_{\mathrm{even}}(T) \cap
\Ev
(\traj^T) ,\\
\EvA_{\mathrm{even}}^d(T, \traj^T)&=&\EvC_{\mathrm{even}}^d(T) \cap
\Ev
(\traj^T),\qquad d\ge0 .
\end{eqnarray*}
We similarly define $\EvA_{\mathrm{odd}}, \EvA_{\mathrm{odd}}^d$.

We now proceed to define
$\Psi_{{\mathrm{even}},T}^d(\sigma_\rroot^T\midd u^T)$ and
$\Psi_{{\mathrm{odd}},T}^d(\sigma_\rroot^T\midd u^T)$ as
probabilities. Immediately after the new definitions, we
show that they are consistent with the recursive definitions in
Theorem \ref{thm:LowerBound}.
\begin{definition}\label{def:psi_probability_defn}
\begin{eqnarray*}
\Psi_{{\mathrm{even}},T}(\sigma_\rroot^T\midd u^T) &\equiv&
\prob(\EvA_{\mathrm{even}}(T, \sigma_\rroot^T )\midd u^T) ,\\
\Psi_{{\mathrm{even}},T}^d(\sigma_\rroot^T\midd u^T)&\equiv&
\prob(\EvA_{\mathrm{even}}^d(T, \sigma_\rroot^T)\midd u^T)
,\qquad d\ge0 .
\end{eqnarray*}
We similarly define $\Psi_{{\mathrm{odd}},T}(\sigma_\rroot
^T\midd u^T), \Psi_{{\mathrm{odd}},T}^d(\sigma_\rroot
^T\midd u^T)$.
\end{definition}

It follows from Lemma \ref{lemma:alt_children_seq} that $\EvA_
{\mathrm{even}}(T)=\bigcap_{d\ge0}\EvA_{\mathrm{even}}^d(T)$.
Therefore\break
$ \Psi_{{\mathrm{even}},T}^d(\sigma_\rroot^T\midd u^T)$ is
nonincreasing in
$d$ and by the monotone convergence theorem
%
%e62 ###
%
\begin{equation}
\Psi_{{\mathrm{even}},T}(\sigma_\rroot^T\midd u^T) = \lim
_{d\rightarrow\infty} \Psi_{{\mathrm{even}},T}^d(\sigma_\rroot
^T\midd u^T) .
\end{equation}
Similarly, we have
%
%e63 ###
%
\begin{equation}\label{eq:psi_odd_limit_proved}
\Psi_{{\mathrm{odd}},T}(\sigma_\rroot^T\midd u^T)=
\lim_{d\rightarrow
\infty}\Psi_{{\mathrm{odd}},T}^d(\sigma_\rroot^T\midd u^T) .
\end{equation}
This is
consistent with the definition of
$\Psi_{{\mathrm{odd}},T}(\sigma_\rroot^T\midd u^T)$ in Theorem
\ref{thm:LowerBound}.

The values for $d=0$ follow from Definition
\ref{def:psi_probability_defn},
%
%e64 ###
%
\begin{eqnarray}\label{eq:initial_psi_redef}
\Psi_{{\mathrm{odd}},T}^0(\sigma_\rroot^T\midd u^T) &=& \prob
(\sigma
_\rroot^T\midd u^T) , \nonumber\\[-8pt]\\[-8pt]
\Psi_{{\mathrm{even}},T}^0(\sigma_\rroot^T\midd u^T) &=&
\prob(\sigma_\rroot^T\midd u^T) \ind\bigl(\sigma_\rroot
(T)=-1\bigr) .\nonumber
\end{eqnarray}
Note consistency with (\ref{eq:initial_psi}).

Next, in Lemma \ref{lemma:PsiRecursion}, we show that
$\Psi_{{\mathrm{even}},T}^d(\sigma_\rroot^T\midd u^T)$ and
$\Psi_{{\mathrm{odd}},T}^d(\sigma_\rroot^T\midd u^T)$---as per
Definition \ref{def:psi_probability_defn}---satisfy
(\ref{eq:Alt_fp1}), (\ref{eq:Alt_fp2}) [repeated as
(\ref{eq:Alt_fp1_redef}), (\ref{eq:Alt_fp2_redef}) below].

%s5.2 ###
\subsection{\texorpdfstring{Proof of Theorem
\protect\ref{thm:LowerBound}}{Proof of Theorem 2.9}}
\begin{lemma}\label{lemma:PsiRecursion}
The following iterative equations are satisfied for all $d\ge0$:
%
%e66 ###
%e65 ###
%
\begin{eqnarray}\label{eq:Alt_fp1_redef}
&&\Psi_{{\mathrm{odd}},T}^{d+1}(\sigma_\rroot^T\midd u^T)\nonumber\\
&&\qquad=
\prob_0(\sigma_\rroot(0))
\sum_{r=\lceil({k+1})/{2} \rceil-1}^{k-1}\pmatrix
{k-1\cr r} \nonumber\\
&&\qquad\quad{}\times\sum_{\sigma_1^{T}\cdots\sigma_{k-1}^{T}}
\prod_{t=0}^{T-1}
\K_{u(t)}\bigl(\sigma_{\rroot}(t+1)|\sigma_{\partial\rroot
}(t)\bigr)\\
&&\qquad\quad\hphantom{{}\times\sum_{\sigma_1^{T}\cdots\sigma_{k-1}^{T}}}
{}\times\prod_{i=1}^{r}\Psi_{{\mathrm{even}},T}^d
(\sigma_i^{T}\midd \sigma_\rroot^{T})\nonumber\\
&&\qquad\quad\hphantom{{}\times\sum_{\sigma_1^{T}\cdots\sigma_{k-1}^{T}}}
{}\times\prod_{i=r+1}^{k-1}\bigl(\prob(\sigma_i^{T}\midd \sigma_\rroot
^{T})- \Psi_{{\mathrm{even}},T}^d(\sigma_i^{T}\midd \sigma
_\rroot^{T})\bigr) ,\nonumber
\\
\label{eq:Alt_fp2_redef}
&&\Psi_{{\mathrm{even}},T}^{d+1}(\sigma_\rroot^T\midd u^T) \nonumber\\
&&\qquad=
\ind\bigl(\sigma_\rroot(T)=-1\bigr) \prob_0(\sigma_\rroot(0))
\sum_{r=\lceil({k+1})/{2} \rceil-1}^{k-1}\pmatrix
{k-1\cr r} \nonumber\\
&&\qquad\quad{}\times\sum_{\sigma_1^{T}\cdots\sigma_{k-1}^{T}}
\prod_{t=0}^{T-1}
\K_{u(t)}\bigl(\sigma_{\rroot}(t+1)|\sigma_{\partial\rroot
}(t)\bigr) \\
&&\qquad\quad\hphantom{{}\times\sum_{\sigma_1^{T}\cdots\sigma_{k-1}^{T}}}
{}\times \prod_{i=1}^{r}\Psi_{{\mathrm{odd}},T}^d
(\sigma_i^{T}\midd \sigma_\rroot^{T})
\nonumber\\
&&\qquad\quad\hphantom{{}\times\sum_{\sigma_1^{T}\cdots\sigma_{k-1}^{T}}}
{}\times \prod_{i=r+1}^{k-1}\bigl(\prob(\sigma_i^{T}\midd \sigma_\rroot
^{T})- \Psi_{{\mathrm{odd}},T}^d(\sigma_i^{T}\midd \sigma
_\rroot^{T})\bigr) ,\nonumber
\end{eqnarray}
\begin{eqnarray*}
\\[-12pt]
\K_{u(t)}(\cdots)
&\equiv&\cases{
\displaystyle \ind\Biggl\{\sigma_\rroot(t+1) = \sign\Biggl(\sum_{i=1}^{k-1}\sigma_i(t)+
u(t)\Biggr)\Biggr\},\cr
\qquad\hspace*{12.7pt}\mbox{if $\displaystyle \sum_{i=1}^{k-1}\sigma_i(t)+
u(t)\neq0 $},\cr
\dfrac{1}{2}, \qquad \mbox{otherwise}.}
\end{eqnarray*}
\end{lemma}

Lemma \ref{lemma:PsiRecursion} contains a type of ``cavity recursion''
similar to Lemma \ref{lemma:ExactCavity}. The main difference is that
here we iterate over depth $d$ instead of time $T$. The proof is
similar to that of Lemma \ref{lemma:ExactCavity} and can be found in
Appendix \ref{app:ExactCavityRec}.

Let the vector of values taken by $\Psi_{{\mathrm{odd}},T}( \cdot
\midd
\cdot)$ be denoted by $\bar{\Psi}_{{\mathrm{odd}},T}$. Define
$\bar{\Psi}_{{\mathrm{even}},T}$ similarly. Define $\bar{\Psi
}_T=( \bar{\Psi}_{{\mathrm{odd}},T}, \bar{\Psi}_{
{\mathrm{even}},T} )$.\vspace*{1pt}

As before, $\prob_0(-1)=\frac{1-\theta}{2}$ and $\prob_0(+1)=\frac
{1+\theta}{2}$. Define $\theta_{\mathrm{lb}}(k,T)=\sup\{
\theta\dvtx\break\bar{\Psi}_{{\mathrm{odd}},T} \succ0 \}$, where $\bar{v}
\succ0$, denotes that every component of
the vector $\bar{v}$ is strictly positive.

Finally, we relate quantities on the process on the rooted graph $\G
_\rroot$ to the process on the infinite $k$-ary tree $\G$. Pick an
arbitrary node $v \in\V$. Let $\G^d=(\V^d,\Ed^d)$, be the induced
subgraph of $\G$ containing all vertices that are at a distance less
than or equal to $d$ from $v$. For example, $\G^0$ contains $v$ alone.
Denote by $\tilde{\partial} \G^d$, the set of leaves of $\G^d$. For
example, $\tilde{\partial} \G^0=\{v\}$.
\begin{definition}
\label{def:partial_r_core}
$\HH$ is a \textit{depth-$d$ alternating $r$-core}
of $\G$ with respect to spins $\sigma\dvtx\V^d \rightarrow\{-1,+1\}$,
if $\HH$ is an connected induced subgraph of $\G^d$ such that:
\begin{enumerate}
\item$v \in\V_\HH$,
\item
$|\tilde{\partial}_{\HH} i| \geq r-1$ for all
$i \in\V_\HH\setminus\tilde{\partial} \G^d$,
\item
there is a partition $(\V_{-,\HH},\V_{*,\HH})$ of $\V_\HH$ such
that:
\begin{enumerate}[(a)]
\item[(a)] $\sigma_i = -1$ for all $i \in\V_{-,\HH}$,
\item[(b)] $\partial_\HH i \subseteq\V_{-,\HH}$ for all
$i \in\V_{*,\HH}$ and $\partial_\HH i \subseteq\V_{*,\HH}$
for all $i \in\V_{-,\HH}$, that is, $\HH$ is bipartite with respect
to the vertex partition $(\V_{-,\HH},\V_{*,\HH})$. We call $\V
_{-,\HH}$ the \textit{even} vertices and $\V_{*,\HH}$ the \textit
{odd} vertices.
\end{enumerate}
\end{enumerate}
\end{definition}

We define $\widehat{\mathcal{H}}_{{\mathrm{even}}}(T)$, as the maximal
alternating $\lceil\frac{k+1}{2}\rceil$-core
of $\G$ with respect to $\us(T)$, such that $v$ is an even vertex.
For all $d\ge0$,
we define $\widehat{\mathcal{H}}_{{\mathrm{even}}}^d(T)$, as the
maximal depth-$d$ alternating
$\lceil\frac{k+1}{2}\rceil$-core of $\G$ with respect
to $\us(T)$ restricted to~$\V^d$, such that $v$ is even.
%Here, $\us^d(T)$ is the restriction of $\us(T)$, to $\V^d$.
We similarly define $\widehat{\mathcal{H}}_{{\mathrm{odd}}}(T)$ and
$\widehat{\mathcal{H}}_{{\mathrm{odd}}}^d(T)$.\vspace*{1pt}

We now\vspace*{1pt} proceed to define $\hEvC_{\mathrm{even}}(T)$, $\hEvC_
{\mathrm{even}}^d(T)$,
$\hEvC_{\mathrm{odd}}(T)$, $\hEvC_{\mathrm{odd}}^d(T)$, $\hEv
(\traj
^T)$ and
$ \hEvA_{\mathrm{even}}(T,\traj^T)$, $\hEvA_{\mathrm
{even}}^d(T,\traj^T)$,
$\hEvA_{\mathrm{odd}}(T,\traj^T)$, $\hEvA_{\mathrm{odd}}^d(T,\traj
^T)$ for $\G$, analogously to the definitions of $C_{\mathrm{even}}(T)$
etc. for $\G_\rroot$.
An analog of Lemma \ref{lemma:alt_children_seq} holds.

Define the probabilities
\begin{eqnarray*}
\Psih_{{\mathrm{even}},T}(\sigma^T) &=& \prob(\hEvA
_{\mathrm{even}}(T, \sigma^T )) ,\\
\Psih_{{\mathrm{even}},T}^d(\sigma^T)&=& \prob(\hEvA
_{\mathrm{even}}^d(T, \sigma^T)) ,\qquad d\ge0 .
\end{eqnarray*}
As before, we have $ \Psih_{{\mathrm{even}},T}^d(\sigma^T)$ is
nonincreasing in
$d$ and
%
%e67 ###
%
\begin{equation}
\Psih_{{\mathrm{even}},T}(\sigma^T) = \lim_{d\rightarrow\infty
}\Psih_{{\mathrm{even}},T}^d(\sigma^T) .
\end{equation}
We similarly define $\Psih_{{\mathrm{odd}},T}(\sigma^T), \Psih
_{{\mathrm{odd}},T}^d(\sigma^T)$ and have
$\Psih_{{\mathrm{odd}},T}^d(\sigma^T)$ converging to
$\Psih_{{\mathrm{odd}},T}(\sigma^T)$ as $d\to\infty$.
\begin{lemma}
\label{lemma:relate_rooted_complete}
The following identities are satisfied for all $d\ge0$:
%
%e70 ###
%e69 ###
%e68 ###
%
\begin{eqnarray}\label{eq:Alt_rc_fp1}
\Psih_{{\mathrm{odd}},T}^{d+1}(\sigma^T)
&=& \prob_0(\sigma(0))
\sum_{r=\lceil({k+1})/{2} \rceil}^{k}\pmatrix{k\cr
r} \nonumber\\[-1pt]
&&{}\times\sum_{\sigma_1^{T}\cdots\sigma_{k}^{T}}
\prod_{t=0}^{T-1}
\tK\bigl(\sigma(t+1)|\sigma_{\partial v}(t)\bigr)
\prod_{i=1}^{r}\Psi_{{\mathrm{even}},T}^d
(\sigma_i^{T}\midd \sigma^{T})\\[-1pt]
&&\hphantom{{}\times\sum_{\sigma_1^{T}\cdots\sigma_{k}^{T}}}
{}\times\prod_{i=r+1}^{k}\{\prob(\sigma_i^{T}\midd \sigma^{T}
)- \Psi_{{\mathrm{even}},T}^d(\sigma_i^{T}\midd \sigma^{T}
)\} ,\nonumber
\\[-1pt]
\label{eq:Alt_rc_fp2}
\Psih_{{\mathrm{even}},T}^{d+1}(\sigma^T)
&=& \ind\bigl(\sigma(T)=-1\bigr)
\prob_0(\sigma(0))
\sum_{r=\lceil({k+1})/{2} \rceil}^{k}\pmatrix{k\cr
r} \nonumber\\[-1pt]
&&{}\times\sum_{\sigma_1^{T}\cdots\sigma_{k}^{T}}
\prod_{t=0}^{T-1}
\tK\bigl(\sigma(t+1)|\sigma_{\partial v}(t)\bigr)
\prod_{i=1}^{r}\Psi_{{\mathrm{odd}},T}^d
(\sigma_i^{T}\midd \sigma^{T})\\[-1pt]
&&\hphantom{{}\times\sum_{\sigma_1^{T}\cdots\sigma_{k}^{T}}}
{}\times
\prod_{i=r+1}^{k}\{\prob(\sigma_i^{T}\midd \sigma^{T}
)- \Psi_{{\mathrm{odd}},T}^d(\sigma_i^{T}\midd \sigma^{T})\} ,
\nonumber\\[-12pt]\nonumber
\end{eqnarray}
\begin{eqnarray}
\tK(\cdots)
&\equiv&\cases{
\displaystyle \ind\Biggl\{\sigma(t+1) = \sign\Biggl(\sum_{i=1}^{k}\sigma_i(t)
\Biggr)\Biggr\}, &\quad if $\displaystyle \sum_{i=1}^{k}\sigma_i(t)\neq0 $,\vspace*{2pt}\cr
\dfrac{1}{2}, &\quad otherwise.}
\end{eqnarray}
\end{lemma}
\begin{pf}
The proof is
very similar to the one of Lemma \ref{lemma:PsiRecursion} (in Appendix
\ref{app:ExactCavityRec}),
and we omit it for the sake of space.
\end{pf}
\begin{lemma}
\label{lemma:theta_lb_reinterpreted}
Assume that $\bar{\Psi}_{{\mathrm{odd}},T} \succ0$
for some $T\ge0$ and $\theta\in[0,1]$. Then for the same $\theta$ and
$T$, there exists an alternating $\lceil\frac{k+1}{2} \rceil$-core
of $\G$ with positive probability with
respect to $\us(T)$.
\end{lemma}
\begin{pf}
Take the limit $d \rightarrow\infty$ in (\ref{eq:Alt_rc_fp2}).
We have,
%
%e71 ###
%
\begin{eqnarray}\label{eq:Alt_rc_asymp}
&&\Psih_{{\mathrm{even}},T}(\sigma^T)\nonumber\\[-1pt]
&&\qquad= \ind\bigl(\sigma(T)=-1\bigr) \prob
_0(\sigma(0))\nonumber\\[-8pt]\\[-8pt]
&&\qquad\quad{}\times
\sum_{r=\lceil({k+1})/{2} \rceil}^{k}\pmatrix{k\cr
r} \sum_{\sigma_1^{T}\cdots\sigma_{k}^{T}}
\prod_{t=0}^{T-1}
\tK\bigl(\sigma(t+1)|\sigma_{\partial v}(t)\bigr)
\prod_{i=1}^{r}\Psi_{{\mathrm{odd}},T}
(\sigma_i^{T}\midd \sigma^{T})
\nonumber\\[-1pt]
&&\qquad\quad\hphantom{{}\times
\sum_{r=\lceil({k+1})/{2} \rceil}^{k}\pmatrix{k\cr
r} \sum_{\sigma_1^{T}\cdots\sigma_{k}^{T}}}
{}\times\prod_{i=r+1}^{k}\{\prob(\sigma_i^{T}\midd \sigma^{T}
)- \Psi_{{\mathrm{odd}},T}(\sigma_i^{T}\midd \sigma^{T})
\} .\nonumber
\end{eqnarray}
Now, consider any $\theta$ such that $\bar{\Psi}_{{\mathrm{odd}},T}
\succ0$.
Consider $\Psih_{{\mathrm{even}},T}(\sigma^T)$ for any $\sigma^T$
with $\sigma(T)=-1$.
Note that every term in the summation over $r$ in (\ref{eq:Alt_rc_asymp})
is nonnegative, and, in fact, positive when $\bar{\Psi}_{\mathrm
{odd},T} \succ0$ holds.
Hence,\vspace*{1pt} $\Psih_{{\mathrm{even}},T}(\sigma^T)>0 \Rightarrow\prob
_\theta(\exists\mbox{ alternating } \lceil\frac{k+1}{2}
\rceil\mbox{-core } \HH\mbox{ of } \G\mbox{ with
respect to } \us(T) \mbox{ s.t. } v \in\HH)> 0$.
\end{pf}

The lower bound on $\theta_*(k)$ is an immediate consequence of the above
lemmas.
\begin{pf*}{Proof of Theorem \ref{thm:LowerBound}}
The thesis follows Lemmas \ref{lemma:alt_core_persistence}, \ref
{lemma:PsiRecursion} and \ref{lemma:theta_lb_reinterpreted}, the
definition of $\theta_*$ in (\ref{eq:ThresholdDef}) and
(\ref{eq:psi_odd_limit_proved}).
\end{pf*}

%
%**************************************************
%
%s5.3 ###
\subsection{Evaluating the lower bound}

Equations (\ref{eq:Alt_fp1}) and (\ref{eq:Alt_fp2}) can be iterated
with initial values given by (\ref{eq:initial_psi}) to compute
$\theta_{\mathrm{lb}}(k,T)$.
To simplify the recursion, we notice that the dynamics is ``bipartite'':
each of $\coinset$
and $\us(0)$ can be partitioned $\coinset= (\widehat{\coinset
},\widetilde
{\coinset}), \us(0) = (\widehat{\us}(0),\widetilde{\us}(0))$
such that $(\widehat{\coinset},\widehat{\us}(0)$ and $(\widetilde
{\coinset
},\widetilde{\us}(0))$ never ``interact''
in the majority dynamics on an infinite tree. This remark reduces the
number of variables in the recursions equations (\ref{eq:Alt_fp1})
and (\ref{eq:Alt_fp2}). Further, for small values of $T$, instead of summing
over all possible trajectories of children, it is faster to sum over
all possibilities
for the histogram of the trajectories followed by children.
%the `histogram' corresponds to k-1 balls dropped into 2^(T/2) bins. I
%did not use
%empirical distributions in my computations

%In Table \ref{table:lower_bounds}, we present some of
%the lower bounds $\theta_{\mathrm{lb}}(k,T)$ computed through this
%approach, and compare them with the empirical threshold $\theta_*(k)$
%deduced from numerical simulations.

In Table \ref{table:lower_bounds}, we present some of the lower bounds
$\theta_{\mathrm{lb}}(k,T)$ computed through this approach, and compare
them with the empirical threshold $\theta_{*,\mathrm{rgraph}}(k)$
deduced from numerical simulations (cf. Section
\ref{sec:numerical_results}). In the same table, we present the large
$k$ asymptotic behavior of $\theta_{\mathrm{lb}}(k,T)$ for fixed $T$.

%
%t4 ###
%
\begin{table}[b]
\caption{Computed lower bound values
$\theta_{\mathrm{lb}(k,T)}$}\label{table:lower_bounds}\vspace*{-2pt}
\begin{tabular*}{\tablewidth}{@{\extracolsep{\fill}}
l@{\hspace*{-0pt}}c@{\hspace*{-2pt}}k{2.5}k{2.6}k{2.7}
k{2.6}@{\hspace*{-2pt}}c@{\hspace*{-2pt}}c@{}}
\hline
& & \multicolumn{4}{c}{$\bolds{T}$} & &\\[-4pt]
& & \multicolumn{4}{c}{\hrulefill\hspace*{-6pt}} & &\\
\hphantom{...\hspace*{1pt}}$\bolds{k}$ &
& \multicolumn{1}{c}{$\bolds{0}$\hspace*{2pt}}
& \multicolumn{1}{c}{$\bolds{1}$\hspace*{2pt}}
& \multicolumn{1}{c}{$\bolds{2}$\hspace*{2pt}}
& \multicolumn{1}{c}{$\bolds{3}$\hspace*{2pt}}
& & \multicolumn{1}{c@{}}{\multirow{2}{85pt}[12pt]{{\centering\textbf{\textit{Simulation
threshold}}
$\hphantom{000}\hspace*{3pt}\bolds{\theta_{*,\mathrm{rgraph}}(k)}$}}\hspace*{-6pt}}\\
\hline
\hphantom{000}3 & \multirow{15}{2pt}{\hspace*{1pt}\rule{0.5pt}{158.5pt}
\hspace*{-6pt}}
&
\bolds{+0},\bolds{.508}  &
\bolds{+0},\bolds{.568} & \bolds{+0},\bolds{.572} &
\bolds{+0},\bolds{.574} &
\multirow{15}{2pt}{\hspace*{-0pt}$\matrix{\vspace*{-12pt}
\cr\vdots\vspace*{-4pt}
\cr\vdots\vspace*{-4pt}
\cr\vdots\vspace*{-4pt}
\cr\vdots\vspace*{-4pt}
\cr\vdots\vspace*{-4pt}
\cr\vdots\vspace*{-4pt}
\cr\vdots\vspace*{-4pt}
\cr\vdots\vspace*{-4pt}
\cr\vdots\vspace*{-4pt}
\cr\vdots\vspace*{-4pt}
\cr\vdots\vspace*{-4pt}
\cr\vdots\vspace*{-4pt}
\cr\vdots\vspace*{-5pt}
\cr\cdot}$\hspace*{-2pt}}
& \textit{0.58}\hphantom{0}\\
\hphantom{000}5 &  & -0,.084 & \bolds{+0},\bolds{.026} & \bolds{+0},\bolds{.048} &
\bolds{+0},\bolds{.052} & &\textit{0.054}\\
\hphantom{000}7 &  & -0,.14 & -0,.020 &
\bolds{+0},\bolds{.002} & \bolds{+0},\bolds{.008} & & \textit{0.010}\\
\hphantom{000}9 &  & -0,.14 & -0,.030 & -0,.006 & -0,.0008 &\\
\hphantom{00}11 &  & -0,.12 & -0,.028 & -0,.010 & -0,.0028 & \\
\hphantom{00}15 &  & -0,.12 & -0,.024 & -0,.008 & -0,.0028 & \\
\hphantom{00}21 &  & -0,.084 & -0,.018 & -0,.0054 & -0,.0018& \\
\hphantom{00}31 &  & -0,.080 & -0,.014 & -0,.0032 & -0,.0010& \\
\hphantom{00}51 &  & -0,.046 & -0,.0070 & -0,.0014 & -0,.00038 &\\
\hphantom{0}101 &  & -0,.026 & -0,.0032 & -0,.00048 &&\\
\hphantom{0}201 &  & -0,.016 & -0,.0014 & -0,.00014 &&\\
\hphantom{0}401 &  & -0,.0084 & -0,.00048 & -0,.000040 &&\\
1001 &  & -0,.0035 & -0,.00012 & -0,.000008 &&\\
[4pt]
\textit{Asymptotics}%\hspace*{4pt}}
&  & \multicolumn{1}{c}{$-\Theta(\frac{\sqrt{\log k}}{k})$}
& \multicolumn{1}{c}{$-\Theta(\frac{\sqrt{\log k}}{k^{3/2}})$}
& \multicolumn{1}{c}{$-\Theta(\frac{\sqrt{\log
k}}{k^2})$}&&\\[3pt]
\hline
\end{tabular*}
\end{table}

%As observed in the \hyperref[intro]{Introduction} $\theta_*(k)\ge0$
%by symmetry and
%monotonicity. Therefore the lower bounds are nontrivial only if
%$\theta_{\mathrm{lb}}(k,T)>0$. It turns out that for any fixed $T$,
%$\theta_{\mathrm{lb}}(k,T)$ becomes negative at large $k$.
%We present in the same table the asymptotic behaviors.
%Nevertheless, for $k\le7$, our lower bounds provide good estimates of the
%actual threshold.

As observed in the \hyperref[intro]{Introduction}, $\theta_*(k)\ge 0$
by symmetry and monotonicity. Therefore, the lower bounds are
nontrivial only if $\theta_{\mathrm{lb}}(k,T)>0$. It turns out that for
any fixed $T$, $\theta_{\mathrm{lb}}(k,T)$ becomes negative at large
$k$. Nevertheless, for $k\le 7$, our lower bounds are positive and
closely approximate $\theta_{*,\mathrm{rgraph}}(k)$, indicating that the
bounds may provide good estimates of $\theta_*(k)$.

The values of $\theta_{\mathrm{lb}}(k,T)$ are much lower for even
values of
$k$.
For example, for $k=4$, $6$, $8$, $\theta_{\mathrm{lb}}(k,3)
\approx-0.22$, $-0.09$, $-0.05$, respectively.
This is as expected, since our requirement of an alternating
$\lceil\frac{k+1}{2} \rceil$-core is more stringent for
even $k$.
On the other hand, numerical simulations suggest
that $\theta_*(k)=0$ for small even values of $k$.

\begin{appendix}

%s6 ###
\section{Proofs of preliminary results}\label{app:prelim}

This section presents the proofs of Lemmas \ref{lemma:LessThanOne}
and \ref{lemma:Local}, with some auxiliary results proved in the
second subsection.
%
%*********************************************************
%
%s6.1 ###
\subsection{Proofs}

\mbox{}

\begin{pf*}{Proof of Lemma \ref{lemma:LessThanOne}}
Consider the subgraph $\G_+$ of $\G$ induced by
vertices $i\in\V$, such that $\sigma_i(0)=+1$:
each vertex belongs to this subgraph independently with probability
$(1+\theta)/2$.
Let $\G_{+,q}$ be the maximal subgraph of $\G_+$ with minimum degree
$q=k-\lfloor(k+1)/2\rfloor+1$. It is clear that no vertex in $\G
_{+,q}$ ever
flips to $-1$ under the majority process. Consider a modified
initial condition such that $\sigma_i(0) = +1$ for
$i\in\G_{+,q}$, and $\sigma_i(0) = -1$ otherwise. By monotonicity
of the dynamics, it is sufficient to show that such a modified
initial condition converges to $\up$ under the majority process.

Notice that $\HH=\G\setminus\G_{+,q}$ is the subgraph induced by the
final set of a bootstrap percolation process with initial density
$\rho= (1-\theta)/2$ and threshold $\lfloor(k+1)/2\rfloor$ (a~vertex
joins if at least $\lfloor(k+1)/2\rfloor$ of its neighbors have
joined). It is proven in~\cite{FontesTree}, Theorem 1.1, that there
exists $\rho_{\mathrm{c}}(k)>0$ such that, for
$\rho<\rho_{\mathrm{c}}(k)$, $\HH$ is almost surely the disjoint union
of a of countable number of finite trees. This implies the thesis.
Indeed, we can restrict our attention to any such finite tree occupied
by $-1$, and surrounded by $+1$ elsewhere. On such a tree, the set of
vertices such that of $\sigma_i(t) = -1$ never increases, and at least
one vertex quits the set at each iteration. Therefore, any such tree
turns to $+1$ in finitely many iterations.
\end{pf*}
\begin{pf*}{Proof of Lemma \ref{lemma:Local}}
Let $\G_n = ([n],\Ed_n)$ be a random graph of degree $k$ over
$n$ vertices distributed according to the configuration model.
We recall that a graph is generated with this distribution by
attaching $k$ labeled half-edges to each vertex $i\in[n]$
and pairing them according to a uniformly random matching among $nk$ objects.

The proof of Lemma \ref{lemma:Local} is based on the analysis
of the majority process on the graph $\G_n$.
We will denote by $\prob_{\theta,n}$ the law of this process
when the spins $\{\sigma_i(0)\}_{i\in[n]}$ are initialized to
i.i.d. random variables with $\E_{\theta,n}\{\sigma_i(0)\} = \theta$.
We use the following auxiliary results.
\begin{lemma}\label{lemma:GraphTree}
For any fixed $i\in\naturals$, $j\in\V$ and $t\ge0$ we have
%
%e72 ###
%
\begin{equation}
\lim_{n\to\infty}\E_{\theta,n}\{\sigma_i(t)\} =
\E_{\theta}\{\sigma_j(t)\} .
\end{equation}
\end{lemma}
\begin{lemma}\label{lemma:Concentration}
Let $\{\sigma_i(t)\}_{i\in[n], t\in\Z_+}$ be distributed according to
the majority process on $\G_n$, and define $\const(k,t) \equiv
4(t+1)(k^{t+1}-1)^2/(k-1)^2$.
Then
%
%e73 ###
%
\begin{equation}
\prob_{\theta,n}\Biggl\{
\Biggl|\sum_{i=1}^n\sigma_i(t)-n\E_{\theta,n}\sigma_1(t)
\Biggr|
\ge n\ve\Big| \G_n\Biggr\}\le2 \exp\biggl\{-\frac{n\ve
^2}{2\const(k,t)}\biggr\}.
\end{equation}
\end{lemma}
\begin{lemma}\label{lemma:Attractivity}
There exists $\delta_*$, $k_*>0$ such that for any $k\ge k_*$
there is a set $\SSS_{k,n}$ of ``good graphs'' such that
$\prob\{\G_n\in\SSS_{k,n}\}\to1$, and
the following happens. For any $\G_n\in\SSS_{k,n}$ and any
initial condition $\{\sigma_i(0)\}_{i\in[n]}$ on the vertices of
$\G_n$ with $\sum_{i=1}^n\sigma_i(0)\ge n(1-2\delta_*/k)$, we have
%
%e74 ###
%
\begin{equation}
\sum_{i=1}^n\bigl(1-\sigma_i(1)\bigr)\le\frac{3}{4} \sum_{i=1}^n\bigl(1-\sigma
_i(0)\bigr) .
\end{equation}
\end{lemma}

Let us now turn to the actual proof. Choose $\delta_*$ and
$k_*$ as per Lemma \ref{lemma:Attractivity} and assume $k\ge k_*$.
By assumption, there exists a time $t_*$ such that
$\E_{\theta}\{\sigma_i(t_*)\}\ge1-\delta_*/k$. By Lemmas
\ref{lemma:GraphTree} and \ref{lemma:Concentration},
for all $n$ large enough we have
%
%e75 ###
%
\begin{equation}
\prob_{\theta,n}\Biggl\{\sum_{i=1}^n\sigma_i(t_*)\ge
n\biggl(1-2\frac{\delta}{k}\biggr)\Biggr\}
\ge1-e^{-Cn} .
\end{equation}

Assume $\sum_{i=1}^n\sigma_i(t_*)\ge n(1-2\frac{\delta
_*}{k})$ and $\G_n\in\SSS_{k,n}$.
Then, by Lemma \ref{lemma:Attractivity}, and any $t\ge t_*$ we have
%
%e76 ###
%
\begin{equation}
\sum_{i=1}^n\bigl(1-\sigma_i(t)\bigr)\le n (3/4)^{t-t_*} .
\end{equation}
Combining this with the above remarks, and using the
symmetry of the graph distribution with respect to permutation of the vertices,
we get
%
%e77 ###
%
\begin{equation}
\prob_{\theta,n}\{\sigma_1(t)\neq+1\}
\le2(3/4)^{t-t_*}+\prob\{\G_n\notin\SSS_{k,n}\}
+e^{-Cn} .
\end{equation}
By Lemma \ref{lemma:GraphTree}, this implies
$\prob_{\theta}\{\sigma_i(t)\neq+1\}
\le5(3/4)^{t-t_*}$ which, by Borel--Cantelli implies
$\sigma_i(t)\to+1$ almost surely, whence the thesis follows.
\end{pf*}

\subsection{Proofs of auxiliary results}

\mbox{}

\begin{pf*}{Proof of Lemma \ref{lemma:GraphTree}} Fix a vertex $i$ in
$\G_n$,
and denote by $\Ball_i(t)$ the subgraph induced by vertices
whose distance from $i$ is at most $t$.
The value of $\sigma_{i}(t)$ only depends on $\G_n$ through the
$\Ball_i(t)$. If $\Ball_i(t)$ is a $k$-regular tree of depth $t$
[to be denoted by $\Tree(t)$] then the distribution of $\sigma_j(t)$
is the same that would be obtained on $\G$, whence
\[
|\E_{\theta,n}\{\sigma_i(t)\}-\E_{\theta}\{\sigma_j(t)\}
|
\le2 \prob_{\theta,n}\{\Ball_i(t)\not\simeq\Tree(t)\} .
\]
The thesis follows since
$\prob_{\theta,n}\{\Ball_i(t)\not\simeq\Tree(t)\}\le\const^t/n$
for some constant
$\const$ (dependent only on $k$).
\end{pf*}
\begin{pf*}{Proof of Lemma \ref{lemma:Concentration}}
Let $X_n(t) \equiv\sum_{i=1}^n\sigma_i(t)$. This is a
deterministic function of the $n(t+1)$ bounded\vspace*{1pt} random variables
$\{\sigma_i(0)\}_{i\in[n]}$ and of $\{A_{i,s}\}_{i\in[n], s\le t}$.
Further, it is a Lipschitz function with constant
$\widehat{\const}(k,t)\le2(k^{t+1}-1)/(k-1)$, because any change in
$\sigma_i(0)$, or
$A_{i,s}$ only influences the values $\sigma_j(t)$ within a ball of radius
$t$ around $i$. By the Azuma--Hoeffding inequality,
%
%e78 ###
%
\begin{equation}
\prob_{\theta,n}\{|X_n(t)-\E_{\theta,n}X_n(t)|\ge\Delta
\}
\le2 \exp\biggl\{-\frac{\Delta^2}{2n(t+1)\widehat{\const
}^2}\biggr\} ,
\end{equation}
which implies the thesis.
\end{pf*}
\begin{pf*}{Proof of Lemma \ref{lemma:Attractivity}}
Although the proof follows from a standard expansion argument, we
reproduce it here for the convenience of the reader.

Recall that a graph $\G_n$ over $n$ vertices is a $(k(1-\ve),\delta/k)$
(vertex) expander if each subset $\W$ of at most $n\delta/k$ vertices
is connected to at least $k(1-\ve)|\W|$ vertices in the rest of the graph.
It is known that there exists $\delta_*>0$ such that, for all $k$
large enough,
a random $k$ regular graph is, with high probability,
a $(3k/4,\delta_*/k)$ expander \cite{Expander}.
We let $\SSS_{k,n}$ be the set of $k$-regular graphs
$\G_n$ that are $(3k/4,\delta_*/k)$ expanders.

Let $\W$ be the set of vertices $i\in[n]$ such that $\sigma_i(0)=-1$.
By hypothesis, $|\W|\le n\delta/k$. Denote by $n_-$ the number of
vertices in
$[n]\setminus\W$ that have at least $\lceil k/2\rceil$ neighbors in
$\W$ [and hence, such that potentially $\sigma_i(1)=-1$], and by $n_+$
the set of vertices that have between $1$ and $\lceil k/2\rceil-1$
neighbors in $\W$. Further, let $l$ be the number of edges
between vertices in $\W$. Then
\[
\biggl\lceil\frac{k}{2}\biggr\rceil n_- + n_+ + 2l \le k |\W| ,\qquad
n_-+n_+ \ge\frac{3}{4}k |\W| ,
\]
where the first inequality follows by edge-counting and the second
by the expansion property.
By taking the difference of these inequalities, we get
\[
\biggl(\biggl\lceil\frac{k}{2}\biggr\rceil-1\biggr) n_- + 2l \le
\frac{k}{4} |\W| .
\]
Let $\W'$ be the set of vertices such that $\sigma_i(1) = -1$.
Thus $\W'$ is contained in the set of vertices with at
least $\lceil k/2\rceil$ neighbors in $\W$ whence
$|\W'| \le n_-\le n_-+ (2l)/(\lceil k/2\rceil-1)$,
and therefore
\[
|\W'|\le\frac{k}{4(\lceil k/2\rceil-1)} |\W|,
\]
which yields the thesis.\vadjust{\goodbreak}
\end{pf*}

%
%************************************************************
%
%s7 ###
\section{Proof of the exact cavity recursion}\label{app:ExactCavityRec}

\vspace*{-8pt}

\begin{pf*}{Proof of Lemma \ref{lemma:ExactCavity}}
Throughout the proof, we denote the neighbors of the root as $\{
1,\ldots
,k-1\}$.
Let $\us(0)$ be the vector of initial spins of the root and all the
vertices up to a distance $T$ from the root.
For each $i\in\{1,\ldots, k-1\}$, let
$\us_i(0)$ be the vector of initial spins of the sub-tree rooted at
$i$, and not including the root, and up to the same distance $T$ from
the root. Clearly, if we choose an appropriate ordering, we have
$\us(0)=(\sigma_\rroot(0),\us_1(0),\us_2(0), \ldots, \us_{k-1}(0))$.
Finally, we denote by $\coinset^T$ the set of coin flips
$\{\omega_{i,t}\}$ with $t\le T$, and $i$ at distance at most $T$ from
the root.
As above, we have $\coinset^T = (\omega_{\rroot}^T, \coinset
_{1}^{T},\ldots,\coinset_{k-1}^{T})$,
where $\coinset_{i}^{T}$ is the subset of coin flips in the subtree
rooted at
$i\in\{1,\ldots,k-1\}$.
By definition, the trajectory $\sigma_\rroot^{T+1}$
is a deterministic function of $\us(0)$, $u^{T+1}$ and $\coinset^T$.
We shall\vspace*{1pt} denote this function by $\cF$ and write
$\sigma_{\rroot}^{t}= \cF^{t}(\us(0),u^{t+1},\coinset^t)$.
This function is uniquely determined by the update rules.
We shall write the latter as
%
%e79 ###
%
\begin{equation}
\sigma_{\rroot}(t+1) = f(\sigma_{\rroot}(t),\us_{\partial\rroot}(t)
,u(t),\omega_{\rroot,t}) .
\end{equation}
We have therefore
%
%e80 ###
%
\begin{eqnarray}\label{eq:recursion_defn}
&&\prob(\sigma_\rroot^{T+1}=\traj^{T+1}\midd u^{T+1}) \nonumber\\[-8pt]\\[-8pt]
&&\qquad= \E
_{\coinset^T}\sum_{\us(0)}
\prob(\us(0)) \ind\bigl(
\traj^{T+1} = \cF^{T+1}(\us(0),u^{T+1},\coinset^T)\bigr) .\nonumber
\end{eqnarray}
Now we analyze each of the terms appearing in this sum.
Since the initialization is i.i.d., we have
%
%e81 ###
%
\begin{equation}\label{eq:recursion_prob}
\prob(\us(0))=
\prob_0(\sigma_\rroot(0))\prob(\us_1(0)) \prob(\us_2(0)) \cdots
\prob(\us_{k-1}(0)) .
\end{equation}
Further since the coin flips $\omega_{i,t}$ and $\omega_{j,t'}$ are
independent for
$i\neq j$, we have
%
%e82 ###
%
\begin{equation}\label{eq:recursion_prob2}
\E_{\coinset^T}\{\cdots\} = \E_{\omega_{\rroot}^T} \E_{\coinset
_{1}^{T}}\cdots
\E_{\coinset_{k-1}^{T}}\{\cdots\} .
\end{equation}
Finally, the function $\cF^{T+1}(\cdots)$ can be decomposed as follows:
%
%e83 ###
%
\begin{eqnarray}\label{eq:recursion_ind}\quad
&&\ind\bigl(\traj^{T+1} = \cF^{T+1}(\us(0),u^{T+1},\coinset
^T)\bigr)\nonumber\\
&&\qquad=
%)
%
\ind\bigl( \sigma_\rroot(0)=\traj(0) \bigr)\nonumber\\[-8pt]\\[-8pt]
&&\qquad\quad{}\times \sum_{\sigma_1^{T}\cdots\sigma_{k-1}^T}
\prod_{t=0}^{T}\ind\bigl(\traj(t+1) = f(\sigma_{\rroot}(t),\us
_{\partial\rroot}(t)
,u(t),\omega_{\rroot,t}) \bigr)\nonumber\\
&&\hphantom{{}\times\sum_{\sigma_1^{T}\cdots\sigma_{k-1}^T}}
\qquad\quad\times
\prod_{i=1}^{k-1}\ind\bigl(\sigma_i^{T}
=
\cF^{T}(\us_i(0),\traj^{T},\coinset_{i}^{T-1})\bigr).\nonumber\vadjust{\goodbreak}
\end{eqnarray}

Using (\ref{eq:recursion_prob}), (\ref{eq:recursion_prob2})
and (\ref{eq:recursion_ind}) in (\ref{eq:recursion_defn})
and separating terms that depend only on $\us_i(0)$, we get
\begin{eqnarray*}
&&\prob(\sigma_\rroot^{T+1} = \traj^{T+1}\midd u^{T+1})\\
&&\qquad =
\prob(\traj(0))
\sum_{\sigma_1^{T}\cdots\sigma_{k-1}^{T}}
\prod_{t=0}^T \ind\bigl(\traj(t+1) = f(\sigma_{\rroot}(t),\us
_{\partial\rroot}(t)
,u(t),\omega_{\rroot,t}) \bigr) \\
&&\qquad\quad\hphantom{\prob(\traj(0))
\sum_{\sigma_1^{T}\cdots\sigma_{k-1}^{T}}}
{}\times\prod_{i=1}^{k-1} \sum_{\us_i(0)} \prob(\us_i(0))
\ind\bigl(
\sigma_i^{T}
= \cF^{T}(\us_i(0),\traj^{T},\coinset_{i}^{T-1})
\bigr) .
\end{eqnarray*}
The recursion equation (\ref{eq:Recursion}) follows.
\end{pf*}
\begin{pf*}{Proof of Lemma \ref{lemma:PsiRecursion}}
%We should include a section on notation!]}
Consider any $d \ge0$. We denote the neighbors of the root as $\{
1,\ldots,k-1\}$. We reuse the definitions of
$\us(0)$ and $\us_i(0)$ for $1 \le i \le(k-1)$ from Lemma \ref
{lemma:ExactCavity}, with depth $T$ replaced with depth $(T+d+1)$.
We denote by $\coinset^{T-1}$ the set of coin flips
$\{\omega_{i,t}\}$ with $t\le T-1$, and $i$ at distance at most
$T+d+1$ from the root.
We have $\coinset^{T-1} = (\omega_{\rroot}^{T-1}, \coinset
_{1}^{T-1},\ldots,\coinset_{k-1}^{T-1})$,
where $\coinset_{i}^{T-1}$ is the subset of coin flips in the subtree
rooted at
$i\in\{1,\ldots,k-1\}$.
Let $\G_i$ be the subtree rooted at $i$.
Define $\HH_{i,{\mathrm{even}}}^d(T)$, as the maximal depth-$d$ rooted
alternating
$\lceil\frac{k+1}{2}\rceil$-core of $\G_i$ with
respect to $\us_i(T)$,\footnote{More precisely, we take only the
restriction of $\us_i(T)$ up to depth $d$ in defining $\HH_{i,\mathrm
{even}}^d(T)$.} such that $i$ is even.
Define $\EvC_{i,{\mathrm{even}}}^d(T)= \{\rroot\in\V_{\HH
_{i,{\mathrm{even}}}(T)}^d \}$. Let $\Ev_i(\traj^T) \equiv\{\sigma
_i(t)=\traj(t), 0 \le t\le T \}$.
We define $\EvA_{i,{\mathrm{even}}}^d(T, \traj^T)=\EvC_{i,\mathrm
{even}}^d(T) \cap\Ev_i(\traj^T)$.
Hence,\vspace*{1pt} we have mirrored the definitions for
the root $\rroot$ at the child $i$.

Let $\C=\{1,2,\ldots,k-1\}$. By Definition \ref
{def:rooted_partial_alt_r_core}, it follows that (here $\EvA
^\complement$ denotes the complement of an
event $\EvA$)
%
%e84 ###
%
\begin{eqnarray}\label{eq:lb_child_cores}\qquad
\EvC_{{\mathrm{odd}}}^{d+1}(T)&=& \mathop{\bigcup_{\SSS\subseteq\C
}}_{|\SSS| \ge\lceil(k-1)/2 \rceil}
\bigcap_{i \in\SSS} \EvC_{i,{\mathrm{even}}}^d(T) \bigcap_{j \in
\C
-\SSS} (\EvC_{j,{\mathrm{even}}}^d(T))^\complement
,\nonumber\\[-2pt]
\EvC_{{\mathrm{even}}}^{d+1}(T)&=& \ind\bigl(
\sigma_\rroot(T)=-1\bigr)\\[-2pt]
&&{}\times\mathop{\bigcup_{\SSS\subseteq\C}}_{|\SSS| \ge\lceil(k-1)/2
\rceil}
\bigcap_{i \in\SSS} \EvC_{i,{\mathrm{odd}}}^d(T) \bigcap_{j \in
\C
-\SSS} (\EvC_{j,{\mathrm{odd}}}^d(T))^\complement.\nonumber
\end{eqnarray}

Let $\J_{{\mathrm{odd}}}^d(\us(0),u^T,\coinset^{T}) \equiv\ind
(\EvC
_{{\mathrm{odd}}}^d(T))$. When $\us(0)$ is passed as an argument to
$\J
_{{\mathrm{odd}}}^d$, we implicitly mean that only the restriction of
$\us(0)$
to depth $T+d$ from the root is under consideration.
Note that $\J_{{\mathrm{odd}}}^d$
is a deterministic function. Similarly define $\J_{{\mathrm{even}}}^d$.
From (\ref{eq:lb_child_cores}), we have
%
%e85 ###
%
\begin{eqnarray}\label{eq:lb_child_core_fn}\quad
&&\J_{{\mathrm{odd}}}^{d+1}(\us(0),u^T,\coinset^{T})\nonumber\\[-2pt]
&&\qquad= \mathop{\sum
_{\SSS\subseteq\C}}_{|\SSS| \ge\lceil(k-1)/2 \rceil}
\prod_{i \in\SSS} \J_{i,{\mathrm{even}}}^d(\us_i,\sigma_\rroot
^T,\coinset_{i}^{T})\\[-2pt]
&&\qquad\quad\hphantom{\mathop{\sum
_{\SSS\subseteq\C}}_{|\SSS| \ge\lceil(k-1)/2 \rceil}}
{}\times\prod_{j \in\C-\SSS} \bigl(1-\J_{j,{\mathrm{even}}}^d(\us
_j,\sigma_\rroot^T,\coinset_{j}^{T})\bigr).\nonumber
\end{eqnarray}

Define\vspace*{1pt} $f(\cdot, \cdot,\cdot,\cdot)$ and $\cF(\cdot,\cdot,\cdot
)$ as in the proof of Lemma \ref{lemma:ExactCavity} (cf. Appendix \ref
{app:ExactCavityRec}).
We have
$\ind(\EvA_{{\mathrm{odd}}}^{d+1}(T,\traj^T)) = \ind(\Ev(\traj^T))
\ind(\EvC_{{\mathrm{odd}}}^{d+1}(T))$,
leading to
%
%e86 ###
%
\begin{eqnarray}\label{eq:lb_recursion_defn}
\Psi_{{\mathrm{odd}}}^{d+1}(\traj^{T}\midd u^{T})
&=& \E
_{\coinset^{T-1}}\sum_{\us(0)}
\prob(\us(0)) \ind\bigl(
\traj^{T} = \cF^{T}(\us(0),u^{T},\coinset^{T-1})\bigr)\nonumber\\[-10pt]\\[-10pt]
&&\hphantom{\E
_{\coinset^{T-1}}\sum_{\us(0)}}
{}\times\J
_{{\mathrm{odd}}}^{d+1}(\us(0),u^{T},\coinset^{T+d}) .\nonumber
\end{eqnarray}
Subtracting (\ref{eq:lb_recursion_defn}) from (\ref
{eq:recursion_defn}) after replacing $T+1$ by $T$,
we get
%
%e87 ###
%
\begin{eqnarray}\label{eq:lb_recursion_diff}
&&\prob(\traj^{T}\midd u^{T})-\Psi_{{\mathrm{odd}}}^{d+1}
(\traj^{T}\midd u^{T}) \nonumber\\
&&\qquad= \E_{\coinset^{T-1}}\sum_{\us(0)}
\prob(\us(0))\ind\bigl(
\traj^{T} = \cF^{T}(\us(0),u^{T},\coinset^{T-1})\bigr)\\
&&\hphantom{\E_{\coinset^{T-1}}\sum_{\us(0)}}\qquad\quad
{}\times\bigl(1-\J
_{{\mathrm{odd}}}^{d+1}(\us(0),u^{T},\coinset^{T+d})\bigr) .\nonumber
\end{eqnarray}
Equations (\ref{eq:recursion_prob}) and (\ref{eq:recursion_prob2})
(with $T$ replaced by $T-1$) continue to hold.
Using (\ref{eq:lb_child_core_fn}), we have the following
decomposition, similar to (\ref{eq:recursion_ind}):
%
%e88 ###
%
\begin{eqnarray}\label{eq:lb_recursion_ind}
&&\ind\bigl(\traj^{T} = \cF^{T}(\us(0),u^{T},\coinset^{T-1})\bigr)
\J_{{\mathrm{odd}}}^{d+1}(\us(0),u^{T},\coinset^{T+d})
\nonumber\\
&&\qquad=\ind\bigl( \sigma_\rroot(0)=\traj(0) \bigr)\nonumber\\
&&\qquad\quad{}\times\sum_{\sigma_1^{T}\cdots\sigma_{k-1}^{T}}
\prod_{t=0}^{T-1}\ind\bigl(\traj(t+1) = f(\sigma_{\rroot}(t),\us
_{\partial\rroot}(t),u(t),\omega_{\rroot,t}) \bigr)
\nonumber\\
&&\hphantom{{}\times\sum_{\sigma_1^{T}\cdots\sigma_{k-1}^{T}}}
\qquad\quad{}
\times\mathop{\sum_{\SSS\subseteq\C}}_{|\SSS| \ge\lceil(k-1)/2
\rceil}
\prod_{i \in\SSS} \ind\bigl(\sigma_i^{T} = \cF^{T}(\us
_i(0),\traj^{T},\coinset_{i}^{T-1})\bigr)\\
&&\hspace*{160.8pt}{}\times\J_{i,{\mathrm{even}}}^d(\us_i(0),\sigma_\rroot^T,\coinset_{i}^{T})
\nonumber\\
&&\hphantom{{}\times\sum_{\sigma_1^{T}\cdots\sigma_{k-1}^{T}}}
\qquad\quad\hspace*{65.1pt}{}
\times\prod_{j \in\C-\SSS} \ind\bigl(\sigma_j^{T} = \cF^{T}(\us
_j(0),\traj^{T},\coinset_{j}^{T-1})\bigr)\nonumber\\
&&\qquad\quad\hspace*{150.6pt}{}\times\bigl(1-\J_{j,{\mathrm{even}}}^d(\us_j(0),\sigma_\rroot^T,\coinset
_{j}^{T})\bigr)\nonumber.
\end{eqnarray}
Using (\ref{eq:recursion_prob}), (\ref{eq:recursion_prob2})
and (\ref{eq:lb_recursion_ind}) in (\ref{eq:lb_recursion_defn})
and separating terms that depend only on $\us_i(0)$, we get
\begin{eqnarray*}
&&\Psi_{{\mathrm{odd}}}^{d+1}(\traj^{T}\midd u^{T})\\
&&\qquad =
\prob(\traj(0))
\sum_{\sigma_1^{T}\cdots\sigma_{k-1}^{T}}
\prod_{t=0}^{T-1} \ind\{\traj(t+1) = f(\sigma_{\rroot
}(t),\us_{\partial\rroot}(t)
,u(t),\omega_{\rroot,t}) \} \\
&&\qquad\quad
{} \times\mathop{\sum_{\SSS\subseteq\C}}_{|\SSS| \ge
\lceil(k-1)/2 \rceil} \biggl\{ \prod_{i \in\SSS} \sum_{\us_i(0)}
\prob(\us_i(0)) \ind\bigl(\sigma_i^{T} =
\cF^{T}(\us_i(0),\traj^{T},\coinset_{i}^{T-1})\bigr)\\
&&\hphantom{}\hspace*{139.3pt}
{}\times\J_{i,{\mathrm{even}}}^d(\us_i(0),\sigma_\rroot^T,\coinset
_{i}^{T}) \\
&&\qquad\quad\hphantom{{} \times\mathop{\sum_{\SSS\subseteq\C}}_{|\SSS| \ge
\lceil(k-1)/2 \rceil}\biggl\{}
{} \times\prod_{j \in\C-\SSS} \sum_{\us_j(0)}
\prob(\us_j(0)) \ind\bigl(\sigma_j^{T} =
\cF^{T}(\us_j(0),\traj^{T},\coinset_{j}^{T-1})\bigr)\\
&&\hspace*{204.5pt}
{}\times\bigl(1-\J_{j,{\mathrm{even}}}^d(\us_j(0),\sigma_\rroot^T,\coinset
_{j}^{T})\bigr) \biggr\} .
\end{eqnarray*}
Using the ``even'' versions of (\ref{eq:lb_recursion_defn}) and
(\ref{eq:lb_recursion_diff}), and noticing the symmetry in the
expression between the $k-1$ children, we recover
equation (\ref{eq:Alt_fp1_redef}).

Equation (\ref{eq:Alt_fp2_redef}) follows similarly, with the additional
$\ind(\sigma_\rroot(T)=-1)$ term appearing due to the modification in
(\ref{eq:lb_child_cores}).
\end{pf*}

\vspace*{-15pt}

%s8 ###
\section{Proof of Lemma 4.3}
\label{app:RCexp}

Equation
(\ref{eq:CorrRecursion}) follows directly from (\ref{eq:cavity_def1})
and (\ref{eq:traj_prob}). We only
need to prove (\ref{eq:RespRecursion}).

Let $\I= \{0,\ldots,t\}$ and, for $\SSS\subset\I$, define the
rectangle $\R(\traj,\SSS,R,h)\subseteq\reals^{t+1}$ as the set of vectors
$\eta^{t}=(\eta(0),\ldots,\eta(t))$ such that
%
%e90 ###
%e89 ###
%
\begin{eqnarray}\qquad
\eta(r)+\sum_{s=0}^{r-1}R(r,s)\traj(s)+h(r)&=&0
\qquad\mbox{for all }
r \in\SSS,\\
\sign\Biggl(\eta(r)+\sum_{s=0}^{r-1}R(r,s)\traj(s)+h(r)
\Biggr)&=&\traj(r+1) \nonumber\\[-8pt]\\[-8pt]
&&\eqntext{\mbox{for all } r \in\I\setminus
\SSS.}
\end{eqnarray}
Equation (\ref{eq:asymp_dyn_def}) defines $\sigma(t+1)$ as a function of
$\sigma(0)$, $\eta^t$ and $h$. Let us denote this function by writing
$\sigma(t+1) = \Fs_{\sigma(t+1)}(\sigma(0),\eta^t;h)$:
\begin{eqnarray*}
\hspace*{-4pt}&&
\Ex_{C,R,h}[\sigma(t+1)] \\
\hspace*{-4pt}&&\qquad= \frac{1}{2}\sum_{\traj(0)\in\{\pm1\}}
\int_{\reals^{t+1}}
\Fs_{\sigma(t+1)}(\traj(0),\eta^t;h)
\phi_{0,C_t}(\eta^t) \prod_{i =0}^t\de\eta(i) \\
\hspace*{-4pt}&&\qquad= \frac{1}{2}\sum_{\traj^t\in\{\pm1\}^{t+1}} \int_{\reals
^{t+1}} \traj(t+1)
\phi_{0,C_t}(\eta^t)\\
\hspace*{-4pt}&&\hspace*{-30pt}\qquad\quad\hphantom{\frac{1}{2}\sum_{\traj^t\in\{\pm1\}^{t+1}} \int_{\reals
^{t+1}}}
{}\times\prod_{i=0}^t
\ind\Biggl\{\sign\Biggl(\eta(i)+\sum_{s=0}^{i-1}R(i,s)\traj(s)+h(i)
\Biggr)=\traj(i+1)
\Biggr\}\,
\de\eta(i) \\
\hspace*{-4pt}&&\qquad= \frac{1}{2} \sum_{\traj^t\in\{\pm1\}^{t+1}} \traj(t+1)
\int_{\R(\traj,\varnothing,R,h)} \phi_{0,C_t}(\eta^t) \prod_{i =0}^t
\de\eta(i) \\
\hspace*{-4pt}&&\qquad=\frac{1}{2} \sum_{\traj^t\in\{\pm1\}^{t+1}} \traj(t+1)
\Phi_{0,C_t}(\R(\traj,\varnothing,R,h)) .
\end{eqnarray*}
Since $C_t$ is strictly positive definite by Lemma
\ref{lemma:cavity_non_degenerate}, $x\mapsto\phi_{0,C_t}(x)$
is a continuous function.
By the fundamental theorem of calculus, we have
\begin{eqnarray*}
\frac{\partial\Phi_{0,C_t}(\R(\traj,\varnothing
,R,h))}{\partial h(s)}
\bigg|_{h=0} =
\cases{
\Phi_{0,C_t}(\R(\traj,\{s\},R,0)), &\quad if $\traj
(s+1)=+1$,\cr
- \Phi_{0,C_t}(\R(\traj,\{s\},R,0)), &\quad if $\traj(s+1)=-1$.}
\end{eqnarray*}
The definition of $R(t,s)$ for a cavity process in (\ref{eq:cavity_def2})
now leads to
\[
R(t+1,s) = \frac{1}{2}\sum_{\traj^{t+1} \in\{\pm1\}^{t+2}} \traj(t+1)
\traj(s+1) \Phi_{0,C_t}(\R(\traj,\{s\},R,0))
\]
for all $t\ge s\ge0$.
The result follows by the change $x_i' = x_i +\mu_i(\traj^t)$
in the Gaussian integral defining $\Phi$.

%s9 ###
\section{Proof of Lemma 4.7}
\label{app:qminusptstar}

Throughout the proof, we let $T=T_*-1$. Equation (\ref{eq:qminusp2})
continues to hold. We rewrite it as
%
%e92 ###
%e91 ###
%
\begin{eqnarray}
\label{eq:SumOverR}
\qprob(\sigma_\rroot^{T+1}\midd u^{T+1})-
\prob(\sigma_\rroot^{T+1}\midd u^{T+1}) &=& \frac{1}{2}
\sum_{r=1}^{k-1}\D(r,k)+O\bigl(k^{-(T_*+1)/2}\bigr) ,\\[-12pt]\nonumber
\end{eqnarray}
\begin{eqnarray}
\label{eq:RthTerm}
\D(r,k) &\equiv&
\pmatrix{k-1\cr r} \sum_{\sigma_1^{T}\cdots\sigma_{r}^{T}} \prod
_{i=1}^r \{\qprob(\sigma_i^{T}\midd\sigma_\rroot^T)-
\prob(\sigma_i^{T}\midd\sigma_\rroot^T)\}
\nonumber\\[-8pt]\\[-8pt]
&&{}
\times\sum_{\sigma_{r+1}^{T}\cdots\sigma_{k-1}^{T}} \prod_{i=r+1}^{k-1}
\prob(\sigma_i^{T}\midd \sigma_\rroot^{T}) \prod_{t=0}^T
\K_{u(t)}\bigl(\sigma_{\rroot}(t+1)|\sigma_{\droot}(t)\bigr)
.\nonumber
\end{eqnarray}

Let $r_0=\lfloor\log k\rfloor$. Split the summation over $r$ in
(\ref{eq:SumOverR}) into two parts: the first for $1\le r\le r_0$,
the second for $r_0< r\le k-1$. We will first show that the second part is
of order $o(k^{-1/2})$. Indeed, by
Lemma \ref{lemma:qminusp}, we know that
$\qprob(\sigma_i^{T}\midd\sigma_\rroot^T)- \prob
(\sigma_i^{T}\midd\sigma_\rroot^T) \le B/k$ for some constant $B$
and all
$\sigma_i^{T}\in\{\pm1\}^{T+1}$. Using the fact
that the innermost sum in (\ref{eq:RthTerm}) is bounded by $1$,
we get
%
%e94 ###
%e93 ###
%
\begin{eqnarray}\qquad
\Biggl|\sum_{r=r_0+1}^{k-1} \D(r,k)\Biggr|&\le&
\sum_{r=r_0+1}^{k-1} \pmatrix{k-1\cr r} \sum_{\sigma_1^{T}\cdots
\sigma_{r}^{T}}
\prod_{i=1}^r \bigl|\qprob(\sigma_i^{T}\midd\sigma_\rroot^T)-
\prob(\sigma_i^{T}\midd\sigma_\rroot^T)\bigr|\\
\label{eq:LargeR}
&\le&\sum_{r=r_0+1}^{k-1}\pmatrix{k-1\cr r} \biggl(\frac
{2^{T+1}B}{k}\biggr)^r\nonumber\\[-8pt]\\[-8pt]
&\le&\sum_{r\ge\log(k)}\frac{1}{r!} (2^{T+1}B)^r
= o(k^{-1/2}) ,\nonumber
\end{eqnarray}
where the last estimate follows from standard tail bounds on Poisson random
variables.

We are left with the sum of $\D(r,k)$ over $r\in\{0,\ldots,r_0\}$.
As in Lemma \ref{lemma:qminusp}, let
\[
\SSS_t\equiv
\{\sigma_{r+1}^{T}\cdots\sigma_{k-1}^{T}\dvtx|\sigma
_{r+1}(t)+\cdots+\sigma_{k-1}(t)+u(t)|\le r_0\} .
\]
If $\sigma_{r+1}^{T}\cdots\sigma_{k-1}^{T}$ is
not in $\bigcup_{t=0}^T \SSS_t$, then the sum
over $\sigma_{1}^{T}\cdots\sigma_{r}^{T}$ is $0$ due to the
normalization of
$\qprob( \cdot\midd\sigma_\rroot^T)$ and
$\prob( \cdot\midd\sigma_\rroot^T)$
(the same argument was already used in the proof of Lemma \ref{lemma:qminusp}).
Restricting the innermost sum and letting as before $\Sh_{t_0}\equiv
\SSS_{t_o}\cap\{\bigcap_{t\neq t_0}\overline{\SSS}_t\} $ with
$\SSS_t$
defined as in (\ref{eq:Sdef}), we then have
%
%e95 ###
%
\begin{eqnarray}\label{eq:DrkOneTie}
\D(r,k)
&=&
\pmatrix{k-1\cr r} \sum_{t_0=0}^T\sum_{\sigma_1^{T}\cdots\sigma
_{r}^{T}} \prod_{i=1}^r \{\qprob(\sigma_i^{T}\midd\sigma
_\rroot^T)-
\prob(\sigma_i^{T}\midd\sigma_\rroot^T)\}
\nonumber\\
&&\hphantom{\pmatrix{k-1\cr r} \sum_{t_0=0}^T}
{}\times\sum_{(\sigma_{r+1}^{T}\cdots\sigma_{k-1}^{T})\in\Sh_{t_0}}
\prod_{i=r+1}^{k-1}
\prob(\sigma_i^{T}\midd \sigma_\rroot^{T}) \\
&&\hspace*{132.6pt}{}\times\prod_{t=0}^T
\K_{u(t)}\bigl(\sigma_{\rroot}(t+1)|\sigma_{\droot}(t)\bigr)
 +
\ER(r,k) .\nonumber
\end{eqnarray}
By inclusion--exclusion, the error term is bounded as
\begin{eqnarray*}
|\ER(r,k)|&\le&\pmatrix{k-1\cr r}\sum_{t_1\neq t_2}\sum_{\sigma
_1^{T}\cdots\sigma_{r}^{T}} \prod_{i=1}^r \bigl|\qprob(\sigma
_i^{T}\midd\sigma_\rroot^T)-
\prob(\sigma_i^{T}\midd\sigma_\rroot^T)\bigr|
\\
&&\hphantom{\pmatrix{k-1\cr r}\sum_{t_1\neq t_2}}
{}
\times\sum_{(\sigma_{r+1}^{T}\cdots\sigma_{k-1}^{T})\in\SSS
_{t_1}\cap\SSS_{t_2}}\prod_{i=r+1}^{k-1}
\prob(\sigma_i^{T}\midd \sigma_\rroot^{T})\\
&&\hspace*{151.5pt}{}\times \prod_{t=0}^T
\K_{u(t)}\bigl(\sigma_{\rroot}(t+1)|\sigma_{\droot}(t)\bigr)\\
&\le&\pmatrix{k-1\cr r}\sum_{t_1\neq t_2}\sum_{\sigma_1^{T}\cdots
\sigma_{r}^{T}} \prod_{i=1}^r \bigl|\qprob(\sigma
_i^{T}\midd\sigma_\rroot^T)-
\prob(\sigma_i^{T}\midd\sigma_\rroot^T)\bigr| \frac
{Br_0^2}{k}\\
&\le&\pmatrix{k-1\cr r}T^2 2^{Tr}\biggl(\frac{B}{k}\biggr)^r
\frac{Br_0^2}{k} .
\end{eqnarray*}
The first inequality follows by applying Lemma
\ref{thm:sum_approx} to the $N=k-r-1\ge k-\log(k)-1$ i.i.d. random vectors
$(\sigma_{r+1})^{T},\ldots, (\sigma_{k-1})^{T}$, which are nondegenerate
for all $k$ large enough by Lemma \ref{lemma:RCconvergence},
and summing over the values of $a_{t_1}= \sum_{i=r+1}^{k-1}\sigma
_i(t_1)+u(t_1)$
and $a_{t_2}= \sum_{i=r+1}^{k-1}\sigma_i(t_2)+u(t_2)$, with
$|a_{t_1}|, |a_{t_2}|\le r_0$. The second inequality is instead implied by
Lemma \ref{lemma:qminusp}. It is now easy to sum over $r$ to get
\[
\Biggl|\sum_{r=1}^{r_0}\ER(r,k)\Biggr|\le
\sum_{r=0}^{\infty}\frac{1}{r!}T^2(2^TB)^rB \frac{(\log k)^2}{k} =
o(k^{-1/2})
.
\]
Therefore, the error terms $\ER(r,k)$ can be neglected.

Let us now consider the main term in (\ref{eq:DrkOneTie}),
and define
\begin{eqnarray*}
&&J'_{t_0}(\sigma_\rroot^{T},(\sigma_{1})_0^{T},\ldots,(\sigma_{r})_0^{T})\\
&&\qquad\equiv\sum_{(\sigma_{r+1}^{T}\cdots\sigma_{k-1}^{T})\in\Sh_{t_0}}
\prod_{i=r+1}^{k-1}
\prob(\sigma_i^{T}\midd \sigma_\rroot^{T}) \prod_{t=0}^T
\K_{u(t)}\bigl(\sigma_{\rroot}(t+1)|\sigma_{\droot}(t)\bigr) .
\end{eqnarray*}
We now proceed exactly as in the proof of Lemma \ref{lemma:qminusp},
cf. (\ref{eq:Jdef}) to (\ref{eq:Jestimate}) with
$\Omega(t)=\sigma_\rroot(t+1)(\sum_{i=1}^r \sigma_i(t))$
and $r_0= \log(k)$. Notice Theorem \ref{thm:sum_approx} continues
to hold and $r_0$ times
the $O(k^{-1/4})$ error is still $o(1)$. We arrive at
\[
J_{t_0} = \frac{1}{\sqrt{k}}
\sigma_{\rroot}(t_0+1)\Biggl(\sum_{i=1}^{r}\sigma_i(t_0)\Biggr)
J_{t_0}^* \bigl(1+\widetilde{\ER}_{t_0}(k)\bigr) ,
\]
where $\widetilde{\ER}_{t_0}(k) \rightarrow0$ as $k \rightarrow
\infty$
for any fixed $t_0$.

If we use this estimate in (\ref{eq:DrkOneTie}), we get
\begin{eqnarray*}
\D(r,k)
&=&\pmatrix{k'\cr r} \sum_{t_0=0}^T\sum_{\{\sigma_i^{T}\}} \prod
_{i=1}^r \{\qprob(\sigma_i^{T}\midd\sigma_\rroot^T)-
\prob(\sigma_i^{T}\midd\sigma_\rroot^T)\} \frac
{J^*_{t_0}}{\sqrt{k}}
\sigma_{\rroot}(t_0+1)\\
&&\hphantom{\pmatrix{k'\cr r} \sum_{t_0=0}^T\sum_{\{\sigma_i^{T}\}} \prod
_{i=1}^r}
{}\times\sum_{i=1}^{r} \sigma_i(t_0)
\bigl(1+ \widetilde{\ER}_{t_0}(k) \bigr) + o(k^{-1/2}) \\
&=& r \pmatrix{k'\cr r} \sum_{t_0=0}^T \sum_{\{\sigma_i^{T}\}}
\prod_{i=1}^r \bigl(\qprob(\sigma_i^{T}\midd\sigma_\rroot^T)-
\prob(\sigma_i^{T}\midd\sigma_\rroot^T)\bigr)
\frac{J^*_{t_0}}{\sqrt{k}} \sigma_\rroot(t_0+1)\\
&&\hphantom{\pmatrix{k'\cr r} \sum_{t_0=0}^T\sum_{\{\sigma_i^{T}\}} \prod
_{i=1}^rr}
{}\times\sigma_1(t_0)
\bigl(1+ \widetilde{\ER}_{t_0}(k)\bigr)+ o(k^{-1/2}) ,
\end{eqnarray*}
where $k'\equiv k-1$ and we
used the symmetry among the vertices $\{1,\ldots,r\}$
to replace $(\sum_{i=1}^{r} \sigma_i(t))$ by $r\sigma_1(t)$.
If $r\ge2$, the sums over $(\sigma_2)_0^T,\ldots,\sigma_r^T$
vanish except for the error terms $\widetilde{\ER}_{t_0}(k)$
[once more by the normalization
of $\prob( \cdot\midd\sigma_\rroot^T)$ and
$\qprob( \cdot\midd\sigma_\rroot^T)$]. We need to bound
contribution of
such error terms. Find $M$ such that
$ |\qprob(\sigma_i^{T}\midd(\sigma_\rroot)_0^T)- \prob
(\sigma_i^{T}\midd\sigma_\rroot^T)| \le M/k$.
We have
%
%e96 ###
%
\begin{eqnarray}\label{eq:qminusp_rgt2_err}
&& \biggl| r\pmatrix{k-1\cr r} \sum_{\sigma_1^{T}\cdots\sigma
_{r}^{T}} \{\qprob(\sigma_i^{T}\midd\sigma_\rroot^T)-
\prob(\sigma_i^{T}\midd\sigma_\rroot^T)\} \widetilde{\ER
}_{t_0}(k) \biggr|\nonumber\\
&&\qquad\le r\biggl(\frac{(k-1)e}{r} \biggr)^r 2^T \biggl(\frac{M}{k}
\biggr)^r
|\widetilde{\ER}_{t_0}(k) |\nonumber\\[-8pt]\\[-8pt]
&&\qquad\le r \biggl(\frac{2^TeM}{r} \biggr)^r |\widetilde{\ER}_{t_0}(k)|
\nonumber\\
&&\qquad\le
\biggl(\frac{M'}{2^r}\biggr)|\widetilde{\ER}_{t_0}(k)|\nonumber
\end{eqnarray}
for suitable $M'$. Here we have used the standard bound $ {n\choose m}
\le(\frac{ne}{m} )^m$.
Summing (\ref{eq:qminusp_rgt2_err}) over $t_0$ and $r$, we see that
${\sum_{r=2}^{r_0}}|\D(k,r)| \le C|J^*_{t_0}\widetilde{\ER
}_{t_0}(k)|/\sqrt{k}= o(k^{-1/2})$.

Further,
\begin{eqnarray*}
&&\sum_{\sigma_{1}^{T}}
\sigma_1(t)\{\qprob(\sigma_1^{T}\midd\sigma_\rroot^T)-
\prob(\sigma_1^{T}\midd\sigma_\rroot^T)\}\\
&&\qquad=\sum_{\sigma_{1}^{t}} \sigma_1(t)\{\qprob(\sigma
_1^{t}\midd\sigma_\rroot^t)- \prob(\sigma_1^{t}\midd\sigma
_\rroot^t)\}\\
&&\qquad= 2\frac{\bias_{t}}{k^{(T_*-t+1)/2}} \bigl( 1 + o(1)\bigr) ,
\end{eqnarray*}
where the second equality follows by Lemma \ref{lemma:qminusp}. Note
that for $t<T_*-1$, this sum is $o(k^{-1})$. As a consequence, only the
$t_0=T$ term
is relevant in the sum over~$t_0$.

Using these two remarks, we finally obtain
\begin{eqnarray*}
\sum_{r=1}^{r_0}\D(k,r) & = & \D(k,1) + o(k^{-1/2}) \\
& = & k \sum_{t_0=0}^T
\frac{J^*_{t_0}}{\sqrt{k}} 2\frac{\bias_{t_0}}{k^{(T_*-t_0+1)/2}}
\sigma_\rroot(t_0+1)
\bigl(1+o(1)\bigr) + o(k^{-1/2})\\
%= & \frac{J^*_T}{\sqrt{k}} \sigma_\rroot(T+1)2\bias_T \ +
%o(k^{-1/2})\\
& = & 2 \frac{\bias_{T_*-1}}{k^{1/2}} \sigma_\rroot(T_*)
I_{T_*-1}(\sigma_\rroot^{T_*-1}) \bigl(1+o(1)\bigr) ,
\end{eqnarray*}
which, together with (\ref{eq:LargeR}) and (\ref
{eq:SumOverR}), proves our thesis.
Equation (\ref{eq:sumqminusptstar}) follows as in the previous
lemma.
%

%s10 ###
\section{Proof of the local central limit theorem}
\label{app:CLT}

The proof repeats the arguments of \cite{McDonald2},
while keeping track explicitly of error terms. We will therefore
focus on the differences with respect to \cite{McDonald2}.
We will indeed prove a result that is
slightly stronger than Theorem \ref{thm:sum_approx}.
Apart from a trivial rescaling, the statement below differs from
Theorem \ref{thm:sum_approx} in that we allow for larger
deviations from the mean.
\begin{theorem}\label{thm:sum_approx2}
Let $X_1,\ldots, X_N$ be i.i.d.
vectors $X_i=(X_{i,1},X_{i,2},\ldots,X_{i,d})$ $\in\{0,1\}^d$ with
%
%e97 ###
%
\begin{equation}
\biggl|\prob\{X_{1,\ell}=1\}-\frac{1}{2}\biggr|\le\frac{B}{\sqrt
{N}}
\end{equation}
for $\ell\in\{1,\ldots, d\}$. Further
assume $\prob\{X_i=s\}\ge1/B$ for all $s\in\{0,1\}^d$.\vspace*{1pt}

Let $a\in\Z^d$ be such that $\sup_i|a_i-N/2|\le B\sqrt{N}$, and define,
for a partition $\{1,\ldots,d\}=\I_0\cup\I_+$,
\begin{eqnarray*}
A(a,\I) &\equiv& \{z \in\Z^d \dvtx z_i=a_i \ \forall i \in\I_0,
z_i\ge a_i \ \forall i \in\I_+\} ,\\
\A_{\infty}(a,\I) &\equiv& \bigl\{z \in\reals^d \dvtx z_i=a_i/\sqrt
{N} \
\forall i \in\I_0,
z_i\ge a_i/\sqrt{N} \ \forall i \in\I_+\bigr\} .
\end{eqnarray*}
Let $p_N$ be the distribution of $S_N=\sum_{i=1}^N X_i$. Then, there
exists a finite constant $\const=\const(B,d)$ such that for $K \equiv
|\I_0|$,
%
%e98 ###
%
\begin{eqnarray}
\biggl|F(a,\I) -\frac{1}{N^{K/2}} \Phi_{\sqrt{N}\Ex X_1,
\Cov(X_1)}(\A_{\infty}(a,\I))\biggr|&\le&\frac{\const
(B,d)}{N^{(K+(K+1)^{-1})/2}},\nonumber\\[-8pt]\\[-8pt]
F(a,\I) &\equiv&\sum_{y \in A(a,\I) } p_N(y).\nonumber
\end{eqnarray}
\end{theorem}

Since $\Phi_{\sqrt{N}\Ex X_1, \Cov(X_1)}(\A_{\infty}(a, \I))$
is bounded away from $0$ for $B$ bounded, the error estimate in the last
statement is equivalent to the one in Theorem
\ref{thm:sum_approx}.
For $K=0$ our claim is implied by the multi-dimensional Berry--Esseen theorem
\cite{BR}, and we will
therefore focus on $K\ge1$.

Recall that the Bernoulli decomposition of \cite{McDonald2} allows to
write, for $S_N=(S_{N,1},\ldots,S_{N,d})$ and $r\in\{1,\ldots,d\}$
%
%e99 ###
%
\begin{equation}
S_{N,r} = Z_{N,r} + \sum_{i=1}^{M_{N,r}} L_{i,r},
\end{equation}
where $Z_{N}$ is a lattice random variable,
$M_{N,r}\sim$Binom$(N,q_r)$ for $r=1,\ldots, d$, and $\{L_{i,r}\}$
is a collection of i.i.d. Bernoulli$(1/2)$ random variables
independent from $Z_N$ and $M_N$. Finally, it is easy to check that
$q_r\ge1/(Bd)$.

We have the following key estimate.
\begin{lemma}\label{lemma:LipCLT}
There exists a numerical constant $\widehat{\const}$ such that,
for any $a,b\in\Z^d$
%
%e100 ###
%
\begin{equation}
|F(a,\I)-F(b,\I)|\le\widehat{\const}\biggl(\frac
{Bd}{N}\biggr)^{(K+1)/2}
\|a-b\| ,
\end{equation}
where \mbox{$\|\cdot\|$} denotes the $L^1$ norm.
\end{lemma}
\begin{pf}
As in \cite{McDonald2}, we let, for $x,m\in\Z^d$,
%
%e101 ###
%
\begin{equation}
r_m(x) \equiv\prod_{i=1}^d\frac{1}{2^{m_i}}\pmatrix{m_i\cr x_i} ,
\end{equation}
be the probability mass function of the
vector $\Lambda_{m}
\equiv(\sum_{i=1}^{m_1}L_{i,1},\ldots, \sum_{i=1}^{m_d}L_{i,d})$.
It then follows immediately that
%
%e102 ###
%
\begin{equation}
\biggl|\sum_{x\in A(a,\I)}r_m(x) - \sum_{y\in A(b,\I)}r_m(y)
\biggr|\le\frac{\widetilde{\const}}{\min_i(m_i)^{(K+1)/2}}\|a-b\|
\end{equation}
for some numerical constant $\widetilde{\const}$. This is a slight
generalization
of Lemma 2.2 of \cite{McDonald2}, and follows again immediately from
the same estimates on the combinatorial coefficients used
in \cite{McDonald2}.

We then proceed analogously to the proof of Theorem 2.1 of \cite{McDonald2},
namely, for $h\in\Z^d$,
\begin{eqnarray*}
\hspace*{-4pt}&& \sup_{a\in\Z^d}|F(a+h,\I)-F(a,\I)|\\
\hspace*{-4pt}&&\qquad\le\sup_{a\in\Z^d}
\sum_{m\in\Z^d}\prob\{M_N=m\}\bigl|\prob\{S_N\in A(a,\I)| M_N=m\}\\
\hspace*{-4pt}&&\qquad\quad\hphantom{\sup_{a\in\Z^d}
\sum_{m\in\Z^d}\prob\{M_N=m\}\bigl|}\hspace*{-3pt}{}
-\prob\{S_N\in A(a+h,\I)| M_N=m\}\bigr|\\
\hspace*{-4pt}&&\qquad=\sup_{a\in\Z^d}
\sum_{m\in\Z^d}\prob\{M_N=m\}\bigl|\prob
\{Z_{N}+\Lambda_{m}\in A(a,\I)
| M_N=m\}\\
\hspace*{-4pt}&&\qquad\quad\hphantom{\sup_{a\in\Z^d}
\sum_{m\in\Z^d}\prob\{M_N=m\}\bigl|}\hspace*{-3pt}{}
-\prob\{Z_{N}+\Lambda_{m}\in A(a+h,\I)
| M_N=m\}\bigr|\\
\hspace*{-4pt}&&\qquad\le\sup_{a\in\Z^d}
\sum_{m\in\Z^d}\prob\{M_N=m\} \sum_{l\in\Z^d}\prob\{Z_N=l\}\\
&&\qquad\quad\hphantom{\sup_{a\in\Z^d}
\sum_{m\in\Z^d}\prob\{M_N=m\} \sum_{l\in\Z^d}}
{}\times
\bigl|\prob
\{\Lambda_{m}\in A(a-l,\I)
| M_N=m\}\\
\hspace*{-4pt}&&\qquad\quad\hphantom{\sup_{a\in\Z^d}
\sum_{m\in\Z^d}\prob\{M_N=m\}\bigl|}\hspace*{65pt}\hspace*{-35.7pt}{}
-\prob\{\Lambda_{m}\in A(a+h-l,\I)
| M_N=m\}\bigr|\\
\hspace*{-4pt}&&\qquad\le\sum_{m\in\Z^d} \frac{\widetilde{\const}}{\min
_i(m_i)^{(K+1)/2}}\|h\|,
\end{eqnarray*}
which is bounded as in the statement by the same argument used in
\cite{McDonald2}.
\end{pf}

We are now in a position to prove Theorem \ref{thm:sum_approx2}.
\begin{pf*}{Proof of Theorem \ref{thm:sum_approx2}}
For $a$ as in the statement and $\ell> 0$, let
\begin{eqnarray*}
R(a,\ell) &=& \{z\in\Z^d\dvtx|z_i-a_i|\le\ell\ \forall i\in
\I_0,
z_i = a_i\ \forall i\in\I_+\} ,\\
\R_{\infty}(a,\ell) &= &\bigl\{z\in\reals^d\dvtx\bigl|z_i-a_i/\sqrt
{N}\bigr|\le\ell/\sqrt{N}
\ \forall i\in\I_0,
z_i \geq a_i/\sqrt{N}\ \forall i\in\I_+\bigr\}.
\end{eqnarray*}
Then, by Lemma \ref{lemma:LipCLT}, there exists a constant $\const_1(B,d)$
such that
%
%e103 ###
%
\begin{equation}\label{eq:CLTFin1}
\biggl|F(a,\I)-\frac{1}{|R(a,\ell)|}\sum_{z\in R(a,\ell)}F(z,\I
)\biggr|
\le\frac{\const_1(B,d)\ell}{N^{(K+1)/2}} .
\end{equation}
On the other hand, by the Berry--Esseen theorem,
%
%e104 ###
%
\begin{equation}\label{eq:CLTFin2}\qquad
\biggl|\sum_{z\in R(a,\ell)}F(z,\I)-
\int_{\R_{\infty}(a,\ell) }\Phi_{\sqrt{N}\Ex X_1, \Cov(X_1)}
(\A_{\infty}(z,\I)) \,\de
z\biggr|
\le\frac{\const_2(d)}{N^{1/2}} .
\end{equation}
Finally,\vspace*{1pt} it is easy to see that
$\Phi_{\sqrt{N}\Ex X_1, \Cov(X_1)}(\A_{\infty}(z,\I
))$ is
Lipschitz continuous in $z$ with Lipschitz constant bounded uniformly
in $N$,
whence
%
%e105 ###
%
\begin{eqnarray}\label{eq:CLTFin3}
&&\biggl|\Phi_{\sqrt{N}\Ex X_1, \Cov(X_1)}
(\A_{\infty}(a,\I))\nonumber\\
&&\quad{}-\frac{1}{|\R_{\infty}(a,\ell)|}
\int_{\R_{\infty}(a,\ell) }\Phi_{\sqrt{N}\Ex X_1, \Cov(X_1)}
(\A_{\infty}(z,\I))\,\de
z\biggr|\\
&&\qquad\le\frac{\const_3\ell}{\sqrt{N}} .\nonumber
\end{eqnarray}
The\vspace*{1pt} proof is completed by putting together (\ref{eq:CLTFin1}),
(\ref{eq:CLTFin2}) and (\ref{eq:CLTFin3}), using
$|R(a,\ell)|=\Theta(\ell^{K})$, $|R_{\infty}(a,\ell)|=\Theta(\ell
^{K}N^{-K/2})$,
and setting $\ell=N^{K/(2K+2)}$.
\end{pf*}
\end{appendix}

% imsref loaded by lrinkeviciute, 2011-03-03 09:46:10
%

%
\printaddresses


\begin{thebibliography}{36}
% BibTex style file: ims.bst, 2010-01-14
% Default style options (sort=0,type=number).
% Used options (sort=1,type=number).

%b1 ###
\bibitem{AldousSteele}
%
\begin{bincollection}[mr]
\bauthor{\bsnm{Aldous},~\bfnm{David}\binits{D.}} \AND
\bauthor{\bsnm{Steele},~\bfnm{J.~Michael}\binits{J.~M.}}
(\byear{2004}).
\btitle{The objective method: Probabilistic combinatorial optimization and
local weak convergence}.
In \bbooktitle{Probability on Discrete Structures}.
\bseries{Encyclopaedia of Mathematical Sciences}
\bvolume{110}
\bpages{1--72}.
\bpublisher{Springer}, \baddress{Berlin}.
\bid{mr={2023650}}
\bptnote{check year}
\end{bincollection}
%
\endbibitem

%b2 ###
\bibitem{BalaGoyal98}
%
\begin{barticle}[auto:STB|2010-11-18|09:18:59]
\bauthor{\bsnm{Bala},~\bfnm{V.}\binits{V.}} \AND
\bauthor{\bsnm{Goyal},~\bfnm{S.}\binits{S.}}
(\byear{1998}).
\btitle{Learning from neighbours}.
\bjournal{Rev. Econom. Stud.}
\bvolume{65}
\bpages{595--621}.
\end{barticle}
%
\endbibitem

%b3 ###
\bibitem{PeresCore}
%
\begin{barticle}[mr]
\bauthor{\bsnm{Balogh},~\bfnm{J{\'o}zsef}\binits{J.}},
\bauthor{\bsnm{Peres},~\bfnm{Yuval}\binits{Y.}} \AND
\bauthor{\bsnm{Pete},~\bfnm{G{\'a}bor}\binits{G.}}
(\byear{2006}).
\btitle{Bootstrap percolation on infinite trees and non-amenable groups}.
\bjournal{Combin. Probab. Comput.}
\bvolume{15}
\bpages{715--730}.
\bid{doi={10.1017/S0963548306007619}, mr={2248323}}
\end{barticle}
%
\endbibitem

%b4 ###
\bibitem{DemboEtAl}
%
\begin{barticle}[mr]
\bauthor{\bsnm{Ben~Arous},~\bfnm{G{\'e}rard}\binits{G.}},
\bauthor{\bsnm{Dembo},~\bfnm{Amir}\binits{A.}} \AND
\bauthor{\bsnm{Guionnet},~\bfnm{Alice}\binits{A.}}
(\byear{2006}).
\btitle{Cugliandolo--{K}urchan equations for dynamics of spin-glasses}.
\bjournal{Probab. Theory Related Fields}
\bvolume{136}
\bpages{619--660}.
\bid{doi={10.1007/s00440-005-0491-y}, mr={2257139}}
\end{barticle}
%
\endbibitem

%b5 ###
\bibitem{BenjSchr}
%
\begin{barticle}[mr]
\bauthor{\bsnm{Benjamini},~\bfnm{Itai}\binits{I.}} \AND
\bauthor{\bsnm{Schramm},~\bfnm{Oded}\binits{O.}}
(\byear{1996}).
\btitle{Percolation beyond {$\bold Z\sp d$}, many questions and a few answers}.
\bjournal{Electron. Comm. Probab.}
\bvolume{1}
\bpages{71--82 (electronic)}.
\bid{mr={1423907}}
\end{barticle}
%
\endbibitem

%b6 ###
\bibitem{BergerEtAl}
%
\begin{barticle}[mr]
\bauthor{\bsnm{Berger},~\bfnm{Noam}\binits{N.}},
\bauthor{\bsnm{Kenyon},~\bfnm{Claire}\binits{C.}},
\bauthor{\bsnm{Mossel},~\bfnm{Elchanan}\binits{E.}} \AND
\bauthor{\bsnm{Peres},~\bfnm{Yuval}\binits{Y.}}
(\byear{2005}).
\btitle{Glauber dynamics on trees and hyperbolic graphs}.
\bjournal{Probab. Theory Related Fields}
\bvolume{131}
\bpages{311--340}.
\bid{doi={10.1007/s00440-004-0369-4}, mr={2123248}}
\end{barticle}
%
\endbibitem

%b7 ###
\bibitem{BR}
%
\begin{bbook}[mr]
\bauthor{\bsnm{Bhattacharya},~\bfnm{R.~N.}\binits{R.~N.}} \AND
\bauthor{\bsnm{Ranga~Rao},~\bfnm{R.}\binits{R.}}
(\byear{1976}).
\btitle{Normal Approximation and Asymptotic Expansions}.
\bpublisher{Wiley}, \baddress{New York}.
\bid{mr={0436272}}
\end{bbook}
%
\endbibitem

%b8 ###
\bibitem{BollobasConf}
%
\begin{barticle}[mr]
\bauthor{\bsnm{Bollob{\'a}s},~\bfnm{B{\'e}la}\binits{B.}}
(\byear{1980}).
\btitle{A probabilistic proof of an asymptotic formula for the number of
labelled regular graphs}.
\bjournal{European J. Combin.}
\bvolume{1}
\bpages{311--316}.
\bid{mr={0595929}}
\end{barticle}
%
\endbibitem

%b9 ###
\bibitem{BouchaudEtAl}
%
\begin{bincollection}[auto:STB|2010-11-18|09:18:59]
\bauthor{\bsnm{Bouchaud},~\bfnm{J.~P.}\binits{J.~P.}},
\bauthor{\bsnm{Cugliandolo},~\bfnm{L.~F.}\binits{L.~F.}},
\bauthor{\bsnm{Kurchan},~\bfnm{J.}\binits{J.}} \AND
\bauthor{\bsnm{Mezard},~\bfnm{M.}\binits{M.}}
(\byear{1997}).
\btitle{Out of equilibrium dynamics in spin-glasses and other glassy
systems}.
In \bbooktitle{Spin Glass Dynamics and Random Fields}
(\beditor{A. P. Young}, ed.).
\bpublisher{World Scientific}, \baddress{Singapore}.
\end{bincollection}
%
\endbibitem

%b10 ###
\bibitem{CaputoMartinelli}
%
\begin{barticle}[mr]
\bauthor{\bsnm{Caputo},~\bfnm{Pietro}\binits{P.}} \AND
\bauthor{\bsnm{Martinelli},~\bfnm{Fabio}\binits{F.}}
(\byear{2006}).
\btitle{Phase ordering after a deep quench: The stochastic {I}sing and hard
core gas models on a tree}.
\bjournal{Probab. Theory Related Fields}
\bvolume{136}
\bpages{37--80}.
\bid{doi={10.1007/s00440-005-0475-y}, mr={2240782}}
\end{barticle}
%
\endbibitem

%b11 ###
\bibitem{Chatterjee}
%
\begin{barticle}[auto:STB|2010-11-18|09:18:59]
\bauthor{\bsnm{Chatterjee},~\bfnm{S.}\binits{S.}}
(\byear{2009}).
\btitle{Spin glasses and Stein's method}.
\bjournal{Probab. Theory Related Fields}
\bvolume{148}
\bpages{567--600}.
\end{barticle}
%
\endbibitem

%b12 ###
\bibitem{McDonald2}
%
\begin{barticle}[mr]
\bauthor{\bsnm{Davis},~\bfnm{Burgess}\binits{B.}} \AND
\bauthor{\bsnm{McDonald},~\bfnm{David}\binits{D.}}
(\byear{1995}).
\btitle{An elementary proof of the local central limit theorem}.
\bjournal{J. Theoret. Probab.}
\bvolume{8}
\bpages{693--701}.
\bid{doi={10.1007/BF02218051}, mr={1340834}}
\end{barticle}
%
\endbibitem

%b13 ###
\bibitem{DeGroot74}
%
\begin{barticle}[auto:STB|2010-11-18|09:18:59]
\bauthor{\bsnm{DeGroot},~\bfnm{M.~H.}\binits{M.~H.}}
(\byear{1974}).
\btitle{Reaching a consensus}.
\bjournal{J. Amer. Statist. Assoc.}
\bvolume{69}
\bpages{118--121}.
\end{barticle}
%
\endbibitem

%b14 ###
\bibitem{DemboBrazil}
%
\begin{barticle}[mr]
\bauthor{\bsnm{Dembo},~\bfnm{Amir}\binits{A.}} \AND
\bauthor{\bsnm{Montanari},~\bfnm{Andrea}\binits{A.}}
(\byear{2010}).
\btitle{Gibbs measures and phase transitions on sparse random graphs}.
\bjournal{Braz. J. Probab. Stat.}
\bvolume{24}
\bpages{137--211}.
\bid{doi={10.1214/09-BJPS027}, mr={2643563}}
\end{barticle}
%
\endbibitem

%b15 ###
\bibitem{Dobrushin}
%
\begin{barticle}[mr]
\bauthor{\bsnm{Dobru{\v{s}}in},~\bfnm{R.~L.}\binits{R.~L.}}
(\byear{1968}).
\btitle{The problem of uniqueness of a {G}ibbsian random field and the problem
of phase transitions}.
\bjournal{Funkcional. Anal. i Prilo\v zen.}
\bvolume{2}
\bpages{44--57}.
\bid{mr={0250631}}
\end{barticle}
%
\endbibitem

%b16 ###
\bibitem{Fontes}
%
\begin{barticle}[mr]
\bauthor{\bsnm{Fontes},~\bfnm{L.~R.}\binits{L.~R.}},
\bauthor{\bsnm{Schonmann},~\bfnm{R.~H.}\binits{R.~H.}} \AND
\bauthor{\bsnm{Sidoravicius},~\bfnm{V.}\binits{V.}}
(\byear{2002}).
\btitle{Stretched exponential fixation in stochastic {I}sing models at zero
temperature}.
\bjournal{Comm. Math. Phys.}
\bvolume{228}
\bpages{495--518}.
\bid{doi={10.1007/s002200200658}, mr={1918786}}
\end{barticle}
%
\endbibitem

%b17 ###
\bibitem{FontesTree}
%
\begin{barticle}[mr]
\bauthor{\bsnm{Fontes},~\bfnm{L.~R.~G.}\binits{L.~R.~G.}} \AND
\bauthor{\bsnm{Schonmann},~\bfnm{R.~H.}\binits{R.~H.}}
(\byear{2008}).
\btitle{Bootstrap percolation on homogeneous trees has 2 phase transitions}.
\bjournal{J. Stat. Phys.}
\bvolume{132}
\bpages{839--861}.
\bid{doi={10.1007/s10955-008-9583-2}, mr={2430783}}
\end{barticle}
%
\endbibitem

%b18 ###
\bibitem{Georgii}
%
\begin{bbook}[mr]
\bauthor{\bsnm{Georgii},~\bfnm{Hans-Otto}\binits{H.-O.}}
(\byear{1988}).
\btitle{Gibbs Measures and Phase Transitions}.
\bseries{de Gruyter Studies in Mathematics}
\bvolume{9}.
\bpublisher{de Gruyter}, \baddress{Berlin}.
\bid{mr={0956646}}
\end{bbook}
%
\endbibitem

%b19 ###
\bibitem{Olivos}
%
\begin{barticle}[mr]
\bauthor{\bsnm{Goles},~\bfnm{E.}\binits{E.}} \AND
\bauthor{\bsnm{Olivos},~\bfnm{J.}\binits{J.}}
(\byear{1980}).
\btitle{Periodic behaviour of generalized threshold functions}.
\bjournal{Discrete Math.}
\bvolume{30}
\bpages{187--189}.
\bid{doi={10.1016/0012-365X(80)90121-1}, mr={0566436}}
\end{barticle}
%
\endbibitem

%%b20 ###
%%
%(\byear{2004}).
%model}.
%%

%b21 ###
\bibitem{Hatchett04}
%
\begin{barticle}[mr]
\bauthor{\bsnm{Hatchett},~\bfnm{J.~P.~L.}\binits{J.~P.~L.}},
\bauthor{\bsnm{Wemmenhove},~\bfnm{B.}\binits{B.}},
\bauthor{\bsnm{P{\'e}rez~Castillo},~\bfnm{I.}\binits{I.}},
\bauthor{\bsnm{Nikoletopoulos},~\bfnm{T.}\binits{T.}},
\bauthor{\bsnm{Skantzos},~\bfnm{N.~S.}\binits{N.~S.}} \AND
\bauthor{\bsnm{Coolen},~\bfnm{A.~C.~C.}\binits{A.~C.~C.}}
(\byear{2004}).
\btitle{Parallel dynamics of disordered {I}sing spin systems on finitely
connected random graphs}.
\bjournal{J. Phys. A}
\bvolume{37}
\bpages{6201--6220}.
\bid{doi={10.1088/0305-4470/37/24/001}, mr={2073601}}
\end{barticle}
%
\endbibitem

%b22 ###
\bibitem{Expander}
%
\begin{barticle}[mr]
\bauthor{\bsnm{Hoory},~\bfnm{Shlomo}\binits{S.}},
\bauthor{\bsnm{Linial},~\bfnm{Nathan}\binits{N.}} \AND
\bauthor{\bsnm{Wigderson},~\bfnm{Avi}\binits{A.}}
(\byear{2006}).
\btitle{Expander graphs and their applications}.
\bjournal{Bull. Amer. Math. Soc. (N.S.)}
\bvolume{43}
\bpages{439--561 (electronic)}.
\bid{doi={10.1090/S0273-0979-06-01126-8}, mr={2247919}}
\end{barticle}
%
\endbibitem

%b23 ###
\bibitem{Howard}
%
\begin{barticle}[mr]
\bauthor{\bsnm{Howard},~\bfnm{C.~Douglas}\binits{C.~D.}}
(\byear{2000}).
\btitle{Zero-temperature {I}sing spin dynamics on the homogeneous tree of
degree three}.
\bjournal{J. Appl. Probab.}
\bvolume{37}
\bpages{736--747}.
\bid{mr={1782449}}
\end{barticle}
%
\endbibitem

%b24 ###
\bibitem{Kleinberg}
%
\begin{bincollection}[mr]
\bauthor{\bsnm{Kleinberg},~\bfnm{Jon}\binits{J.}}
(\byear{2007}).
\btitle{Cascading behavior in networks: Algorithmic and economic issues}.
In \bbooktitle{Algorithmic Game Theory}
\bpages{613--632}.
\bpublisher{Cambridge Univ. Press}, \baddress{Cambridge}.
\bid{mr={2391771}}
\end{bincollection}
%
\endbibitem

%b25 ###
\bibitem{Quantum1}
%
\begin{barticle}[auto:STB|2010-11-18|09:18:59]
\bauthor{\bsnm{Krzakala},~\bfnm{F.}\binits{F.}},
\bauthor{\bsnm{Rosso},~\bfnm{A.}\binits{A.}},
\bauthor{\bsnm{Semerjian},~\bfnm{G.}\binits{G.}} \AND
\bauthor{\bsnm{Zamponi},~\bfnm{F.}\binits{F.}}
(\byear{2008}).
\btitle{On the path integral representation for quantum spin models
and its
application to the quantum cavity method and to Monte Carlo simulations}.
\bjournal{Phys. Rev. B}
\bvolume{78}
\bpages{134428}.
\end{barticle}
%
\endbibitem

%b26 ###
\bibitem{Quantum2}
%
\begin{barticle}[auto:STB|2010-11-18|09:18:59]
\bauthor{\bsnm{Laumann},~\bfnm{C.}\binits{C.}},
\bauthor{\bsnm{Scardicchio},~\bfnm{A.}\binits{A.}} \AND
\bauthor{\bsnm{Sondhi},~\bfnm{S.}\binits{S.}}
(\byear{2008}).
\btitle{Cavity method for quantum spin glasses on the Bethe lattice}.
\bjournal{Phys. Rev. B}
\bvolume{78}
\bpages{134424}.
\end{barticle}
%
\endbibitem

%b27 ###
\bibitem{Liggett}
%
\begin{bbook}[mr]
\bauthor{\bsnm{Liggett},~\bfnm{Thomas~M.}\binits{T.~M.}}
(\byear{1985}).
\btitle{Interacting Particle Systems}.
\bseries{Grundlehren der Mathematischen Wissenschaften [Fundamental Principles
of Mathematical Sciences]}
\bvolume{276}.
\bpublisher{Springer}, \baddress{New York}.
\bid{mr={0776231}}
\end{bbook}
%
\endbibitem

%b28 ###
\bibitem{Martinelli}
%
\begin{bincollection}[auto:STB|2010-11-18|09:18:59]
\bauthor{\bsnm{Martinelli},~\bfnm{F.}\binits{F.}},
\bauthor{\bsnm{Sinclair},~\bfnm{A.}\binits{A.}} \AND
\bauthor{\bsnm{Weitz},~\bfnm{D.}\binits{D.}}
(\byear{2003}).
\btitle{The Ising model on trees: Boundary conditions and mixing
time}.
In \bbooktitle{Proc. of IEEE Symposium on Found. of Computer
Science}.
\bpublisher{IEEE Computer Soc.}, \baddress{Los Alamitos, CA}.
\end{bincollection}
%
\endbibitem

%b29 ###
\bibitem{McDonald1}
%
\begin{barticle}[mr]
\bauthor{\bsnm{McDonald},~\bfnm{David~R.}\binits{D.~R.}}
(\byear{1979}).
\btitle{On local limit theorem for integer valued random variables}.
\bjournal{Teor. Veroyatnost. i Primenen.}
\bvolume{24}
\bpages{607--614}.
\bid{mr={0541375}}
\end{barticle}
%
\endbibitem

%b30 ###
\bibitem{IPC}
%
\begin{bbook}[mr]
\bauthor{\bsnm{M{\'e}zard},~\bfnm{Marc}\binits{M.}} \AND
\bauthor{\bsnm{Montanari},~\bfnm{Andrea}\binits{A.}}
(\byear{2009}).
\btitle{Information, Physics, and Computation}.
\bpublisher{Oxford Univ. Press}, \baddress{Oxford}.
\bid{doi={10.1093/acprof:oso/9780198570837.001.0001}, mr={2518205}}
\end{bbook}
%
\endbibitem

%b31 ###
\bibitem{SpinGlass}
%
\begin{bbook}[mr]
\bauthor{\bsnm{M{\'e}zard},~\bfnm{Marc}\binits{M.}},
\bauthor{\bsnm{Parisi},~\bfnm{Giorgio}\binits{G.}} \AND
\bauthor{\bsnm{Virasoro},~\bfnm{Miguel~Angel}\binits{M.~A.}}
(\byear{1987}).
\btitle{Spin Glass Theory and Beyond}.
\bseries{World Scientific Lecture Notes in Physics}
\bvolume{9}.
\bpublisher{World Scientific}, \baddress{Teaneck, NJ}.
\bid{mr={1026102}}
\end{bbook}
%
\endbibitem

%b32 ###
\bibitem{Morris}
%
\begin{barticle}[mr]
\bauthor{\bsnm{Morris},~\bfnm{Stephen}\binits{S.}}
(\byear{2000}).
\btitle{Contagion}.
\bjournal{Rev. Econom. Stud.}
\bvolume{67}
\bpages{57--78}.
\bid{doi={10.1111/1467-937X.00121}, mr={1745852}}
\end{barticle}
%
\endbibitem

%b33 ###
\bibitem{Neri2009}
%
\begin{barticle}[auto:STB|2010-11-18|09:18:59]
\bauthor{\bsnm{Neri},~\bfnm{I.}\binits{I.}} \AND
\bauthor{\bsnm{Boll\'e},~\bfnm{D.}\binits{D.}}
(\byear{2009}).
\btitle{The cavity approach to parallel dynamics of Ising spins on a graph}.
\bjournal{J. Stat. Mech.}
\bnote{P08009}.
\end{barticle}
%
\endbibitem

%b34 ###
\bibitem{Talabook}
%
\begin{bbook}[mr]
\bauthor{\bsnm{Talagrand},~\bfnm{Michel}\binits{M.}}
(\byear{2003}).
\btitle{Spin Glasses: A Challenge for Mathematicians: Cavity and Mean Field Models}.
\bseries{Ergebnisse der Mathematik und Ihrer Grenzgebiete. 3. Folge. A Series
of Modern Surveys in Mathematics [Results in Mathematics and Related Areas.
3rd Series. A Series of Modern Surveys in Mathematics]}
\bvolume{46}.
\bpublisher{Springer}, \baddress{Berlin}.
\bid{mr={1993891}}
\end{bbook}
%
\endbibitem

%b35 ###
\bibitem{Weitz}
%
\begin{bincollection}[mr]
\bauthor{\bsnm{Weitz},~\bfnm{Dror}\binits{D.}}
(\byear{2006}).
\btitle{Counting independent sets up to the tree threshold}.
In \bbooktitle{S{TOC}'06: {P}roceedings of the 38th {A}nnual {ACM} {S}ymposium
on {T}heory of {C}omputing}
\bpages{140--149}.
\bpublisher{ACM}, \baddress{New York}.
\bid{doi={10.1145/1132516.1132538}, mr={2277139}}
\end{bincollection}
%
\endbibitem

\end{thebibliography}
\end{document}